\documentclass[12pt,reqno]{amsart}
 \usepackage{amsmath,amsfonts,amssymb,amsthm,mathrsfs}
 \usepackage{hyperref}
 \usepackage{bm}
 \usepackage{color}
\usepackage{graphicx}
\usepackage{graphics}
\usepackage{geometry}
\usepackage{tikz}
\usepackage[all]{xy}

\newtheorem{theorem}{Theorem}[section]
\newtheorem{lemma}[theorem]{Lemma}
\newtheorem{proposition}[theorem]{Proposition}
\newtheorem{corollary}[theorem]{Corollary}

\newtheorem{key lemma}[theorem]{\bf Key Lemma}
\newtheorem{sub lemma}[theorem]{\bf Sublemma}
\theoremstyle{definition}
\newtheorem{definition}[theorem]{Definition}
\newtheorem{example}[theorem]{Example}

\newtheorem{conjecture}[theorem]{Conjecture}
\newtheorem{submetry conjecture}[theorem]{Submetry Conjecture}
\newtheorem{soul conjecture}[theorem]{Soul Conjecture}
\newtheorem{canonical bundle conjecture}[theorem]{Canonical bundle Conjecture}
\newtheorem{open problem}[theorem]{Open problem}

\theoremstyle{remark}
\newtheorem{remark}[theorem]{Remark}

\numberwithin{equation}{theorem}%
\numberwithin{figure}{theorem}

\begin{document}
\title{Open Alexandrov spaces of nonnegative curvature}

\author{Xueping Li$^1$ and Xiaochun Rong$^2$}
\date{\today}

\address{X. Li: Mathematics Department, Jiangsu Normal University, Xuzhou, Jiangsu Province, P. R. China}

\address{X. Rong: Mathematics Department, Capital Normal University, Beijing,
P.R.China}

\address{Mathematics Department, Rutgers University
New Brunswick, NJ 08903 USA}


\noindent 
\thanks{1.\it {Partially supported by NSFC 11501258, research founds from Jiangsu Normal University and Capital Normal University. Part of this work is done during the author's visit at Capital Normal University}}.
\thanks{2.\it{Partially supported by NSFC 11821101, BNSF Z190003, and a research fund from Capital Normal University. \hfill{$\,$}}}

\maketitle

\begin{abstract} Let $X$ be an open (i.e. complete, non-compact and without boundary) Alexandrov $n$-space of nonnegative curvature with a soul $S$.  In this
paper, we will establish several structural results on $X$ that can be viewed as counterparts of structural results
on an open Riemannian manifold with nonnegative sectional curvature.
\end{abstract}

\section {Introduction}

\vskip4mm

In this paper, we will investigate interplays of geometric and topological structures on an open (complete non-compact) Alexandrov space of nonnegative curvature, all Alexandrov spaces in this paper are of finite dimensional.

An Alexandrov space with curvature $\ge \kappa$ is a complete length metric space on which
the Toponogov triangle comparison holds with respect to a simply connected surface of constant curvature $\kappa$. Alexandrov geometry is a synthetic geometry introduced by Burogo-Gromov-Perel'man in \cite{BGP}. A partial motivation is that the Gromov-Hausdorff limit of a sequence of
Riemannian $n$-manifolds $M_i$ of sectional curvature, $\text{sec}_{M_i}\ge \kappa$, $M_i\xrightarrow{GH}X$, may not be a Riemannian manifold, but $X$ is always an Alexandrov
space with curvature $\ge \kappa$.

A core issue in Alexandrov geometry is interplays between geometric and topological structures on an Alexandrov space, most of which, if not all, are counterparts to results in Riemannian geometry that rely on the Toponogov triangle comparison.

Our main results in this paper can be viewed as (partial) `counterpart' to classical results in Riemannian geometry on open manifolds of nonnegative sectional curvature, which
we briefly review below.

\begin{theorem} {\rm (Soul, Cheeger-Gromoll \cite{CG})} \label{sol-cg} Let $M$ be an open Riemannian
manifold of $\text{sec}\ge 0$. Then $M$ contains a compact totally convex submanifold
$S$ (called a soul of $M$), and there is a diffeomorphism, $f: T^\perp S\to M$,
where $T^\perp S$ denotes the normal bundle of $S$. Moreover, if $\text{sec}_M>0$, then
$S=\{\text{pt}\}$, thus $M$ is diffeomorphic to an Euclidean space (\cite{GM}).
\end{theorem}

Note that the exponential map, $\exp: T^\perp S\to M$,
may not be injective; so the diffeomorphism, $f: T^\perp S\to M$, may
not be canonically determined.

There is a canonically defined distance non-increasing map, $\phi: M\to S$, called
the Sharafutdinov projection (\cite{Sh}, \cite{Per1}); which is obtained through a deformation retracting process:
for $p\in M$, the Busemann function $b_p$ is a proper concave function, thus $b_p$ has a well-defined gradient field with a compact convex maximum point set. Moving $x\in M$ along the gradient curve of $b_p$ to its maximum point set $S_{\max b_p}$.
If $\partial S_{\max b_p}\ne \emptyset$, then one continuously moves the point along the gradient curve of $d_{\partial S_{\max b_p}}$. Repeating this process until one gets a maximum point set, $S$, without boundary. We will call $b_p$ a Busemann function of $S$.

Since Theorem \ref{sol-cg}, the most significant advance is the Perel'man discovery of the rigidity of
Sharafutinov projection, $\phi: M\to S$.

\vskip2mm

\begin{theorem} {\rm (Perel'man, \cite{Per3})} \label{sol} Let $M$ be an open manifold of $\text{sec}_M\ge 0$, and let $S$ be a soul. Then any distance non-increasing
retraction from $M$ to $S$ coincides with a Sharafutinov projection, $\phi: M\to S$,
which satisfies the following properties:

\vskip1mm

\noindent {\rm (0.2.1)} For any $\bar x\in S$ and $\bar v\in T_{\bar x}^\perp S$, $\phi(\exp_{\bar x}t\bar v)=\bar x$ for all $t>0$.

\vskip1mm

\noindent {\rm (0.2.2) (Flat strip)} For $\bar p\in S, \bar v\in T^\perp_{\bar p}S$,
if $\bar p\ne \bar q\in S$, then $\exp_{\bar \gamma(s)} t\bar v(s)$ form
a flat strip i.e., an embedding, $[0,d(\bar p,\bar q)]\times
\mathbb R_+\to M$, and isometric embedding on $[0,d(\bar p,\bar q)]\times [t-\delta_t,t+\delta_t]$
($\delta_t>0$), where $\bar v(s)$ denotes the parallel transport of $\bar v$
along a normal minimal geodesic $\bar \gamma(s)$ from $\bar p$ to $\bar q$.

\vskip1mm

\noindent {\rm (0.2.3) (Submetry)} $\phi: M\to S$, is a submetry i.e.,
$\phi(B_r(x))=B_r(\phi(x))$ for all $x\in M$ and $r>0$. Consequently, $\phi$ is a $C^1$ Riemannian submersion (thus a fiber bundle map).

\vskip1mm

\noindent {\rm (0.2.4) (Soul conjecture of Cheeger-Gromoll)} If there is $q\in M$ where any sectional curvature is positive, then $S$ is a point.
\end{theorem}

\vskip2mm

In \cite{CS} and \cite{Wi}, it was independently proved that $\phi$ is $C^\infty$,
thus $\phi: M\to S$ is a Riemann submersion  (\cite{Guj}).
Because $\phi$ is solely determined by the metric structure on $M$, we will call $(M, S, \phi)$ the canonical fiber bundle. In \cite{Wi}, it was showed that there is a diffeomorphism,
$f: T^\perp S\to M$, such that $\phi\circ f=\text{proj}: T^\perp S\to S$, i.e., two bundles
$(T^\perp S,S,\text{proj})$ and $(M,S,\phi)$ are equivalent. In the Appendix, we will present
a proof different from \cite{Wi}.

In various situations, the canonical fiber bundle poses strong rigidities.

\vskip2mm

\begin{theorem} {\rm(Rigidities of canonical fiber bundles)} \label{rig} Let $(M,S,\phi)$
be as in Theorem \ref{sol}.

\vskip1mm

\noindent {\rm (0.3.1) (Canonical bundle equivalence)} If there is $\bar p\in S$ such that $\exp_{\bar p}: T_{\bar p}^\perp S\to M$ is injective, then $\exp: T^\perp S\to M$ is a diffeomorphism, thus $\exp$ is a bundle isomorphism of $(T^\perp S,S,\text{proj})$ and $(M,S,\phi)$.

\vskip1mm

\noindent {\rm (0.3.2) (Integrable horizontal distribution rigidity, \cite{Str}, \cite{Yim})} If the horizontal distribution of $(M,S,\phi)$ is integrable, then the Riemannian
universal cover, $\pi: (\tilde M,\tilde p)\to (M,\bar p)$, $\bar p\in S$, splits, $\tilde M=\pi^{-1}(S)\times \phi^{-1}(\bar p)$, and $\pi^{-1}(S)$ splits, $\pi^{-1}(S)=\hat S\times \mathbb R^k$, $k\ge 0$, and $\hat S$ is compact.

\vskip1mm

\noindent {\rm (0.3.3) (Codimension $2$ rigidity,  \cite{Wa}, \cite{Str}, \cite{Yim})} If $\dim(M)-\dim(S)=2$, then $M$ satisfies either (1.3.1) or (1.3.2).
\end{theorem}

We observe that (0.3.1) follows from (0.2.2). In (0.3.3), the holonomy group of
$(T^\perp S,S,\text{proj})$ is $S^1$ or trivial (correspondingly, $\exp$ is a
diffeomorphism, or $\tilde M$ splits). Note that if $S$ has codimension $>2$,
then geometric structures on $X$ are more complicated (cf. Theorem 2.9 in \cite{GW}).

Let's now review counterparts and conjectured counterparts of Theorems 0.1-0.3 in Alexandrov geometry. Let $\text{Alex}^n(\kappa)$ denote the set of complete $n$-dimensional Alexandrov spaces
of curvature $\ge \kappa$.

First, a counterpart of Theorem \ref{sol-cg} is:

\vskip2mm

\begin{theorem} {\rm (Soul, \cite{Per1})} \label{sol-alex} For an open $X\in \text{Alex}^n(0)$, $X$ contains a compact totally convex subset $S$ without boundary and there is a distance non-increasing map
(the Sharafutdinov projection), $\phi: X\to S$.
\end{theorem}

\vskip2mm

The following are conjectured `counterparts' in Alexandrov geometry to Theorem \ref{sol}

\begin{submetry conjecture} [\cite{Per3}] \label{0e} Let $(X,S,\phi)$ be as
in Theorem \ref{sol-alex}. Then the Sharafutinov projection $\phi: X\to S$ is a submetry.
\end{submetry conjecture}

\begin{soul conjecture} [\cite{Per3}] \label{0f} Let $(X,S,\phi)$ be as
in Theorem \ref{sol-alex}. If $X$ contains an open subset where the curvature is bounded below by
a positive constant, then $S$ is a point.
\end{soul conjecture}

As in the Riemannian case, Submetry Conjecture \ref{0e} implies Soul Conjecture \ref{0f}, and different
from the Riemannian case, a simple example (Example \ref{1t}) shows that $\phi: X\to S$ is a submetry but not a fiber bundle map. Hence for $(X,S,\phi)$ to be a fiber bundle, additional restrictions are required.

Recall that a point $\bar x\in Y\in \text{Alex}^m(-1)$ is weakly $k$-strained (\cite{Fuj3}), if there are $(k+1$)-points, $\bar p_1,...,\bar p_k, \bar w$, such that
$$\measuredangle \bar p_i\bar x\bar p_j>\frac \pi2,\quad \measuredangle \bar p_i\bar x\bar w>\frac \pi2,$$
and $k$ is the largest integer (e.g., $\bar x$ is weakly $k$-strained implies that $\bar x$ is not weakly $j$-strained for any $j>k$).

Note that in a small neighborhood $U$ around a weakly $k$-strained point $\bar p$, the map, $h=(d_{\bar p_1},..., d_{\bar p_k}): U\to \Bbb R^k$, is strictly non-critical (\cite{Per1}) and $h$ defines a fiber bundle structure  (see Theorem \ref{1g}). We conjecture the following:

\begin{canonical bundle conjecture}
\label{0g} Let $(X,S,\phi)$ be as in Theorem \ref{sol-alex},
$m=\dim(S)\ge 1$. If all points in $S$ are weakly $m$-strained, then $\phi: X\to S$ is fiber bundle map.
\end{canonical bundle conjecture}

We will show that Conjecture \ref{0e} implies Conjecture \ref{0g} (see Theorem D).

In Theorem \ref{sol}, (0.2.3) and (0.2.4) follows from that the flat strip property (0.2.2). A major obstacle in Conjectures 0.5-0.7 is from the fact that $X$ in Theorem \ref{sol} may not have a flat strip property (\cite{Li}).

Partial progresses on Conjectures 0.5-0.7 have been made in low dimensions, and Submetry Conjecture \ref{0e} has been verified when $S$ has a small codimension (for details, see Section 1.5).

In order to describe (conjectured) counterparts to Theorem \ref{rig}, we introduce the following notion of integrability for a submetry, $f: X\to Y$, where $X\in \text{Alex}^n(\kappa)$ and $Y\in \text{Alex}^m(\kappa)$.

Because a distance function has a well-defined `differential', a submetry $f$ has a well-defined differential on tangent cones,
$Df_{x}: C(\Sigma_{x}X)\to C(\Sigma_{f(x)}Y)$, such that $Df_{x}(
tv)=tDf_{x}(v)$ ($t\ge 0$), where $x\in X$, $v\in \Sigma_{x}X$. The subspace, $H_x=\{v\in \Sigma_{x}X,\,\, Df_{x}(v)\in \Sigma_{f(x)}Y\}\subseteq \Sigma_{x}X$, is convex and is called the horizonal directions at $x$, and $Df_{x}: H_x\to \Sigma_{f(x)}Y$ is also a submetry. The
orthogonal complement to $H_x$, $V_{x}=\{w\in \Sigma_{x}X,\, |wH_x|=\frac \pi2\}$ (note that in \cite{KL}, the condition is that $|wH_x|\ge \frac \pi2$; which easily implies that $|wH_x|=\frac \pi2$, see Lemma \ref{8a} in Appendix), is also
closed convex subset, called the vertical directions at $x$, and $\Sigma_{x}X=[H_xV_{x}]=\bigcup_{v\in H_x, w\in V_{x}}[vw]$, where $[vw]$ denotes a minimal geodesic connecting $v$ and $w$ (\cite{Lyt2}, \cite{KL}).

We will call the set of horizontal spaces of directions, $H(f)=\{H_x,\,\, x\in X\}$, the horizontal (directions) distribution of $f$.

\begin{definition} [Integrability of submetries] \label{0h} Let $f: X\to Y$ be a submetry, $X\in \text{Alex}^n(\kappa), Y\in \text{Alex}^m(\kappa)$.

\noindent (0.8.1) (Weakly integrable) We say that $f$ is weakly integrable at $x\in X$ (resp. $\bar x\in Y$), if (resp. for any $x\in f^{-1}(\bar x)$) there is a subset $W_x\ni x$ such that $f: W_x\to W_{\bar x} $ is an isometry, and $W_{\bar x}$ is a convex neighborhood of $\bar x=f(x)$ (the existence of $W_{\bar x}$ can be found in \cite{Per2}; note that throughout this paper, `convex' means that for $\bar y, \bar z\in W_{\bar x}$, there is one minimal geodesic from $\bar y$ to $\bar z$ that is contained in $W_{\bar x}$); note that $W_x\subset X$ and $\Sigma_xW_x\subset H_x$ are also convex subsets.

\noindent (0.8.2) (Integrable) We say that $f$ is integrable at $x\in X$ (resp. $\bar x\in Y$), if (resp. for any $x\in f^{-1}(\bar x)$) $f$ is weakly integrable at $x$ such that $\Sigma_{x}W_x=H_x$.

We say that $f$ is (weakly) integrable if $f$ is (weakly) integrable at all $x\in X$.

\noindent (0.8.3) (Globalization) In (0.8.1) or (0.8.2), if $W_{\bar x}=Y$, then we say that $f$ is global weakly integrable (or global integrable) at $x$. If $f$ is global weakly integrable (or global integrable) for all $x\in X$, then we say that $f$ is global weakly integrable or global integrable.
\end{definition}

If $f: X\to Y$ is a Riemannian submersion (i.e., $X$ and $Y$ are Riemannian manifolds), then $f$ is weakly integrable if and only if $f$ is integrable, because $f: W_x\to W_{\bar x}$ is an
isometry implies that $H_x=\Sigma_xW_x$. We conjecture that this holds in Alexandrov geometry.

\begin{conjecture} \label{0i} (0.9.1) (Weakly integrable is integrable) Let $X\in \text{Alex}^n(\kappa), Y\in \text{Alex}^m(\kappa)$ with $\partial Y=\emptyset$. If $f: X\to Y$ is weakly integrable, then $f$ is integrable.

\noindent (0.9.2) (Weakly integrable is local isometry) Let $\Sigma\in \text{Alex}^n(1), \Sigma_0\in \text{Alex}^m(1)$ with $\partial \Sigma_0=\emptyset$. If $h: \Sigma\to \Sigma_0$ is weakly integrable, then $n=m$.
\end{conjecture}

Conjecture  \ref{0i} is false if one removes the restriction that $\partial Y=\emptyset$ or $\partial \Sigma_0=\emptyset$ (see Example  \ref{1q}).

Note that (0.9.1) holds trivially if $X$ and $Y$ are Riemannian manifolds, and
(0.9.2) is not trivial if $\Sigma$ and $\Sigma_0$ are Riemannian manifolds (\cite{Wi}). To the contrary, we find that in Alexandrov geometry, (0.9.1) and (0.9.2) are equivalent, or equally non-trivial.

\begin{proposition} \label{0j} Conjectures in (0.9.1) and (0.9.2) are equivalent.
\end{proposition}

The following is the conjectured counterpart of (0.3.2):

\begin{conjecture} \label{0k} Let $(X,S,\phi)$ be as in Theorem \ref{sol-alex}. If $\phi$ is integrable, then
the metric universal cover of $X$ splits, $\tilde X=\tilde S\times \phi^{-1}(\bar p)$ ($\bar p\in S$), where $\pi: (\tilde X,\tilde p)\to (X,\bar p)$ is the metric universal cover (thus $\tilde S=\pi^{-1}(S)$ splits as $\hat S\times \mathbb R^k$, where $\hat S$ is compact).
\end{conjecture}

We now begin to state our results in this paper on Conjectures \ref{0g}, \ref{0i} and \ref{0k}; which are divided into two parts: the first part concerns a general weakly integrable submetry (Theorems A-C), and the second part focus on $(X,S,\phi)$ in Theorem \ref{sol-alex} and $\phi$ is a submetry (either assumed or verified) (Theorems D-F). Two main results in this paper are Theorems A and F.

A closed subset $E\subseteq Y$ is called extremal, if the gradient flows of any distance
function on $Y$ preserves $E$, and
an extremal subset $E$ is called primitive, if there is $\bar x\in E$ such that $E$ is the smallest extremal subset containing $\bar x$; in particular a primitive extremal subset is connected. Two points in a primitive extremal subset, $\bar x$ and $\bar y$ are equivalent, if $E$ is the smallest extremal subset containing $\bar x$ and $\bar y$. The subset of equivalent class of $\bar x$, $\overset o E$, is called the interior of $E$ (\cite{PP}).

\vskip2mm

\noindent{\bf Theorem A} {\rm (Properties of weakly integrable submetries)} {\it Let $X\in \text{Alex}^n(\kappa)$, $Y\in \text{Alex}^m(\kappa)$, and let $f: X\to Y$ be weakly integrable.

\noindent {\rm (A1)} For any $x\in X$, $Df_{x}: H_x\to \Sigma_{\bar x}Y$ is weakly integrable and global weakly integrable at any $v\in \Sigma_xW_x\subseteq H_x$. Moreover, if $\partial \Sigma_{\bar x}Y=\emptyset$, then $\partial H_x=\emptyset$.

Assume that $\partial Y=\emptyset$.

\noindent {\rm (A2)} {\rm ($f^{-1}$ preserves extremal subsets)} If $E\subseteq Y$ is an extremal subset, then $f^{-1}(E)$ is an extremal subset of $X$.

\noindent {\rm (A3)} {\rm (Constant vertical dimension over a primitive extremal subset)} If $E$ is a primitive extremal subset, then for all $x$ in a component of $f^{-1}(\overset o E)$, $\dim(V_x)$ is a constant independent of $x$.}

\vskip2mm

In investigating Conjecture  \ref{0i}, (A1) allows one to use inductive argument (e.g., in the proof of (A2), Theorems B and F, and Proposition  \ref{0j}), (A2) and (A3) are used in partially verifying Conjecture  \ref{0i} (Theorem B). Note that (A2) implies that any submetry from a Riemannian manifold
to an Alexandrov space without boundary and non-empty extremal subsets (\cite{KL} is not weakly integrable (e.g., Riemannian manifold to the orbit space of an isometric compact Lie group action with non-trivial isotropy groups (see Example  \ref{1u}).

A primitive extremal subset $E$ is called minimal (denoted by $E_{\min}$), if $E$ contains no proper extremal subset. A point $x\in X$ is called topologically nice, if the iterated space of
directions, $\Sigma_x, \Sigma(\Sigma_x),...,\Sigma(\Sigma (\cdots \Sigma_x))\cdots )$ are all homeomorphic to spheres. We call $X$ topologically nice, if every point in $X$ is topologically nice (\cite{Ka1}).

\vskip2mm

\noindent{\bf Theorem B} {\rm (Partial verifications of Conjecture  \ref{0i})} {\it \noindent {\rm (B1)} Let $f: X\to Y$ be as in (0.9.1). Then $f$ is integrable in the following cases: $n-m\le 1$, or $Y=E_{\min}$, or for every $E_{\min}$ contains an integrable point (e.g., each component of $f^{-1}(E_{\min})$ contains a point $x_0$ with $\dim(V_{x_0})=n-m-1$, or $x_0$ is topologically nice in $X$).

\noindent {\rm (B2)} Let $h: \Sigma\to \Sigma_0$ be as in (0.9.2). Then $n=m$ in the
following cases: $n-m\le 1$, or $\Sigma$ is topologically nice homeomorphic sphere.}

\vskip2mm

A consequence of (B1) is that if $X$ is topologically nice, then any weakly integrable submetry, $f: X\to Y$, is integrable such that $Y$ is topologically nice and a fiber is a topological manifold (see Corollary \ref{3e}).

The next result is closely related to Conjecture  \ref{0k}.

\vskip2mm

\noindent {\bf Theorem C} {\rm(Canonical local trivialization of integrable submetries)} {\it Let $X\in \text{Alex}^n(\kappa)$, $Y\in \text{Alex}^m(\kappa)$, and let $f: X\to Y$ be weakly integrable.
Then

\noindent {\rm (C1) (Canonical topologically splitting)} Assume that $f$ is integrable.
Then $(X,Y,f)$ is a fiber bundle with a canonical local trivialization: for $\bar x\in Y$, $\phi_{\bar x}: f^{-1}(W_{\bar x})\to f^{-1}(\bar x)\times W_{\bar x}$, $\phi_{\bar x}(z)=(x_z,f(z))$, $z\in W_{x_z}$. Moreover, if $Y$ is simply connected, then $X$ canonically splits into a product, $f^{-1}(\bar x)\times Y$.

\noindent {\rm (C2) (Splitting)} If every $f$-fiber is $\eta$-convex (i.e., if $x_1, x_2\in f^{-1}(\bar x)$ with $|x_1x_2|<\eta$, then there is a minimal geodesic from $x_1$ to $x_2$ that is contained in $f^{-1}(\bar x)$), then $f$ is integrable, and $\phi_{\bar x}$ is a homeomorphism and local isometry from $f^{-1}(W_{\bar x})$ to the metric product, $f^{-1}(\bar x)\times W_{\hat x}$.}

\vskip2mm

Note that by (C2), Conjecture  \ref{0k} reduces to show that if $(X,S,\phi)$ is weakly integrable
(see (0.9.1)), then a $\phi$-fiber is $\eta$-convex (see the proof of (F1)).

In the rest of the introduction, we will focus on $(X,S,\phi)$ in Theorem \ref{sol-alex}, and
we will always assume that $S$ is the soul of a Busemann function $b_{\bar p}$, and
$\bar p\in S$ is a regular point (see Lemma \ref{6a}). In particular, $b_{\bar p}(S)=0$, and
for $c<0$, $\Omega_c=b_{\bar p}^{-1}([c,0])\supsetneq S$.

\vskip2mm

\noindent {\bf Theorem D} {\rm (Conjecture \ref{0e} implies Conjecture \ref{0g})} {\it Let $(X,S,\phi)$ be as in Theorem \ref{sol-alex}. Assume that $\phi: X\to S$ is a submetry. If all points on $S$ are weakly $m$-strained, then $\phi: X\to S$ is a fiber bundle.}

\vskip2mm

It may be possible that Theorem D still holds when one weakens that $\phi$ is a submetry to
that $\phi$ is an $\epsilon$-submetry $(\epsilon<<1$) i.e.,
for $r>0$ and $x\in X$, $B_{e^{-\epsilon}r}(\phi(x))\subseteq \phi(B_r(x))\subseteq B_{e^\epsilon r}(\phi(x))$ (\cite{RX}).

Our proof relies on Perelman's construction of a local trivialization (\cite{Per3}).

For $\bar x\in S$, let $\Sigma^\perp_{\bar x} S=\{\bar v\in \Sigma_{\bar x}X,\,\, |\bar v\bar w|=\frac \pi2, \, \bar w\in \Sigma_{\bar x}S\}$, $C(\Sigma^\perp_{\bar x}S)=\{t\bar v,\, \bar v\in \Sigma^\perp_{\bar x}S, t\ge 0\}\ne \emptyset$, and let $C(\Sigma^\perp S)=\bigcup_{\bar x\in S}C(\Sigma_{\bar x}^\perp S)$. Then the gradient-exponential map, $g\exp: C(\Sigma^\perp S)\to X$, is onto (see Lemma \ref{ya}), thus $C(\Sigma^\perp S)$ is equipped with the pullback topology.

The following result is a counterpart of (0.3.1) (and (0.2.1)) in Alexandrov geometry.

\vskip2mm

\noindent {\bf Theorem E} {\it Let $(X,S,\phi)$ be as in Theorem \ref{sol-alex}. Assume $S$ has
a regular point $\bar p$ such that every $v\in \Sigma_{\bar p}^\perp S$ tangents to a ray. Then the following hold:

\vskip1mm

\noindent {\rm (E1) (Canonical foliation, submetries)} $\phi: X\to S$ is a submetry, and $g\exp: C(\Sigma^\perp S)\to X$ is a bijection such that $\phi\circ g\exp_{\bar x}(t\bar v)=\bar x$ for all $t\ge 0$.

\noindent {\rm (E2) (Canonical bundles)} If all points in $S$ are topologically nice (in $X$), then $(X,S,\phi)$ is a fiber bundle with fiber homeomorphic to an Euclidean space, and $S$ itself is topologically nice. Moreover, $g\exp: C(\Sigma^\perp S)\to X$ is a bundle isomorphism of $(C(\Sigma^\perp S),S,\text{proj})$ and $(X,S,\phi)$.}

\vskip2mm

Note that under the condition of Theorem E, (E1) verifies Submetry Conjecture \ref{0e}, and (E2) verifies a conjecture by Yamaguichi (\cite{Ya}) which says that if $X$ in Theorem \ref{sol-alex} is topologically nice, then a tubular neighborhood of $S$ is a disk bundle over $S$. Note that different from Theorem D, $S$ in (E2) may have an non-empty extremal subset (e.g., the metric product of $\Bbb R^k$ and spherical suspension over a $l$-sphere of constant curvature $\ge 4$.), while (E2) will be false without the assumption that points in $S$ are topologically nice (see Example  \ref{1t}).

Restricting to the case that $S$ is of codimension $2$, we obtain a counterpart in Alexandrov geometry to (0.3.3).

\vskip2mm

\noindent {\bf Theorem F} {\rm (Codimension $2$ rigidity)} {\it Let $(X,S,\phi)$ be as
in Theorem \ref{sol-alex} such that $\dim(X)-\dim(S)=2$. Assume that $X$ is topologically nice. Then $X$ satisfies one of the following two rigid properties:

\noindent {\rm (F1) (Canonical bundle)} $g\exp: C(\Sigma^\perp S) \to X$ defines a bundle isomorphism of $(C(\Sigma ^\perp S),S, \text{proj})$ and $(X,S,\phi)$, where $S$ is topologically nice and $\phi$-fiber is homeomorphic to $\Bbb R^2$. 

\noindent {\rm (F2) (Splitting)} The metric universal cover of $X$ splits, $\tilde X=\tilde S\times \phi^{-1}(\bar p)$, where $\tilde S$ is the metric universal cover of $S$, thus $\tilde S$ splits into $\Bbb R^k\times \hat S$, where $k$ is the maximal rank of a
free abelian subgroup of $\pi_1(X)$ and $\hat S$ is compact and simply connected.}

\vskip2mm

Note that Theorem F does not hold if $X$ is not assumed to be topologically nice, while
a topologically nice $X$ in Theorem F may have a non-empty extremal subset; see Example \ref{1w}.

In early work \cite{Li}, it is already obtained that $(X,S,\phi)$ in Theorem F satisfies that $g\exp: C(\Sigma^\perp S)\to X$ is a homeomorphism or that $\phi: X\to S$ is weakly integrable (see Theorem \ref{7a}). By (E2), $g\exp$ is a bundle isomorphism. We point it out that
in the proof of Theorem F, the main accomplishment is that a weakly integrable $\phi$ is integrable with fiber a convex subset (see (C2)), which verifies Conjecture \ref{0k} in the case that $\dim(S)=\dim(X)-2$.


\vskip4mm

The rest of the paper will be organized as follows:

\vskip2mm

In Section 1,  we will review basic notions and properties of an Alexandrov space that
will be used throughout this paper.

\vskip2mm

In Section 2, we will prove Theorem A and Proposition  \ref{0j}. The main technical result is
Theorem \ref{2a}, in proving it we establish several results that have its-own interest (e.g., Lemma \ref{2f}, Lemma \ref{2h}).

In Section 3, we will prove Theorem B.

In Section 4, we will prove Theorem C.

In Section 5, we will prove Theorem D.

In Section 6, we will prove Theorem E.

In Section 7, we will prove Theorem F.

\vskip4mm

\noindent {\bf Acknowledgement}: The present paper is a revision of arXiv:2311.15174v1. The authors are very grateful to a referee whose detailed report (questions/comments/suggestions)
stimulates authors to re-organize the proofs of Theorem A (with a new Theorem 2.1).
In particular, the referee's suggestions are valuable to our simplifying
proofs, or improving some results in the revision; notably proofs of Lemma \ref{1a},
Proposition  \ref{0j}, (A2), and improving Theorem D, among several others.

\vskip4mm

\section{Preliminaries}

\vskip4mm

In this section, we will review basic notions and properties of an Alexandrov space that
will be used throughout this paper. The general reference are \cite{AKP} and \cite{BBI}.

We start this section with fixing some notations:

\vskip2mm

\noindent $\bullet$ $Z$ denotes a complete metric space. For $x, y\in Z$, or subsets, $A, C\subset Z$:

(a) $|xy|$ denotes the distance between $x$ and $y$,  $|xA|=\inf \{|xy|,\, y\in A\}$, for $r>0$, $B_r(A)=\{x\in X,\,\, |xA|<r\}$.

(b) $\bar A=\{x\in Z, \,\, x_n\in A, x_n\to x\}$, $\overset{o}A=\{x\in A, \exists \, r>0, B_r(x)\subset A\}$, $\partial A=\bar A\setminus \overset{o}A$, $A\setminus  C=\{x\in A,\, x\notin C\}$, $\overset{\vee}A=\bar A\setminus \partial A$.

(c) $\pi: \tilde Z\to Z$ denotes the universal covering map, we always assume that $Z$ has a
simply connected universal cover.

\vskip2mm

\noindent $\bullet$ $\text{Alex}^n(\kappa)$ denotes the set of complete n-dimensional Alexandrov spaces $X$ with curvature $\ge \kappa$. For $x, y\in X$,

\vskip1mm

(d) $\Sigma_xX$ denotes the space of directions at $x\in X$, $C_xX=C(\Sigma_xX)$ denotes the metric cone at $x$ over $\Sigma_xX$, and $g\exp_x: C_xX\to X$ denotes the gradient-exponential map.

(e) $[xy]$ denotes a minimal geodesic from $x$ to $y$, $\uparrow_x^y$ denotes the direction of $[xy]$ at $x$, and $\Uparrow_x^y$ denotes the set of directions at $x$ of all possible $[xy]$.

(f) $F\subset X$ denotes a closed subset of $X$, $d_F: X\to \mathbb R_+$, $d_F(x)=|xF|=\min\{|xz|,\,z\in F\}$, denotes the distance function to $F$, and $\nabla d_F$ denotes the gradient of $d_F$.

(g) $E\subset X$ denotes an extremal subset of $X$ i.e., a closed subset preserved by the gradient flows of any distance function on $X$.

\vskip2mm

\noindent $\bullet$ $(X,Y,f)$: $X\in \text{Alex}^n(\kappa), Y\in \text{Alex}^m(\kappa)$, $f: X\to Y$ denotes a submetry, for $x\in X$, $Df_x: C_xX\to C_{\bar x}Y$ denotes the differential of $f$, $\bar x=f(x)$. For $\bar p\in Y$, we call the function, $d_{f^{-1}(\bar p)}: X\to \mathbb R_+$, the horizontal lifting of the distance function, $d_{\bar p}: Y\to \mathbb R_+$ (more general, if $h: Y\to \mathbb R$ is a function, then we call $h\circ f: X\to \mathbb R$ the horizontal lifting of $h$). A basic property of a submetry is horizontal lifting of a curve (which may not be unique at a given point). For a gradient curve of $d_{\bar p}$, its horizontal horizontal lifting is a gradient curve of $d_{f^{-1}(\bar p)}$ (\cite{Lyt2}, \cite{KL}).

\vskip1mm

(h) $H_x=\{u\in \Sigma_xX,\,\, Df_x(u)\in \Sigma_{\bar x}Y\}$ and $V_x=\{v\in \Sigma_{\bar x}Y,\,\, |vH_x|=\frac \pi2\}$ denote convex subsets of $\Sigma_xX$.

(i) $W_x$ denotes a subset of $X$ such that $f: W_x\to W_{\bar x}=f(W_x)$ is an isometry and $W_{\bar x}$ is a convex neighborhood of $x$. In particular, $\Sigma_xW_x\subseteq H_x$, and
$\Sigma_xX=[H_xV_x]=\bigcup_{v\in H_x, w\in V_{x}}[vw]$, $H_x*V_x$ denotes the join of
$H_x$ and $V_x$ i.e., $H_x*V_x=[H_xV_x]$ and for all $v\in H_x$ and $w\in V_x$, $[vw]$ is unique.

\vskip2mm

\noindent $\bullet$ {\bf $(X,S,\phi)$}: open $X\in \text{Alex}^n(0)$, $S$ denotes a soul of $X$, and $\phi: X\to S$ denotes a Sharafutinov projection map, which is $1$-Lipschitz.

\vskip1mm

(j) For $\bar p\in S$, $b_{\bar p}$ denotes the Bussmann function at $\bar p$, which is a proper and concave function, $S\subseteq b_{\bar p}^{-1}(0)$, for $c<0$, $\Omega_c=b_{\bar p}^{-1}([c,0])$ denotes a suplevel set, $\partial \Omega_c$ denoted the boundary of $\Omega_c$, and for $x\in X$, $b_{\bar p}(x)>c$, $\Uparrow_x^{\partial \Omega_c}$ denotes the set of directions of $[xz], z\in \partial \Omega_c$ such that $|xz|=|x\partial \Omega_c|$.

\subsection {Stratifications, submetries, of Alexandrov spaces, and lifting gradient flows}

A reference for this subsection is \cite{PP}.

Let $Y\in \text{Alex}^m(\kappa)$. A closed subset $E\subseteq Y$ is called extremal, if for any $y\in Y\setminus E$, $q\in E$ is a minimum of $d_y|_E$, $d_y(q)=\min\{d_y(z), \, z\in E\}$, then $q$ is a critical point of $d_y: X\to \mathbb R$. A basic property of $E$ is that the gradient
flows of $d_y$ preserves $E$.

For $y\in Y$, let $\text{Ext}(y)$ denote the smallest extremal subset that contains $y$, which is called a primitive extremal subset. Two points $y, y'\in Y$ are equivalent, if and only if $\text{Ext}(y)=\text{Ext}(y')$. Then $Y$ is
a disjoint union of equivalent classes, which form a stratification on $Y$ i.e., $Y$ is disjoint
union of strata, and each stratum is an open manifold. We may group strata in terms of dimensions, $Y_i=\bigcup _j Y_{i,j}$, $\dim(Y_{i,j})=\dim(Y_{i,k})
=m_i$, $0\le m_1<m_2<\cdots <m_s=m$. A point $x\in Y_{i,j}$ is called regular (in $Y_{i,j}$) if
$C_xY_{i,j}\cong \mathbb R^{m_i}$ (\cite{Fuj1}). We will always use the following stratification
notions:
$$Y_1, \,\, Y_2\,\, ,...,\,\, Y_s, \quad \bar Y_{i,j}=Y_{i,j} \quad \text{or} \quad
\bar Y_{i,j}= Y_{i,j}\cup \left(\bigcup_{Y_{j,k}\subset \bar Y_{i,j}}Y_{j,k}\right),$$
and the closure, $\bar Y_{i,j}$, is a primitive extremal subset. Any extremal subset of
$Y$ is a union of a number of $\bar Y_{i,j}$.

\begin{lemma} {\rm (Deforming interior points in $Y_{i,j}$)} \label{1a} Let $Y\in \text{Alex}^m(\kappa)$,
and let $Y_1,...,Y_s$ be a stratification as above. Then for $p\in Y_{i,j}$, $q\in \bar Y_{i,j}$, there is a finite number of distance functions whose
gradient flows deform $q$ to $p$ in a finite time. In particular, any $p, q\in Y_{i,j}$ can
be push back and forth in a finite time along gradient flows.
\end{lemma}

Lemma \ref{1a} is a basic known property; because we do not find it in literature, for
completeness we present a proof below.

\begin{proof} [Proof of Lemma \ref{1a}] For $p\in Y_{i,j}$, let $G(p)=\{q|  \ \  \exists \text{ semi-concave functions }f_i, 1\leq i\leq N, \text{ and }a_i\in \mathbb R,1\leq i\leq N, \text{ such that }
q=\Phi_{f_N}^{t_N}\circ \Phi_{f_{N-1}}^{t_{N-1}} \circ \dots \circ \Phi_{f_1}^{t_1}(p), 0\leq t_i\leq a_i\}$. We claim that $G(p)$ is closed; if $z$ denotes an accumulation point of $G(p)$,
then $z$ has a convex neighborhood in $Y$ defined by a semi-concave function (\cite{Per2}); which implies that $z\in G(p)$.

Note that $G(p)$ is closed and invariant under gradient flow of semi-concave functions. Thus $G(p)$ is an extremal subset, and therefore $G(p)=\bar Y_{i,j}$.
\end{proof}

Assume that $f: X\to Y$ is a submetry. For $\bar x\ne \bar y\in Y$, let $\gamma(t)$ ($0\le t\le 1$) denote a gradient curve of $d_{\bar z}$ that flows an interior point $\bar x=\gamma(t_0)$ to $\bar y$. For any $x\in f^{-1}(\bar x)$, let $\tilde \gamma_x(t)$ be the gradient curve of $d_{f^{-1}(\bar z)}$ through $x$. Then $\psi: f^{-1}(\bar x)\to f^{-1}(\bar y)$ is a locally Lipschitz map, $\psi(x)=\tilde \gamma_x(1)$ (\cite{Pet}).

Recall that if $N$ is a closed totally geodesic submanifold of a Riemannian manifold $M$, then for $p\in N$, restricting to $T_xN\subset T_xM$ the exponential map of $M$ coincides with the exponential map of $N$ equipped with the induced metric.

The following counterpart of the above property in Alexandrov geometry (\cite{ SSW}) will be used in the proof of Lemma \ref{6a}.

\begin{lemma} \label{1b} Let $X\in \text{Alex}^n(\kappa)$, and let $W\subset X$ be a convex subset. For any $x\in W\setminus \partial W$ ($\partial W$ is treated as the boundary of the Alexandrov space $W$), $v\in \Sigma_xW$, then the gradient-exponential maps, $g\exp_x^Xtv=g\exp_x^Wtv$, as long as $g\exp_x^Wtv\in W\setminus \partial W$ is well-defined, where $g\exp_x^Z$ denotes the gradient-exponential map of an Alexandrov space $Z$.
\end{lemma}

A proof of Lemma \ref{1b} can be found in \cite{SSW}; we will also present a different direct proof in the Appendix.

\subsection{Stabilities, and local structures}

A fundamental result in Alexandrov geometry is the following Perel'man stability theorem.

\begin{theorem} [Stability, Perel'man, \cite{Per1}, Kapovitch, \cite{Ka2}] \label{1c} Let compact $X, Y\in \text{Alex}^n(\kappa)$. There is $\epsilon(X)>0$ such that if $d_{\text{GH}}(X,Y)<\epsilon(X)$, then there is a homeomorphism which is also an $\epsilon$-distance distortion (also called Gromov-Hausdorff approximation, briefly GHA), $f: X\to Y$.
\end{theorem}

In the proof of (A1), we shall use the following local refinement of Theorem \ref{1c}.

\begin{theorem} {\rm (\cite{Ka2})} \label{1d} Let $Y\in \text{Alex}^m(\kappa)$. Given $\bar x\in Y$, there is $\rho=\rho(\bar x)>0$ such that

\noindent {\rm (1.4.1)} $d_{\bar x}$ has no critical point in $B_\rho(\bar x)\setminus \{\bar x\}$.

\noindent {\rm (1.4.2)} There is a homeomorphism, $\psi: B_\rho(\bar x)\to B_\rho(o, C_{\bar x}Y)$, $\psi(\bar x)=\bar o$ and $\phi$ maps $B_\rho(\bar x)\cap E$
to an extremal subset $B_\rho(o,C_{\bar x}Y)$, where $E$ denotes an extremal subset in $Y$.
\end{theorem}

\begin{corollary}\label{1e} Let $Y_i, Y \in \text{Alex}^m(\kappa)$, $Y_i\xrightarrow {GH} Y$, For $\bar y\in Y$, let $\rho(\bar y)>0$ in Theorem \ref{1d}, there is a sequence, $\bar y_i\in Y_i$, such that for $i$ large Theorem \ref{1d} holds at $\bar y_i$ with $\rho(\bar y_i)\ge \rho(\bar y)$.
\end{corollary}

For simplicity, we will state the following theorem in somewhat restrictive form that is all
required in this paper; precisely if $\bar x$ is a weakly $k$-strained point, a canonical neighborhood $U_{\bar x}\supset K(h_x,g_x)$, where $h_{\bar x}$ is $(d_{\bar p_1},...,d_{\bar p_k})$ (in general, $h_{\bar x}: U_{\bar x}\to \Bbb R^l, l\le k$, is defined not using
$\bar p_1,...,\bar p_k$).

Let $\b B_r^n(H)$ denote an $r$-ball in the simply connected $n$-manifold of constant sectional curvature $H$ (e.g., $H=0, \pm 1$).

\begin{theorem} {\rm (Canonical neighborhood, \cite{Per2}, \cite{Ka2}, \cite{Fuj2})} \label{1f} Let $Y\in \text{Alex}(\kappa)$, let $\bar x\in Y$ be a weakly $k$-strained point. Then there is an open neighborhood $U_{\bar x}\supset K(h_x,g_x)$ (called a canonical neighborhood of $\bar x$), and map, $\tau_{\bar x}=(h_{\bar x},g_{\bar x}): U_{\bar x}\to \Bbb R^k\times \Bbb R$, such that $K(h_{\bar x},g_{\bar x})$ is homeomorphic to $\b B^k_\rho(0)\times C(\Sigma)$, where $h_{\bar x}=(d_{\bar p_1},...,d_{\bar p_k})$, and $g_{\bar x}^{-1}(0)$ is homeomorphic to $\b B_\rho^k(0)$, and $C(\Sigma)\cong ([0,1]\times \Sigma)/((\{0\}\times \Sigma)\sim 0)$ denotes a topological cone. Moreover, points in
$K(h_{\bar x},g_{\bar x})\setminus g_{\bar x}^{-1}(0)$ are at least $(k+1)$-strained points.
\end{theorem}

Similar to Theorem \ref{1d}, in this paper we need a refinement of local homeomorphism $\phi_{\bar x}$ that
preserves extremal subsets.

\begin{theorem} {\rm (Theorem 9.7 in \cite{Ka2})} \label{1g} Let $E\subset Y\in \text{Alex}^m(\kappa)$ be an extremal subset, let $h: Y\to\Bbb R^k$ be regular at
$\bar p\in E$.
Then there exists an open neighborhood $U$ of $\bar p$, and an MCS-space $A$,
a stratified subspace $B\subset A$, and a homeomorphism $\psi: (U, E\cap U)\to (A,B)\times \Bbb R^k$ such that $\text{proj}_2\circ \phi=h$.
\end{theorem}

\subsection {Rigid structures on $\Sigma_{\bar x}X$}

Let $(X,S,\phi)$ be as in Theorem \ref{sol-alex}, and we will always consider the case that $\dim(S)=m\ge 1$.  Let's begin to review certain rigid geometrical structures on $\Sigma_{\bar x}X$, $\bar x\in S$, which may not have on $\Sigma_xX$, for $x\in X\setminus S$.

First, $\Sigma_{\bar x}X$ contains two compact convex subsets, $\Sigma_{\bar x}S$ with $\partial \Sigma_{\bar x}S=\emptyset$,
and $\Sigma_{\bar x}^\perp S=\{w\in \Sigma_{\bar s}X,\, |wv|=\frac \pi2, \, v\in \Sigma_{\bar x}S\}$. By the following lemma,
for $z\in \partial \Omega_c$ that realizes the distance to $S$ at $\bar x$, $c<0$ ($b_{\bar p}(S)=0$), $\uparrow_{\bar x}^z\in \Sigma_{\bar x}^\perp S\ne \emptyset$.

Here one needs the following fact:

\begin{lemma} {\rm(\cite{Ya})} \label{ya} Let $\Sigma\in \text{Alex}^n(1)$, and let $\Sigma_0\subset\Sigma$ be a closed locally convex subset without boundary (view $\Sigma_0$ as an
Alexandrov space) which is not a point. If there is a $p\in \Sigma$ such that $|p\Sigma_0|\geq \frac \pi 2$, then $|pq|=\frac \pi 2$, for any
$q\in \Sigma_0$.
\end{lemma}

\begin{proposition} {\rm(Rigid structures on $\Sigma_{\bar x}X$)} \label{1i} Let $(X,S,\phi)$ be as in Theorem \ref{sol-alex}, and $\bar x\in S$.

\noindent {\rm (1.9.1)} $\Sigma_{\bar x}S$ and $\Sigma_{\bar x}^\perp S$ are two compact convex subsets of $\Sigma_{\bar x}X$,
which are $\frac \pi2$-apart i.e., $|vw|=\frac \pi2$ for any $v\in \Sigma_{\bar x}S$ and $w\in \Sigma_{\bar x}^\perp S$.

\noindent {\rm (1.9.2)} Given any three points in $\Sigma_{\bar x}S\cup \Sigma_{\bar x}^\perp S$ (not all in one component), say $v_1, v_2\in \Sigma_{\bar x}S$ and $w\in \Sigma_{\bar x}^\perp S$, there is an isometric embedding of a triangle in $S^2_1$, $\tilde \triangle (v_1,v_2,w)\hookrightarrow
\Sigma_{\bar x}X$, whose sides coincide with $[v_1w]$ and $[v_1v_2]$.

\noindent {\rm (1.9.3)} If $\phi: X\to S$ is a submetry, then at $\bar x\in S$, $H_{\bar x}=\Sigma_{\bar x}S$ and $V_{\bar x}=\Sigma_{\bar x}^\perp S$.

\noindent {\rm (1.9.4)} If $\bar x$ is regular in $S$ i.e., $\Sigma_{\bar x}S$ is isometric to $S^{m-1}_1$, then $\Sigma_{\bar x}X\cong \Sigma_{\bar x}S*\Sigma_{\bar x}^\perp S$, the join of $\Sigma_{\bar x}S$ and $\Sigma_{\bar x}^\perp S$.
\end{proposition}

\begin{proof} (1.9.1) A proof is based on Lemma \ref{ya}

(1.9.2) follows from the rigidity of the equal case in the Toponogov comparison theorem.

(1.9.3) It suffices to show that $H_{\bar x}=\Sigma_{\bar x}S$. Recall that $V_x=\{v \in \Sigma_xX, \, |vH_x|=\frac \pi2\}$ (see Appendix, Lemma \ref{8a}), and that $S$ is obtained through a sequence of compact convex subsets, $b_{\bar p}^{-1}(0)=C_0\supset C_1\supset
\cdots \supset C_s=S$. Then
$$\Sigma_{\bar x}S=\{v, |v\Uparrow_{\bar x}^{\partial \Omega_{c<0}}|=\frac \pi2,\, |v\Uparrow_{\bar x}^{\partial C_i}|=\frac \pi2, \, 0\le i\le s\},$$
Because $\Uparrow_{\bar x}^{\partial \Omega_{c<0}}\cup \Uparrow_{\bar x}^{\partial C_0}\cup \cdots \cup \Uparrow_{\bar x}^{\partial C_s}\subseteq V_{\bar x}$, $H_{\bar x}\subseteq \Sigma_{\bar x}S$, thus $H_{\bar x}=\Sigma_{\bar x}S$.

(1.9.4) Because $\bar x\in S$ is regular, $C_{\bar x}X$ contains a $\mathbb R^m$-factor,
thus $C_{\bar x}X$ splits off a $\mathbb R^m$-factor.
\end{proof}

The following result will be used in the proof of Lemma \ref{3a}.

\begin{theorem} [Finite quotient of joins, \cite{RW1}] \label{1j}
Let $X\in \text{Alex}^n(1)$, and let $X_i$ $(i=0,1$) be convex subsets of $X$ such that $X_0$ and $X_1$ are $\frac \pi2$-apart i.e., for any $x_i\in X_i$,
$|x_0x_1|=\frac \pi2$.

\noindent {\rm (1.10.1)} If $\partial X_0=\emptyset$, then $\dim(X_0)+\dim(X_1)\le n-1$.

\noindent {\rm (1.10.2)} If $\partial X_i=\emptyset$ ($i=0, 1$), and  $\dim(X_0)+\dim(X_1)=n-1$, then $X$ is isometric to $\hat X/\Gamma$, such that $\hat X\in \text{Alex}^n(1)$ has a join structure and $\Gamma$ is a finite group of isometries.
\end{theorem}

\begin{remark}\label{1k} Observe that if $X$ in Theorem \ref{1j} does not have a join structure, and if any $x\in X_0\cup X_1$, say $x\in X_0$, $\Sigma_{x}X$ is isometric to $\Sigma_{x}X_0*\Sigma_x^\perp X_0$, then $\Gamma$ acts freely on $\hat X$, thus $X$ is not simply connected.
\end{remark}

\subsection {Partial flat strips}

The main ingredient in Theorems \ref{sol} and \ref{rig} is the flat strip property (0.2.2) discovered by Perel'man; for instance (0.2.2) implies (0.2.3), (0.3.1) and (0.3.2). As mentioned seen
in the introduction, an open Alexandrov space of nonnegative curvature with a soul may not have the above flat strip property (\cite{Li}); which is the main difficulty in
Conjectures \ref{0e}-\ref{0g}, \ref{0i} and \ref{0k}.

A basic tool in the proofs of Theorems E and F is the following partial generalization of (0.2.2):

\begin{theorem} {\rm (Partial flat strip property, Yamaguichi, \cite{Ya})} \label{1l} Let $X\in \text{Alex}^n(0)$, let $\Omega\subset X$ be a convex closed subset, with
$\partial \Omega\neq \emptyset$, let $f=d_{\partial \Omega}$ and let
$\gamma(t)\subset \Omega$ ($t\in [0,b]$) be a minimal geodesic with
$\gamma(0)=p,\ \gamma(b)=q$, such that $f(\gamma(t))$ is a constant. Then
for any minimal geodesic $\gamma_0$ from $p$ to $\partial \Omega$, with
$|\gamma_0^+(0)\gamma^+(0)|=\frac\pi2$, there is a minimal
geodesic $\gamma_1$ from $q$ to $\partial \Omega$, such that $\{\gamma,
\gamma_0,\gamma_1\}$  bounds a flat totally geodesic rectangle.
\end{theorem}

Let $(X,S,\phi)$ be as in Theorem \ref{sol-alex}. For $\bar x\in S$ and $c<0$, let $\Uparrow_{\bar x}^{\partial \Omega_c}=\{v\in \Sigma_{\bar x}^\perp S,\, v=\uparrow_{\bar x}^z, z\in \partial\Omega_c, \,  |\bar xz|=|\bar x\partial \Omega_c|\}\subseteq \Sigma_{\bar x}^\perp S$. Using Theorem \ref{1l}, one constructs, given any $v\in \Uparrow_{\bar x}^{\partial\Omega_c}$, $\bar x\in S$, a flat strip as follows.

\begin{lemma} {\rm (\cite{Li})} \label{1m} Let $(X,S,\phi)$ be as in Theorem \ref{sol-alex}. Then

\noindent {\rm (1.13.1) (Normal rays)} For any $\bar x\in S$, $v\in \Uparrow_{\bar x}^{\partial \Omega_c}$ tangents to a ray in $X$.

\noindent {\rm (1.13.2) (Normal rays and flat strips)} For any $\bar x\ne \bar z\in S, v\in \Uparrow_{\bar x}^{\partial \Omega_c}$, the normal ray $g\exp_{\bar x}tv$ and minimal geodesic $[\bar x\bar z]$ bound a flat strip in $X$ i.e., they are the boundaries of an isometric embedding, $\mathbb R_+\times [0,|\bar x\bar z|]\rightarrow X$.
\end{lemma}

\begin{corollary} {\rm (Global weakly integrable and flat strips)} \label{1n} Let $(X,S,\phi)$ be as in Theorem \ref{sol-alex}. Assume that $\phi: X\to S$ is global weakly integrable. For any $x\in X\setminus S$, $v\in \Uparrow_x^{\partial \Omega_c}$, $x\ne z\in W_x$, the normal ray $g\exp_xtv$ and minimal geodesic $[xz]\subset W_x$ bound a flat strip in $X$ i.e., they are the boundaries of an isometric embedding, $\mathbb R_+\times [0,|xz|]\rightarrow X$.
\end{corollary}

Note that (0.2.2) (Flat strips) is equivalent to a geodesic preserving embedding of $[0,a]\times \mathbb R^1_+$, thus an isometry to the image with the intrinsic metric, while the
image may not be a convex subset. Note that the flat strip in Lemma \ref{1m} is convex; the intrinsic metric on image coincides with the extrinsic metric.

\begin{example}(Fake flat strip) \label{1o} Note that a flat strip in Lemma \ref{1m} is a convex subset of $X$; the flat strip can be expressed as disjoint union of finite segment, or minimal geodesic isometric to $\mathbb R$. The following example is an embedding of $[0,\frac \pi2]\times [0,\frac \pi2] \hookrightarrow X$ satisfies that $[0,\frac \pi2]\times \{t_0\}\hookrightarrow X$, $\{t_0\}\times [0,\frac \pi2]\hookrightarrow X$ are minimal geodesics, but $\{(t,t)\in [0,\frac \pi2]\times [0,\frac \pi2],\, 0\le t\le \frac \pi2\}\hookrightarrow X$ is not a minimal geodesic.

Let $X=[0,\frac \pi2]\times [0,\frac \pi2]\cup _{\partial }\triangle_{\frac \pi2}(S^2_1)$ (a quarter of $S^2_1$), where the boundary map $\partial$ identifies $[0,\frac \pi2]\times \frac \pi2\cup \frac \pi2\times [0,\frac \pi2]$ with two sides of $\triangle_{\frac \pi2}(S^2_1)$.
Observe that $f: Y=[0,\frac \pi2]\times [0,\frac \pi2]\hookrightarrow X$ is an embedding with image satisfying the above properties. Note that $f(Y)$ is not a flat strip in $X$ (as in Lemma \ref{1m}), because $f(Y)$ is not convex in $X$.
\end{example}

\subsection {Partial progress in Conjectures \ref{0e}-\ref{0g}, and examples}

Conjectures \ref{0e}-\ref{0g} have been verified in low dimensions, and Conjecture \ref{0e} has also been verified when $S$ has a small codimension:

\vskip2mm

\noindent (1.16.1) Recall that Submetry Conjecture \ref{0e} implies Soul Conjecture \ref{0f}. Submetry Conjecture \ref{0e} holds for $n=3$ (\cite{SY}), $n=4$ (\cite{RW2}, the case that $X$ is a topological manifold was due to \cite{Li}), and the case that $X$ is topologically nice and $\dim(S)=n-1$ (\cite{Ya}), and $\dim(S)=n-2$ (\cite{Li}).

\vskip1mm

\noindent (1.16.2) The bundle conjecture by Yamaguichi (\cite{Ya}) holds for $n=3$ (\cite{SY}), and $n=4$ (\cite{Ya}, also \cite{Ge}).

\noindent (1.16.3) The bundle conjecture by Yamaguichi holds in the cases of $\dim(S)=n-1$ (follow from Theorem \ref{1l}), and $\dim(S)=n-2$ (\cite{Li}).

Basic tools used in proving the above results are (among others) Proposition \ref{1i}, Theorem \ref{1l} and Lemma \ref{1m}.

\vskip1mm

\stepcounter{theorem}

\begin{example} (`Counterexamples' to Conjecture  \ref{0i}) \label{1q} Let $(\Sigma,\Sigma_0,h)=(S^2_1,S^2_1/S^1, \text{proj})$, where $S^1$ acts isometrically on the unit $2$-sphere $S^2_1$), and let $(X,Y,f)=(C(S^2_1),C(S^2/S^1),f)$, wehre $f(t,u)=(t,h(u))$. Then $h$ and $f$ are weakly integrable, but $h$ is not a local isometry, nor $f$ is integrable; thus
(0.9.2) and (0.9.1) are false if one removes condition $\partial \Sigma_0=\emptyset$ or
$\partial Y=\emptyset$.
\end{example}

\begin{example} (Weakly integrable submetry and extremal subsets) \label{1r}
Consider a submetry between two Alexandrov spaces, $f: X\to Y$. If
$E\subset X$ is an extremal subset, then it is easy to see that $f(E)$ is an
extremal subset of $Y$, while the converse is false if $f$ is not weakly integrable (compare with (A2)).

Let $M$ be a Riemannian manifold of sectional curvature $\ge -1$ which admits an isometric torus $T^k$-action. If the $T^k$-action is not free, then $M/T^k$ is an Alexandrov space of curv $\ge -1$, such that $\text{proj}: M\to M/T^k$ is a submetry and the projection of the fixed point set of a maximal isotropy subgroup is a proper extremal subset of $M/T^k$; while $M$ has no extremal subset.
\end{example}

\begin{example} (A case of Canonical bundle Conjecture \ref{0g}) \label{1s} Let $(X,S,\phi)$ be as in Theorem \ref{sol-alex}. If $\phi$ is a submetry and $S$ is a Riemannian manifold, then $\phi$ is a fiber bundle map.
To see this, let $\rho=\text{injrad}(S)>0$. We now construct a local trivialization: for $z\in B_\rho(\bar s)\subset S$, $x\in \phi^{-1}(\bar s)$, $\psi: B_\rho(\bar s)\times \phi^{-1}(\bar s)\to \phi^{-1}(B_\rho(\bar s))$, $\psi(\bar z,x)=\tilde \gamma(1)$, where $\tilde \gamma$ is the horizontal lifting at $x$ of the unique
minimal geodesic from $\bar s$ to $\bar z\in B_\rho(\bar s)$.
\end{example}

\begin{example} (In Theorem \ref{sol-alex}, $\phi$ is a submetry but may not a bundle map) \label{1t} Let tangent bundle over unit sphere, $TS^2_1$ be equipped with a canonical metric, which has nonnegative sectional curvature, whose
soul is  $S^2_1$. Let $\mathbb Z_h$ acts isometrically on $S^2_1$, whose differentials defines
an isometric $\mathbb Z_h$-action on $TS^2_1$. Let $X=TS^2_1/\mathbb Z_h$, an open Alexandrov space of nonnegative curvature with a soul $S^2_1/\mathbb Z_h$. The Sharafutinov projection, $\phi: T(S^2_1)/\mathbb Z_h\to S^2_1/\mathbb Z_h$ ($h\ge 3$), is a submetry, but not a fiber bundle map.
\end{example}

\begin{example} (A submetry with $Df$ satisfying (A1) and (A2), but $f: X\to Y$ is not weakly integrable). \label{1u}

\noindent (1.21.1) Let $f: M\to N$ be a Riemannian submersion. Then $Df: T_xM\to T_{f(x)}N$ is
a projection which is an integrable submetry. But a Riemann submersion in general is not
integrable. Moreover, a Riemannian submersion trivially satisfies (A2).

\noindent (1.21.2) (A non weakly integrable submetry satisfying the inverse of an extremal subset
is extremal, comparing with (A2)) Let $X$ denote the spherical suspension of a circle of radius $\frac 14$. Then $\mathbb Z_2$ acts isometrically on $X$ with two vertices, $o_i$, fixed. Then $f: X\to X/\mathbb Z_2$ is not a weakly integrable submetry, and $f^{-1}(\bar o_i)=o_i$.
\end{example}

\begin{example} \label{1v} (A fiber not a topological manifold in (E2)) Given a compact Riemannian manifold $S$ of nonnegative sectional curvature, let $X=S\times C(\mathbb RP^2)$ be the metric product. Then $(X,S,\text{proj})$ satisfies (E2), and a fiber, $\text{proj}^{-1}(\bar x)=C(\mathbb RP^2)$, is not a topological manifold.
\end{example}

\begin{example} \label{1w} (Counterexample to Theorem F without $X$ topologically nice) Let $F$ be the double of $([0,1]\times \Bbb R_+)\cup_\partial \nabla/([0,1]\times \{0\}\sim \text{top side of the equilateral triangle $\nabla$ of size one})$. Then the metric product, $X=S^2_1\times F\in \text{Ale}^4(0)$ admits an isometric diagonal $\Bbb Z_2$-action
with two fixed point in $S^2_1\times o$. Then $X$ is topologically nice with an extremal subset $S^2_1\times o$ (the soul $S$), and $X/\Bbb Z_2\in \text{Alex}^4(0)$ with a soul of dimension $2$, but $(X/\Bbb Z_2,S^2_1/\Bbb Z_2,\phi)$ satisfies neither (F1) nor (F2).
\end{example}


\section {Proof of Theorem A}

\vskip4mm

Our proof of Theorem A is quite involved. We will prove (A1) in subsection 2.1, (A2) in subsection 2.2, (A3) in subsection 2.3, and
prove Proposition  \ref{0j} in subsection 2.4.

\subsection {Proof of (A1)}

In the proof of (A1), the main technical result is the following.

\begin{theorem} {\rm (Weakly integrable over $Y_s$ is integrable)} \label{2a} Let $X\in \text{Alex}^n(\kappa), Y\in \text{Alex}^m(\kappa)$, and let $f: X\to Y$ be a weakly integrable submetry. Then $f$ is integrable over $Y_s$, that is, for any $x\in  f^{-1}(Y_s)$, $H_x=\Sigma_xW_x$.
\end{theorem}

Recall that if $\gamma(t)$ is a finite minimal geodesic in $Y$ such that its two ends are regular, then all points in $\gamma(t)$ points are regular. The set of all regular points in
$Y$, $\mathcal R(Y)\subset Y_s$, and the closure, $\overline{\mathcal R(Y)}=Y$.
In our proof of Theorem A, in order to apply Theorem \ref{2a} to an open subset, $U\cap Y_s$,
$U$ may not be convex (e.g., a canonical neighborhood of a weakly $k$-strained point),
to have the above property, we need the condition that regular points in the closure of $U\cap Y_s$ that are away from boundary are all contained in $U\cap Y_s$.

\begin{lemma} {\rm (Properties of integrable submetry)} \label{2b} Let $X\in \text{Alex}^n(\kappa), Y\in \text{Alex}^m(\kappa)$, and let $f: X\to Y$ be a weakly integrable submetry. Assume that $f$ is integrable over an open and path connected subset $U\subset Y$.

\noindent {\rm (2.2.1)} {\rm (Canonical local trivialization of a fiber bundle)} For $x\in f^{-1}(U)$, there is a (maximal) connected subset $U_x\subset f^{-1}(U)$ such that $f: U_x\to U_{\bar x}=U$ is a local isometry, thus a covering map. If $U$ is simply connected,
then for $x\ne x'\in f^{-1}(\bar x)$, $U_x\cap U_{x'}=\emptyset$, thus
$f: f^{-1}(U_{\bar x})\to U_{\bar x}$ is a trivial fiber bundle with the canonical trivialization map, $\phi_{\bar x}: f^{-1}(U_{\bar x})\to f^{-1}(\bar x)\times U_{\bar x}$, $\phi_{\bar x}(z)=(f^{-1}(\bar x)\cap U_x,f(z))$. 

\noindent {\rm (2.2.2)} {\rm (Extension of a homeomorphism and local isometry)} Assume that there is $\bar z\in Y$ such that $x\in W_z$, $W_{\bar z}\cup U_{\bar x}$ is simply connected, $W_{\bar z}\subset \overset{\vee}U=\bar U_{\bar x}\setminus \partial \bar U_{\bar x}$ (i.e., closure without boundary points) whose regular points are in $U_{\bar x}$. Then $f: \overset{\vee}U_x\to \overset{\vee}U_{\bar x}$ is a homeomorphism and local isometry.
\end{lemma}

\begin{proof} (2.2.1) For $\bar x\in U_{\bar x}$, $x\in f^{-1}(\bar x)$, $x'\in W_x$, we may assume that $f(W_x),f(W_{x'})\subset U_{\bar x}$. Then $f: W_x\cap W_{x'}\to W_{\bar x}\cap W_{\bar x'}$ is an isometry, because for $z\in W_x$ and $f(z)\in W_{\bar x}\cap W_{\bar x'}$, $\Sigma_zW_x=H_z=\Sigma_zW_{x'}$. The local compactness implies that $f: W_x\cup W_{x'}\to W_{\bar x}\cup W_{\bar x'}$ is a local isometry. Because $U_{\bar x}$ is path connected, by the above
gluing process one gets a maximal subset at $x$, $U_x\subset f^{-1}(U_{\bar x})$, such that $f: U_x\to U_{\bar x}$ is a local isometry, thus a covering map. We will call $U_x$ the horizontal lifting of $U_{\bar x}$ at $x$; clearly if $z\in U_x$, then $U_z=U_x$. 
If $U$ is simply connected, then the map, $\phi_{\bar x}: f^{-1}(U)\to  f^{-1}(\bar x)\times U$, $\phi_{\bar x}(z)=(U_z\cap f^{-1}(\bar x), f(z))$, defines a trivial bundle, referred as a canonical trivialization. 

(2.2.2) We first show that $f: W_z\cup U_x\to W_{\bar z}\cup U_{\bar x}$ is a homeomorphism and local isometry. First, regular points in $W_{\bar z}$ are of full measure thus are dense. Secondly, $\overset{\vee} U_{\bar x}\supset W_{\bar z}$ and that all regular points in
$\overset{\vee}U_{\bar x}$ are contained in $U_{\bar x}$ implies that any two regular points in
$W_{\bar z}$, a minimal geodesic connecting the two points is in $W_{\bar z}$, thus in
$U_{\bar x}$ (otherwise, $\overset{\vee}U_{\bar x}$ contains regular points not in $U_{\bar x}$.
Consequently, $f: W_z\cup U_x\to W_{\bar z}\cup U_{\bar x}$ is a local isometry. Because
$W_{\bar z}\cup U_{\bar x}$ is simply connected, $f: W_z\cup U_x\to W_{\bar z}\cup U_{\bar x}$
is a homeomorphism and local isometry.

To see that $f: \overset{\vee}U_x\to \overset{\vee} U_{\bar x}$ remains a homeomorphism and local isometry,
for $z, z'\in \overset{\vee}U_x\setminus U_x$, $0<|zz'|<<1$, which are away from $\partial U_x$, let $z_i, z_i\in U_x$, $z_i\to z$ and $z_i'\to z'$, such that $f(z_i), f(z_i)$ are regular points in $U_{\bar x}$.
Because a minimal geodesic from $f(z_i)$ and $f(z_i')$ contains in $U_{\bar x}$ (otherwise,
$\overset{\vee}U_{\bar x}$ contains regular points that are not in $U_{\bar x}$). Because $f: W_z\cup U_x\to W_{\bar z}\cap U_{\bar x}$ is homeomorphic, $|z_iz_i'|=|f(z_i)f(z_i')|$, thus $|zz'|=|f(z)f(z')|$ i.e., $f$ is a homeomorphism and local isometry.
\end{proof}

\begin {proof} [Proof of (A1) by assuming Theorem \ref{2a}] For $x\in X$ and $\lambda_i\to \infty$, $Df_x$ is defined by the following commutative diagrams:
\[\xymatrix{
 (\lambda_iX,x) \ar[r]^{\text{GH}} \ar[d]_{\lambda_i f} & (C_xX,o)  \ar[d] ^{Df_x}  &
 (\lambda_iW_x,x)\ar[r] ^{\text{GH}} \ar[d]^{\lambda_i f} &(C_x(\Sigma_xW_x),o) \ar[d]^{ Df_x}\\
(\lambda_iY,\bar x) \ar[r] ^{\text{GH}} & (C_{\bar x}Y,\bar o),   &
 (\lambda_iW_{\bar x},\bar x)\ar[r] ^{\text{GH}} & (C_{\bar x}(\Sigma_{\bar x}Y),\bar o) }
\]
For $u\in \Sigma_xW_x\subset H_x$, $Df_x: \Sigma_xW_x\to \Sigma_{\bar x} Y$ is an isometry i.e., $Df_x: H_x\to \Sigma_{\bar x}Y$ is global weakly integrable at $u$.

For $u\in H_x\setminus \Sigma_xW_x$, let $x_i\in \lambda_iX$, $x_i\to u$. Because $\lambda_iW_{x_i}$ may converge to $u$, our approach is to enlarge $W_{x_i}$ to $\overset{\vee}U_{x_i}$ such that $f: \overset{\vee} U_{x_i}\to \overset{\vee}U_{\bar x_i}$ is a homeomorphism and local isometry and $\overset{\vee} U_{x_i}\supset B_\rho(x_i,\lambda_iX)$ for some $\rho>0$ independent of $i$. Consequently, $\overset{\vee}U_{x_i}\overset{\text{GH}}\longrightarrow \overset{\vee} U_u\supset B_\rho(u,C_xX)$, $\overset{\vee}U_{\bar x_i}\to \overset{\vee}U_{\bar u}$, such that $Df_x: \overset{\vee}U_u\to \overset{\vee}U_{\bar u}\subset C(\Sigma_{\bar x}Y)$ is an isometry, which implies that $\overset{\vee} U_u\cap H_x\supset B_\rho(u,H_x)$. By replacing $\overset{\vee} U_{\bar u}$ (resp. $\overset{\vee}U_u$) with a convex neighborhood $W_{\bar u}$ of $\bar u$ (resp. $W_u=f^{-1}(W_{\bar u})\cap \overset{\vee} U_u$), one concludes that $Df_x: H_x\to \Sigma_{\bar x}Y$ is weakly integrable at $u$.

By Corollary \ref{1e}, we may assume $\rho>0$ such that $B_\rho(Df_x(u))$ is contractible and
for $i$ large a homeomorphism, $\psi_i: B_\rho(\bar x_i,\lambda_iY)\to B_\rho(\bar u,C_{\bar x}Y)$ such that $\psi_i(\bar x_i)=\bar u$ and $\psi_i$ maps an extremal subset to an extremal subset. Because
$B_\rho(\bar u,C_{\bar x}Y_s)$ radially contracts into any small neighborhood of $\bar u$, for $i$ large we may assume that $U_{\bar x_i}=\lambda_iW_{\bar x_i}\cup [B_\rho(\bar x_i,\lambda_iY)\cap \lambda_i Y_s]$ is simply connected. 
For any $z_i\in W_{x_i}$ such that $f(z_i)\in U_{\bar x_i}\cap \lambda_iY_s$, then $U_{x_i}$ (i.e., the horizontal lifting of $U_{\bar x_i}$ at $z_i$)
satisfies (2.2.2): $\lambda_if: U_{x_i}\to U_{\bar x_i}$ is a homeomorphism and local isometry, thus $\lambda_i f: \overset{\vee} U_{x_i}\to \hat U_{\bar x_i}=B_\rho(\bar x_i,\lambda_iY)$ is
a homeomorphism and local isometry. By now we complete the proof by enlarging $W_{x_i}$ to $\overset{\vee} U_{x_i}$ as desired.
\end{proof}

It turns out that a proof of Theorem \ref{2a} is quite involved, technical and tedious. We will
divide the proof of Theorem \ref{2a} in the following four lemmas (Lemmas \ref{2c}-\ref{2f}).

For $\Sigma_0\in \text{Alex}^m(1)$, we say that $\Sigma_0$ is $(k,\frac \pi2)$-separate,
if there are $u_1,...,u_k$ in $\Sigma_0$ such that $|u_iu_j|>\frac \pi2$, $1\le i\ne j\le k$. If $\bar x\in Y$ is weakly $k$-strained, then $\Sigma_{\bar x}Y$ is $(k+1,\frac \pi2)$-separate with $(k+1)$ points,
$$|\uparrow_{\bar x}^{\bar p_i}\uparrow_{\bar x}^{\bar p_j}|>\frac \pi2,\, i=1,...,k, \qquad  |\uparrow_{\bar x}^{\bar p_i}\uparrow_{\bar x}^{\bar w}|>\frac \pi2.$$

\begin{lemma}\label{2c} Let $\Sigma\in \text{Alex}^n(1), \Sigma_0\in \text{Alex}^m(1)$, $\partial \Sigma_0=\emptyset$, and let $h: \Sigma\to \Sigma_0$ be a submetry.

\noindent {\rm (2.3.1)} If $\Sigma_0$ is $(k,\frac \pi2)$-separate, $k\ge m-1$, then $n=m$.

\noindent {\rm (2.3.2)} If $h$ is global weakly integrable at one point, then $h$ is an isometry.
\end{lemma}

A point $\bar x\in Y$ ($\partial Y=\emptyset$) is called good, if for any $x\in f^{-1}(\bar x)$, $\dim(V_x)=n-m-1$. Because $V_x$ and $H_x$ are convex subsets of $\Sigma_xX$ which are $\frac \pi2$-apart, $\partial H_x=\emptyset$ (Lemma \ref{2h}), by (1.10.1) we have that $m-1\le \dim(H_x)\le n-2-(n-m-1)=m-1$, thus $\dim(H_x)=m-1=\dim(\Sigma_xW_x)$. Because $\partial ( \Sigma_xW_x)=\emptyset$, $H_x=\Sigma_xW_x$ (Lemma \ref{2i}) i.e., $f$ is integrable any $x\in f ^{-1}(\bar x)$. The above shows that $\bar x$ is good implies that $\bar x$ is integrable.

\begin{lemma} \label{2d} Let $X\in \text{Alex}^n(\kappa), Y\in \text{Alex}^m(\kappa)$, and let $f: X\to Y$ be a weakly integrable submetry. If any $\bar x\in Y\setminus \partial Y$ is weakly $k$-strained, $k\ge m-1$, then $\bar x$ is good (thus $\bar x$ is integrable).
\end{lemma}

We point it out that at least weakly $(m-1)$-strained point $\bar x\notin \partial Y$ guarantees that $\bar x$ has a convex neighborhood in which any two points can be
pushed back and forth via gradient flows of a finite number of distance functions (Lemma \ref{1a}); this property is crucial in our proof of Lemma \ref{2e} below.

By Theorem \ref{1f}, a weakly $k$-strained point $\bar x$ has a canonical neighborhood $K(h_{\bar x},g_{\bar x})$, defined by two functions, $(h_{\bar x},g_{\bar x}): K(h_{\bar x},g_{\bar x})\to \Bbb R^k\times \Bbb R$, such that $K(h_{\bar x},g_{\bar x})$
is homeomorphic to $\b B_\rho^k(0)\times C(\Sigma)$, where $g_{\bar x}^{-1}(0)$ is homeomorphic to $\b B_\rho^k(0)$ and $C(\Sigma)$ denotes a homeomorphic cone.

\begin{lemma} {\rm (Good points in $g_{\bar x}^{-1}(0)$)} \label{2e} Let $X\in \text{Alex}^n(\kappa), Y\in \text{Alex}^m(\kappa)$, and let $f: X\to Y$ be a weakly integrable submetry. Assume that if $\bar x\in Y\setminus \partial Y$ is at least weakly $(k+1)$-strained, then $\bar x$ is integrable. Let $\bar x\in Y\setminus \partial Y$ be weakly $k$-strained. If a canonical neighborhood $K(h_{\bar x},g_{\bar x})$ satisfies that $g^{-1}_{\bar x}(0)$ contains one good point, then all points in $g^{-1}_{\bar x}(0)$ are good.
\end{lemma}

\begin{lemma} {\rm (Local criterion of extremal subsets)} \label{2f} Let $Y\in \text{Alex}^m(\kappa)$, and let $F\subset Y$ be a subset consisting of weakly $k$-strained points such that for $\bar x\in F$ and any $(\bar p_1,...,\bar p_k,\bar w)$, the canonical
neighborhood determined by $(h_{\bar x},g_{\bar x}): K(h_{\bar x},g_{\bar x})\to \Bbb R^k\times \Bbb R$, $g_{\bar x}^{-1}(0)\subset F$. Assume that for any $\bar y\in \partial \bar F$ is at most weakly $(k-1)$-strained, then $\bar F$ is an extremal subset of $Y$.
\end{lemma}

Lemma \ref{2f} is false if one removes the condition that any $\bar y\in \partial \bar F$ is at most weakly $(k-1)$-strained; see
example below.

\begin{example} \label{2g} Let $Y=[0,1]^3\setminus \triangle_3$, where $[0,1]^3\subset \Bbb R^3$ is a
unit cube, $\triangle_3\subset [0,1]^3$ denotes a $3$-simplex of length $\frac 13$ with a vertex $(1,1,1)\in [0,1]^3$. Then $\bar Y\in \text{Alex}^3(0)$ with proper extremal subsets. Let $F=\{(1,t,1)\in Y,\,0<t<\frac 13\}$. Then $F$ satisfies the conditions of Lemma \ref{2f} for $k=1$, except that $(1,\frac 13,1)\in \partial \bar F$ is weakly $1$-strained, not weakly $0$-strained. Note that
$\bar F$ is not an extremal subset of $Y$.
\end{example}

\begin{proof} [Proof of Theorem \ref{2a} by assuming Lemmas \ref{2c}-\ref{2f}]

We shall show that any $\bar x\in Y_s$ is good, thus $\bar x$ is integrable i.e., for any $x\in f^{-1}(\bar x)$, $H_x=\Sigma_xW_x$.

Let $\bar x$ be weakly $k$-strained, $1\le k\le m$. We shall proceed the proof by inverse induction on $k$, starting with $k=m$. By Lemma \ref{2d}, $\bar x$ is good for  $k\ge m-1$.

Assume that for $j>k$, any weakly $j$-strained point in $Y_s$ is good.

Let $\bar x\in Y_s$ be weakly $k$-strained, and let $K(h_{\bar x},g_{\bar x})$ be a canonical neighborhood of $\bar x$. Observe that if $g_{\bar x}^{-1}(0)$ contains a weakly $j$-strained point $\bar z$ with $j>k$, then by the inductive assumption
$\bar z$ is good, thus by Lemma \ref{2e} all points in $g_{\bar x}^{-1}(0)$ are good.
Hence, we may assume that all point in $g_{\bar x}^{-1}(0)$ are weakly $k$-strained,
so it remains to show that $g_{\bar x}^{-1}(0)$ contains at least one good point.

Arguing by contradiction, assuming all points in $g_{\bar x}^{-1}(0)$ are weakly $k$-strained and non good. Let $F\ne \emptyset$ denote the subset of $Y_s$
consisting of weakly $k$-strained points, $\bar x$, such that any canonical neighborhood satisfies that all points in $g_{\bar x}^{-1}(0)$ are weakly $k$-strained and non-good.

Because $\bar x\in F$ implies that $g_{\bar x}^{-1}(0)\subset F$, $F$ is an open topological $k$-manifold, if $\bar x\in \partial \bar F$, then $\bar x\notin F$. We shall show that $\bar x\in \partial \bar F$ is at most weakly
$(k-1)$-strained point, thus by Lemma \ref{2f} $\bar F$ is an extremal subset of $Y$, and
therefore $\bar F\cap Y_s=\emptyset$, a contradiction to that $F\subset Y_s$.

We argue by contradiction, assuming $\bar x$ is at least weakly $k$-strained point. We claim that $g_{\bar x}^{-1}(0)\cap F=\emptyset$, thus there is $\bar z\in [K(h_{\bar x},g_{\bar x})\setminus g_{\bar x}^{-1}(0)]\cap F\ne \emptyset$. Because $\bar z$ is at least weakly $(k+1)$-strained point, a contradiction to that $\bar z\in F$ is weakly $k$-strained.

The claim is clear if $\bar x\in Y\setminus Y_s$ (an extremal subset), because $g_{\bar x}^{-1}(0)\subset Y\setminus Y_s$ (\cite{Ka2}). If $\bar x\in Y_s$, then $g_{\bar x}^{-1}(0)$ contains a good point;
otherwise by induction all points in $g_{\bar x}^{-1}(0)$ are weakly $k$-strained and are not good, thus $\bar x\in F$, a contradiction. By Lemma \ref{2e}, all points in $g_{\bar x}^{-1}(0)$ are good, thus $g_{\bar x}^{-1}(0)\cap F=\emptyset$ i.e., the claim holds.
\end{proof}

We now present proofs for Lemmas \ref{2c}-\ref{2f}. In the proofs of Lemmas \ref{2c} and \ref{2d}, we need the following properties in Lemmas \ref{2h} and \ref{2i}.

\begin{lemma}\label{2h} Let $\Sigma\in \text{Alex}^n(1)$, $\Sigma_0\in \text{Alex}^m(1)$, $\partial \Sigma_0=\emptyset$, and let $h:\Sigma\to \Sigma_0$ be a submetry. If $h$ is global weakly integrable at one point, then $\partial \Sigma=\emptyset$.
\end{lemma}

\begin{proof} Let $W\subset\Sigma$ be a convex subset such that $h|_W: W\to \Sigma_0$ is an isometry. Arguing by contradiction, assuming $\partial \Sigma\not=\emptyset$. Because $\Sigma\in \text{Alex}^n(1)$, $d_{\partial \Sigma}: \Sigma\to \Bbb R$ is concave and achieves the maximum at
unique point $p\in \Sigma$. Let $\phi_t$ denote the gradient flow of $\nabla d_{\partial \Sigma}$, $\phi_0=\text{id}_\Sigma$ and $\phi_1: \Sigma\to \{p\}$. Identifying $W$ with $h(W)=\Sigma_0$, the gradient flows on $\Sigma$ induces a flow on $\Sigma_0$, $h\circ \phi_t|_W: I\times \Sigma_0 \, (=h|_W^{-1})\to \Sigma_0$, is a homotopy equivalence of $\Sigma_0$ to a point, a contradiction because
$\partial \Sigma_0=\emptyset$.
\end{proof}

\begin{lemma}\label{2i} Let $X\in \text{Alex}^n(\kappa)$. Let $X_0\subseteq X$ be a compact convex subset.
If $X_0\in \text{Alex}^n(\kappa)$ and $\partial X_0=\emptyset$, then $X_0=X$.
\end{lemma}

\begin{proof} Arguing by contradiction, assuming $x\in X\setminus X_0$, let $x_0\in X_0$ such that $|xx_0|=|xX_0|$. Then a minimal geodesic from $x_0$ to $x$ satisfies that
$|\uparrow_{x_0}^x\Sigma_{x_0}(X_0)|\ge \frac \pi2$. Consequently, because $\partial X_0=\emptyset$ and thus $\partial \Sigma_{x_0}X_0=\emptyset$, $\dim(\Sigma_{x_0}X)>\dim(\Sigma_{x_0}(X_0))$ i.e., $\dim(X)>\dim(X_0)=n$, a contradiction.
\end{proof}

\begin{proof} [Proof of Lemma \ref{2c}]

(2.3.1) We will proceed the proof by induction on $m$, starting with $m=1$. Because $\partial \Sigma_0=\emptyset$, $\Sigma_0$ is isometric to a circle $S^1$, thus $h: \Sigma\to
\Sigma_0$ is a fiber bundle projection (\cite{Per2}). Let $\pi: \Bbb R^1\to \Sigma_0$ denote the universal covering, and let $\pi^*\Sigma\to \Bbb R^1$ be the $\pi$-pull back fiber bundle, $h: \Sigma\to \Sigma_0$; each component of $\pi^*\Sigma$ is a metric covering of $\Sigma$, a contradiction because each component splits off an $\Bbb R^1$-factor.

Assume that (2.3.1) holds for $m$, $k\ge 1$.

Consider $\Sigma_0$, $\dim(\Sigma_0)=m+1$ and $\Sigma_0$ is $(m,\frac \pi2)$-separated
i.e., there are
$\bar u_1,...,\bar u_m$ such that $|\bar u_i\bar u_j|>\frac
\pi2$. Let $S=\{\bar u_1,...,\bar u_m\}$, and let $\bar w\in \Sigma_0$ such that $d_{\bar S}$ achieves the maximum at $\bar w$. Because $h: \Sigma\to \Sigma_0$ is a submetry, $d_{\tilde S}$ achieves a maximum at $w\in \Sigma$, where $\tilde S=\bigcup_{i=1}^mh^{-1}(\bar u_i)$. By a standard triangle comparison, one sees that $d_{\tilde S}$ is strictly decreasing along any direction at $w$, thus there is a small neighborhood $U\ni w$ such that $h^{-1}(\bar w)\cap U=\{w\}$, therefore $H_w=\Sigma_w\Sigma$.

Consider the submetry,  $Df_w: H_w\Sigma\to \Sigma_{\bar w}\Sigma_0$. By a standard triangle comparison, it is easy to see points in $\Sigma_{\bar w}\Sigma_0$, $\{\uparrow_{\bar w}^{\bar u_1},...,\uparrow_{\bar w}^{\bar u_m}\}$, are pair-wisely $\frac \pi2$-separated. Hence,
applying the inductive assumption to $Dh_w: H_w\to \Sigma_{\bar w}\Sigma_0$ we conclude
that $n-1=m$ i.e., $n=m+1$.

(2.3.2) Assume $u\in \Sigma$ and a convex subset $W_u\subset \Sigma$ such that $h: W_u\to \Sigma_0$ is an isometry. Because $\dim(\Sigma)=\dim(\Sigma_0)=\dim(W_u)$, and $\partial W_u=\emptyset$ (Lemma \ref{2h}), $\Sigma=W_u$ (Lemma \ref{2i}).
\end{proof}

\begin{proof} [Proof of Lemma \ref{2d}]

Let $\bar x\in Y\setminus\partial Y$ be weakly $k$-strained, $k\ge m-1$. Let $f: W_x\to W_{\bar x}$ be an isometry.

We first show that $\bar x$ is integrable. Because $Df_x: \Sigma_xW_x\to \Sigma_{\bar x}Y$ is an isometry, the submetry, $Df_x: H_x\to \Sigma_{\bar x}Y$, is global integrable at any $v\in \Sigma_xW_x$. Because $\bar x$ is weakly $k$-strained, $k\ge m-1$, $\Sigma_{\bar x}Y$ is
$(k+1,\frac \pi2)$-separate, thus by Lemma \ref{2c}, $\dim(H_x)=\dim(\Sigma_{\bar x}Y)=\dim(\Sigma_xW_x)$. Because $\partial H_x=\emptyset$ (Lemma \ref{2h}), $H_x=\Sigma_xW_x$ (Lemma \ref{2i}).

Let $\bar z\in W_{\bar x}$ be a regular point i.e., $C_{\bar x}Y\cong \Bbb R^m$. Then for any $z\in W_x\cap f^{-1}(\bar z)$, $C_zX=\Bbb R^m\times C(f^{-1}(\bar z))$ (\cite{KL}). Because $C(f^{-1}(\bar z))=C(V_z)$, $\dim(C(V_z))=n-m$, thus $\dim(V_z)=n-m-1$ i.e., $\bar z$ is good.

We now show that $\bar x$ is good. Consider gradient flows using distance functions, $d_{\bar p_1},...,d_{\bar p_k}$, one pushes $\bar x$ to $\bar z$, one pushes $\bar z$ back to $\bar x$ inside $W_{\bar x}$. Via horizontal lifting, each gradient flows defines a locally Lipschitz map, $\Phi: f^{-1}(\bar x)\to f^{-1}(\bar z)$, and $\Psi: f^{-1}(\bar z)\to f^{-1}(\bar x)$. Because $f: f^{-1}(W_{\bar x})\to W_{\bar x}$ is integrable and thus global integrable ($W_{\bar x}$ is
contractible), $\Phi\circ \Psi=\text{id}_{f^{-1}(\bar z)}$. Consequently, $\Phi$ and $\Psi$ are non-degenerate homeomorphisms, thus $\dim(V_x)=\dim(V_z)=n-m-1$ i.e., $\bar x$ is good.
\end{proof}

\begin{proof} [Proof of Lemma \ref{2e}]
Let $\bar x\in Y\setminus\partial Y$ be weakly $k$-strained (by Lemma \ref{2d}, $k+1\le m-1$). Without loss of generality, we may assume that $\bar x$ is good, and we will show that any $\bar z\in g_{\bar x}^{-1}(0)$ is good i.e., for $z\in f^{-1}(\bar z)$, $\dim(V_z)=n-m-1$.

Our approach is similar to the one seen at the end of proof of Lemma \ref{2d}; we will construct a horizontal lifting of $U_{\bar z}=W_{\bar z}\cup [Z_{\bar x}\setminus g_{\bar x}^{-1}(0)]$ at $z$, $U_z$, where canonical neighborhood $Z_{\bar x}=K(h_{\bar x},g_{\bar x})$ is homeomorphic to $\b B_\rho^k(0)\times C(\Sigma)$, $\b B_\rho^k(0)$ is homeomorphic to $g_{\bar x}^{-1}(0)$, thus
$W_{\bar z}\cup [Z_{\bar x}\setminus g_{\bar x}^{-1}(0)]$ is simply connected.

To guarantee that $U_{\bar z}$ contains all regular points in $\overset{\vee} U_{\bar z}$, 
we shall extend $U_{\bar z}$ to include the subset of $g_{\bar x}^{-1}(0)$ consisting of (possible) weakly $m$-strained points, still denoted by $U_{\bar z}$.
By (2.2.2), we conclude that $f: \overset{\vee} U_z\to \overset{\vee} U_{\bar z}=Z_{\bar x}$ is a homeomorphism and local isometry.
Similar to Lemma \ref{1a}, for any $\bar z\in g_{\bar x}^{-1}(0)$, there are gradient flows of finite number of distance functions at points ($\ne \bar z$) that flows $\bar x$ to $\bar z$ in $Z_{\bar x}$ and vice versa, whose horizontal lifting in $\overset{\vee} U_z$ form a closed curve, thus induces a locally Lipschitz map, $\Phi: f^{-1}(\bar x)\to f^{-1}(\bar z)$
and $\Psi: f^{-1}(\bar z)\to f^{-1}(\bar x)$ such that $\Psi\circ \Phi=\text{id}_{f^{-1}(\bar x)}$.
Consequently, $\Phi$ and $\Psi$ are non-degenerate homeomorphisms, thus $\dim(V_z)=\dim(V_x)=n-m-1$ i.e., $\bar z$ is good.

We now construct a horizontal lifting of $W_{\bar z}\cup[Z_{\bar x}\setminus g_{\bar x}^{-1}(0)]$:
let $z'\in W_z$ such that $f(z')\in Z_{\bar x}\setminus g_{\bar x}^{-1}(0)$. Because any point in $Z_{\bar x}\setminus g_{\bar x}^{-1}(0)$ is at least weakly $(k+1)$-strained, which, by the inductive assumption, is integrable, as in the proof of (2.2.1), $Z_{\bar x}\setminus g_{\bar x}^{-1}(0)$ has a (maximal) horizontal lifting at $z'$, $f: U_{z'}\to Z_{\bar x}\setminus g_{\bar x}^{-1}(0)$ is a local isometry, thus $f: U_z=W_z\cup U_{z'}\to W_{\bar z}\cup [Z_{\bar x}\setminus g_{\bar x}^{-1}(0)]$ is a local isometry, thus a homeomorphism because $W_{\bar z}\cup [Z_{\bar x}\setminus g_{\bar x}^{-1}(0)]$ is  simply connected.
\end{proof}

Lemma \ref{2f} is a new local criterion for an extremal subset formulated for our proof of Theorem \ref{2a}. Indeed, the proof of Lemma \ref{2f} replies on the following local criterion of extremal subsets, which is a strengthened version of Lemma 5.1 in \cite{Fuj4} (cf. \cite{Per4}, \cite{Fuj2}; where a criterion requires a verification for any $(h_{\bar x},g_{\bar x}): K(h_{\bar x},g_{\bar x})\to \Bbb R^\ell\times \Bbb R$, $\ell\le k$).

\begin{lemma} \label{2j} Let $F\subset Y\in \text{Alex}^m(\kappa)$ be a subset, $\bar x\in \bar F$ be a weakly $k$-strained ($1\le k\le m$). Assume for $(\bar p_1,...,\bar p_k,\bar w)$, $(h_{\bar x},g_{\bar x}): K(h_{\bar x},g_{\bar x})\to \Bbb R^k\times \Bbb R$
defines a canonical neighborhood. If $g_{\bar x}^{-1}(0)\subset \bar F$ (for any $\bar x$, $K(h_{\bar x},g_{\bar x})$), then $\bar F$ is an extremal subset of $Y$.
\end{lemma}

In our proof of Lemma \ref{2f}, we shall also apply the following local version of Theorem \ref{1g}: let $\bar x\in Y$ be a weakly $k$-strained point, $\bar y\in Z_{\bar x}\setminus g_{\bar x}^{-1}(0)$, and $C_{\bar y}$ a convex neighborhood such that $C_{\bar y}\subset Z_{\bar x}\setminus g_{\bar x}^{-1}(0)$. Because all points in $C_{\bar y}$ are at least weakly $(k+1)$-strained, applying Theorem \ref{1g} to $(C_{\bar y},F\cap C_{\bar y})$, one may assume that $\text{proj}_2\circ \psi=(h_{\bar x},g_{\bar x})$.

\begin{proof} [Proof of Lemma \ref{2f}]

We shall show that $\bar F$ satisfies the conditions of Lemma \ref{2j} i.e.,
any $\bar x\in \bar F$, $g_{\bar x}^{-1}(0)\subset \bar F$. By definition, for $\bar x\in F$, $g_{\bar x}^{-1}(0)\subset F$, thus we may assume that $\bar x\in \bar F\setminus F$.
Let $T=\{\bar z\in g^{-1}_{\bar x}(0),\,|\bar zF|=0\}\ni \bar x$. We shall show that $T$ is both open and closed subset of $g_{\bar x}^{-1}(0)$, thus $T=g_{\bar x}^{-1}(0)$
which implies $g_{\bar x}^{-1}(0)\subset \bar F$. Because $T$ is clearly closed, it remains to show that $T$ is open.

Let $\bar x$ be weakly $j$-strained, and we shall proceed the proof with inverse induction on $j$, starting with $j=k$, because if $j\ge k+1$,
then all points in $g_{\bar x}^{-1}(0)$ are at least weakly $(k+1)$-strained,
thus $g_{\bar x}^{-1}(0)\cap F=\emptyset$ and $\bar z\in \overset{\vee} Z_{\bar x}\cap F\ne \emptyset$,
a contradiction.

Let $\bar x$ is weakly $k$-strained point. Because $\bar x\notin \partial \bar F$ (by the assumption), if $g_{\bar x}^{-1}(0)$ is not contained in $\bar F$, then $\overset{\vee} Z_{\bar x}\cap F\ne\emptyset$, a contradiction.

Assume that for $\bar x$ at least $(j+1)$-strained point ($j+1\le k$), $g_{\bar x}^{-1}(0)\subset \bar F$. Consider a weakly $j$-strained point $\bar x\in \bar F\setminus F$, and we shall show that $T=g_{\bar x}^{-1}(0)$.

Arguing by contradiction, let $\bar z\in g_{\bar x}^{-1}(0)\setminus T$, $0<\rho=|\bar z\bar F|$,
thus for any $\bar y\in \bar F\setminus g_{\bar x}^{-1}(0)$, $|\bar z\bar y|>\frac 12|z\bar F|$.
Because $g_{\bar x}^{-1}(0)$ is an open $j$-manifold and $F$ is an open $k$-manifold ($j<k$),
we may assume for any $\bar y\in \bar F\setminus g_{\bar x}^{-1}(0)$, $|\bar z\bar y|>\frac 12\rho$.
We shall derive a contradiction by finding some $\bar y'\in \bar F\setminus g_{\bar x}^{-1}(0)$ such that $|\bar z\bar y'|<\frac 12\rho$.

Let $C_{\bar y}\subset Z_{\bar x}\setminus g_{\bar x}^{-1}(0)$ be a small convex neighborhood of $\bar y$. Because points in $\bar C_{\bar y}$ are at least weakly $(j+1)$-strained, applying the inductive assumption on $(\bar F\cap \bar C_{\bar y}, \bar C_{\bar y}\in \text{Alex}^m(\kappa))$ we conclude that $\bar F\cap C_{\bar y}$ is an extremal subset in $\bar C_{\bar y}$. By applying Theorem \ref{1g} to $(\bar C_{\bar y},\bar F\cap \bar C_{\bar y})$, we may assume a homeomorphism, $\psi: (C_{\bar y},\bar F\cap C_{\bar y})\to (A,B)\times \Bbb R^{j+1}$: $C_{\bar y}$ is a bundle over $\Bbb R^{j+1}$ with fiber $A$ homeomorphic to $C_{\bar y}\cap (\{s\}\times \Sigma)$, and $C_{\bar y}\cap \bar F$ is a bundle over $\Bbb R^{j+1}$ with fiber $B=(C_{\bar y}\cap \bar F)\cap (\{s\}\times \Sigma)$, $s=g_{\bar x}(\bar y)$.

Consider the map, $\tau_{\bar x}=(h_{\bar x},g_{\bar x}): Z_{\bar x}\to \b B_\rho^j(0)\times (-\eta,0]$; $\tau_{\bar x}(\bar y)=(u,t)\in \tau_{\bar x}(\bar F\cap Z_{\bar x})$. Observe that for any $t<s<0$, $(u,s)\in \tau_{\bar x}(\bar F)$, and $\tau_{\bar x}(C_{\bar y}\cap \bar F)$ is an open neighborhood of $(u,t)=\tau_{\bar x}(y)\in \tau_{\bar x}(\bar F)$. This openness implies that there is $\bar y'\in \partial (\bar F\cap \bar C_{\bar y})$ such that $\tau_{\bar x}(y')$ is closer to $\tau_{\bar x}(\bar z)$ than $\tau_{\bar x}(\bar y)\in \tau_{\bar x}(\bar F\cap Z_{\bar x})$. Iterating the above operation starting with $\bar y'$, after a finite number of steps one gets a desired $\bar y'$, which can be chosen from $\bar F\setminus g_{\bar x}^{-1}(0)$.
\end{proof}

\begin{proof} [Proof of Lemma \ref{2j}] Argue by contradiction, assuming that $\bar F$ is not an extremal subset i.e., there are points, $\bar q\in Y\setminus \bar F$ and $\bar p\in \bar F$, such that $|\bar q\bar p|=|\bar q\bar F|$ and $\nabla d_{\bar q}(\bar p)\ne 0$.

Assume that $\bar p$ is a weakly $k$-strained point ($1\le k\le m$), and points $(\bar p_1,...,\bar p_k,\bar w)$ such that $(h_{\bar p},g_{\bar p}): K(h_{\bar p},g_{\bar p})\to \Bbb R^k\times \Bbb R^1$ and $g_{\bar p}^{-1}(0)\subset \bar F$. 

To derive a contradiction, we define the `tangent space' of $\bar F$ at $\bar p$, $C(\Sigma_{\bar p}\bar F)$, where
$$\Sigma_{\bar p}\bar F=\{v=\lim_{i\to \infty} \uparrow_{\bar p}^{\bar x_i}, \, \bar x_i\in
\bar F,\, \bar x_i\to \bar p\}.$$
Then $|\uparrow_{\bar p}^{\bar q}v|\ge \frac \pi2$ for any $v\in \Sigma_{\bar p}\bar F$. Hence the proof essentially reduces to find a $(k+1,\frac \pi2)$-separate points, $\{v_j\}_{j=1}^{k+1}\subset \Sigma_{\bar p}\bar F$; because $\nabla d_{\bar q}(\bar p)\ne 0$ implies a $v_0\in \Sigma_{\bar p}Y$ such that the gradient-push by $d_{\uparrow_{\bar p}^{\bar q}}: \Sigma_{\bar p}Y\to \Bbb R$, in the direction $v_0$ yields a $(k+2,\frac \pi2)$-separate set, $\{\uparrow_{\bar p}^{\bar q},v_j'\}_{j=1}^{k+1}\subset \Sigma_{\bar p}Y$ ($v_j'$ is
obtained from $v_j$ along the gradient flow), a contradiction to that $\bar p$ is weakly $k$-strained point.

Here is an outline to convert $(k+1,\frac \pi2)$-separate subset, $\{\uparrow_{\bar p}^{\bar p_1},...,\uparrow_{\bar p}^{p_k},\uparrow_{\bar p}^{\bar w})\subset \Sigma_{\bar p}Y$ to
a $(k+1,\frac \pi2)$-separate subset $\{v_1,...,v_{k+1}\}\subset \Sigma_{\bar p}\bar F$:
let $\tilde \gamma(t)\subset g_{\bar p}^{-1}(0)$ denote the projection of
$\gamma(t)=[\bar p\bar w](t)$, a minimal geodesic from $\bar p$ to $\bar w$. For the simplicity of
expansion, let's assume that $\bar x_i=c_{\bar w}(t_i)$ such that $\uparrow_{\bar p}^{\bar x_i}\to v_{k+1}\in \Sigma_{\bar p}\bar F$.
We first show that $|v_{k+1}\uparrow_{\bar p}^{\bar p_j}|>\frac \pi2$, $1\le j\le k$ (Sublemma 2.11). Due to a lack of control in $|v_{k+1}\uparrow_{\bar p}^{\bar w}|$, there is no estimate for $(|v_{k+1}\uparrow_{\bar p}^{\bar p_j}|-\frac \pi2)$ that forces us to pick $v_1,...,v_k$, by the above argument on $(C_{\bar p}Y,\bar o,C(\Sigma_{\bar p}\bar F))$ (which requires a verification that
$g_{\bar o}^{-1}(0)\subset C(\Sigma_{\bar p}\bar F)$), starting with a $(k+1,\frac \pi2)$ separate subset, $(v_{k+1},\uparrow_{\bar p}^{\bar p_j},...,\uparrow_{\bar p}^{\bar p_k})\subset \Sigma_{\bar p}Y$, and $(h_{\bar o},g_{\bar o}): K(h_{\bar o},g_{\bar o})\to \Bbb R^k\times \Bbb R$, following the above argument we replace $\uparrow_{\bar p}^{\bar p_k}$ with
$v_k\in \Sigma_{\bar p}\bar F$ such that $(v_{k+1},v_k,\uparrow_{\bar p}^{\bar p_1},...,\uparrow_{\bar p}^{p_{k-1}})$ is $(k+1,\frac \pi2)$-separate subset. Iterating
this process, we achieve the desired goal.

First we will show that

\begin{sub lemma}\label{one} There are $\bar x_i\in g_{\bar p}^{-1}(0)$, $\uparrow_{\bar p}^{\bar x_i}\to v_{k+1}\in \Sigma_{\bar p}\bar F$, such that $(v_{k+1},\uparrow _{\bar p}^{\bar p_1},...,\uparrow_{\bar p}^{\bar p_k})$ is $(k+1,\frac \pi2)$-separate in $C(\Sigma_{\bar p}Y)$.
\end{sub lemma}
In order to prove the above sublemma, we need to recall the construction of $g^{-1}_{\bar p}(0)$ and show some properties of it.

Note that by \cite{Ka2}, \cite{Per2}, the function $g_{\bar p}$ is constructed from a strictly concave function $\tilde g_{\bar p}$.
In the following, first we recall the construction of $\tilde g_{\bar p}$.

Let $(h_{\bar p},g_{\bar p}): K(h_{\bar p},g_{\bar p})\to \mathbb R^k\times \Bbb R$. Let $\tilde g_{\bar p}=
\frac  1 N\Sigma_\alpha g_{\alpha}=g_{\bar p}+\max\{\frac  1 N\Sigma_\alpha g_{\alpha}|_{h_{\bar p}^{-1}\circ h_{\bar p}}\}$, where $\{\bar q_\alpha\}_{\alpha=1}^N$
denotes a $\delta$-net in $\partial B_\epsilon(\bar w)$ used in a construction of
$g_{\bar p}$, $g_{\alpha}= \phi_{\alpha}(d_{\bar q_{\alpha}})$ with $g_{\alpha}(\bar p)=0$, where $2\epsilon=\min \{|\uparrow_{\bar p}^{\bar p_i}\uparrow_{\bar p}^{\bar p_j}|, |\uparrow_{\bar p}^{\bar p_i}\uparrow_{\bar p}^{\bar w}|, 1\le i\ne j\le k\}-\frac \pi2$. Because $\bar p$ is weakly $k$-strained, if $v\in \Sigma_{\bar p}Y$ satisfies that $D_{\bar p}d_{\bar p_j}(v)\ge 0$, $1\le j\le k$, $|v\Uparrow_{\bar p}^{\bar q_{\alpha}}|\leq \frac \pi 2$ i.e., $D_{\bar p}(g_{\alpha})(v)=-\phi_{\alpha}'(|\bar p\bar q_{\alpha}| )<\Uparrow_{\bar p}^{\bar q_{\alpha}},v> \leq 0$. According to \cite{Ka2}, \cite{Per2}, there is $\alpha$ such that $|v\Uparrow_{\bar p}^{\bar q_{\alpha}}|< \frac \pi 2$
(thus $D_{\bar p}g_{\bar p}(v)<0$). Then

\begin{equation}
\max D_{\bar p}\tilde g_{\bar p}|_{\{\bigcap_{j=1}^k\{v,\, D_{\bar p}d_{\bar p_j}(v)\geq 0\}}=D_{\bar p}\tilde g_{\bar p}(\bar o), \label{1}
\end{equation}
and $\bar o$ is vertex of $C(\Sigma_{\bar p}Y)$ (also the unique maximal point for $D_{\bar p}\tilde g_{\bar p}$).

\begin{proof}[Proof of Sublemma \ref{one}]

Assume that $t_i\to 0$, $\uparrow_{\bar p}^{\tilde \gamma(t_i)}\to v_{k+1}\in \Sigma_{\bar p}\bar F$, and we need to show that $|v_{k+1}\uparrow_{\bar p}^{\bar p_j}|>\frac \pi 2, 1\leq j\leq k$, equivalently, $D_{\bar p}d_{\bar p_j}(v) >0,1\leq j\leq k$. Observe the following two limits,
$$D_{\bar p}d_{\bar p_j}(v_{k+1}) =\lim_{i\to\infty}\dfrac{d_{\bar p_j}(\tilde \gamma(t_i))-d_{\bar p_j}(\bar p)} {|\bar p\tilde \gamma(t_i))|}$$
and
$$D_{\bar p}d_{\bar p_j}(\uparrow_{\bar p}^{\bar w})=\lim_{i\to \infty}\dfrac {d_{\bar p_j}(\gamma(t_i))-d_{\bar p_j}(\bar p)}{|\bar p\gamma(t_i)|}>0,$$
where $\tilde \gamma(t_i)$ is the projection of $\gamma(t_i)$, thus $h_{\bar p}(\tilde \gamma(t_i))=h_{\bar p}(\gamma(t_i))$ and $\tilde g_{\bar p}(\tilde \gamma(t_i))>\tilde g_{\bar p}(\gamma(t_i))$. Because the numerators in the above limits are equal, $D_{\bar p}d_{\bar p_j}(v)>0$ follows from that $D_{\bar p}d_{\bar p_j}(\uparrow_{\bar p}^{\bar w})>0$, if the denominators of the above limits are proportional i.e.,
$$\bar \lim\dfrac {|\bar p\tilde \gamma(t_i)|}{|\bar p \gamma(t_i) |}=c<\infty.$$
If $c=\infty$, let $\lambda_i=|\bar p\tilde \gamma(t_i)|$, and consider the sequence of
functions: $\lambda_i \tilde g_{\bar p}=\frac 1N\sum_\alpha \lambda_ig_\alpha: (\lambda_iK(h_{\bar p},g_{\bar p})\subset (\lambda_iY, \bar p)\to \Bbb R$; taking limit we may obtain that $D_{\bar p}\tilde g_{\bar p}: (C(\Sigma_{\bar p}),\bar o)\to \Bbb R$ and that $\gamma(t_i)\to z_{\infty}=\bar o$, and $\tilde \gamma(t_i)\to \tilde z_{\infty}=v_{k+1}$. Then $D_{\bar p}\tilde g_{\bar p}(\bar o)=0$. Because $\tilde g_{\bar p}(\gamma(t_i))<\tilde g_{\bar p}(\tilde \gamma(t_i))$, $v_{k+1}\in \bigcap_{j=1}^k\{v,\,D_{\bar p}d_{\bar p_j}(v)\geq 0\}$ on which $D_{\bar p}\tilde g_{\bar p}$ achieves unique maximum at $\bar o$, we derive
$D_{\bar p}\tilde g_{\bar p}(\bar o)\le D_{\bar p}\tilde g_{\bar p}(v_{k+1})<0$, a contradiction.
\end{proof}

We now continue the proof of Lemma 2.10. By Sublemma 2.11, we obtain $(k+1,\frac \pi2)$-separate points, $(v_{k+1},\uparrow_{\bar p}^{\bar p_1},...,\uparrow_{\bar p}^{\bar p_k}\}\subset \Sigma_{\bar p}Y$ i.e., $\bar o$ is weakly $k$-strained point, and we may assume $(h_{\bar o},g_{\bar o}): K(h_{\bar o},g_{\bar o})\to \Bbb R^k\times \Bbb R$. For our purpose, we need that $g_{\bar o}^{-1}(0)\subset C(\Sigma_{\bar p}\bar F)$ (to guarantee that repeating the above argument yields $v_k\in \Sigma_{\bar p}\bar F$).
Viewing $C(\Sigma_{\bar p}Y)$ as a blow-up limit, $(\lambda_iY,\bar p)\to (C(\Sigma_{\bar p}Y),\bar o)$, for $i$ large consider the lifting distance functions, $d_{k+1}, d_{j,i}, 1\le j\le k-1$, of $d_{v_{k+1}},...,d_{\uparrow_{\bar p}^{\bar p_1}},..., d_{\uparrow_{\bar p}^{\bar p_{k-1}}}$. By Lemma 6.14 in \cite{Ka2}, $h_{\bar p,i}$ and $g_{\bar p,i}$ are liftings of $\bar h_{\bar o}$ and $g_{\bar o}$, $K(h_{\bar p,i},g_{\bar p,i})\to K(h_{\bar o},g_{\bar o})$, and $g_{\bar p,i}^{-1}(0)\to g_{\bar o}^{-1}(0)$.
By the conditions of Lemma 2.10, $g_{\bar p,i}^{-1}(0)\subset \bar F$, then $g_{\bar o}^{-1}(0)\subset C(\Sigma_{\bar p}\bar F)$.
\end{proof}

\subsection {Proof of (A2)}

We first establish the following criterion of (A2) (comparing Example \ref{1r}, using which
we shall verify (A2) by induction on $\dim(Y)$.

\begin{lemma} [Criterion for $f^{-1}(E)$ extremal] \label{2k} Let $X\in \text{Alex}^n(\kappa), Y\in \text{Alex}^m(\kappa)$ and $\partial Y=\emptyset$,
and let a submetry $f: X\to Y$ be weakly integrable. If $E$ is an extremal subset of $Y$, then $f^{-1}(E)$ is
an extremal subset of $X$ if for $x\in f^{-1}(E)$,
\begin{equation*}(\Sigma_{x}f^{-1}
(E))\cap H_x \text{ is an extremal subset of $H_x$}.
\end{equation*}
\end{lemma}

Observe that $(\Sigma_{x}f^{-1}(E))\cap H_x=\emptyset$ if and only of $E=\{\bar p\}$; and it is
convention that the above condition automatically holds (see the proof of Lemma \ref{2k}.)

\begin{lemma}\label{2l} Let $\Sigma\in \text{Alex}^n(1)$, $\Sigma_0\in \text{Alex}^m(1)$, $\partial \Sigma_0=\emptyset$, and let $h:\Sigma\to \Sigma_0$ be a submetry.
If $h$ is global weakly integrable at one point, then that $\text{diam}(\Sigma_0)\le \frac \pi2$ implies that $\text{diam}(\Sigma)\le \frac \pi2$.
\end{lemma}

\begin{proof} Arguing by contradiction, assuming $p, q\in \Sigma$ such that $\text{diam}(\Sigma)=|pq| > \frac \pi 2$. Because $f: \Sigma\to \Sigma_0$ is global weakly integrable at some $x_0$, $h: W_{x_0}\to \Sigma_0$ is an isometry. Because $\text{diam}(W_{x_0})=\text{diam}(\Sigma_0)<\frac \pi2$, without loss of generality we may
assume that $p\notin W_{x_0}$. Let $\phi_t$ denote the gradient flow of $d_p$ on $\Sigma\setminus \{p\}$,
$\phi_0=\text{id}_{\Sigma\setminus \{p\}}$ and $\phi_1(\Sigma\setminus \{p\})=q$. Then
$h\circ \phi_t: [0,1]\times W_{x_0}\to \Sigma_0$ is a homotopy deformation of $W_{x_0}\cong \Sigma_0$ to a point, a contradiction to that $\partial \Sigma_0=\emptyset$.
\end{proof}

Note that Lemma \ref{2l} is false if one removes the condition that $h$ is global weakly integrable at
one point (e.g., $f: \Sigma\to \Sigma_0$ is a covering map).

\begin{proof} [Proof of Lemma \ref{2k}] For $\lambda_k\to \infty$, assume the following commutative diagram:
\[\xymatrix{ (\lambda_kX,\lambda_kf^{-1}(E),x)\ar[r]^{\text{GH}}   \ar [d]^f & (C_{x}X,C_xf^{-1}(E),o) \ar[d]^{D_xf}\\ (\lambda_kY,\lambda _k E,\bar x)\ar [r]^{\text{GH}} & (C_{\bar x}Y,C_{\bar x}E,\bar o)}\]
Then $C_xf^{-1}(E)=Df_x^{-1}(C_xE)$.

For any $y\in X\backslash f^{-1}(E)$, let $x\in f^{-1}(E)$ be a minimum for $d_y$ i.e. $d_y(x)=|y f^{-1}(E)|$, thus $\uparrow_x^y\in H_x$. We shall show that for any $\xi\in \Sigma_x X,|\xi \uparrow_x^y|\le \frac \pi 2$ i.e. $x$ is a critical point of $d_y$ on $X$, thus $f^{-1}(E)$ is an extremal subset of $X$.

If $E=\{\bar x\}$, by Lemma \ref{2l}, $\text{diam}(H_x)\le \frac \pi2$. Note that $\Sigma_xX=[H_xV_x]$ (\cite{KL}). We claim that $H_x$ and $V_x$ are $\frac \pi2$-apart. Assuming the claim, via the triangle comparison one sees that for any $\xi \in \Sigma_x X$, $|\xi \uparrow_x^y|\leq \frac \pi 2$ i.e., $x$ is a critical point of $d_y$.

The claim can be seen as follows: because $\partial Y=\emptyset$, $\partial H_x=\emptyset$ (Lemma \ref{2h}), the claim follows from Lemma \ref{8a}.

Note that $(\Sigma_xf^{-1}(E))\cap H_x=\{v\}$ if $\dim(E)=1$ and $\bar x\in \partial E$.
Thus $\{\bar x\}$ is an extremal subset of $Y$, and therefore $\text{diam}(\Sigma_{\bar x}Y)\le \frac \pi2$. By Lemma \ref{2l}, $\text{diam}(H_x)\le \frac \pi2$. By now we apply the above argument to
conclude that $x$ is a critical point of $d_y$.

In the following, we may assume that $(\Sigma_xf^{-1}(E))\cap H_x$ contains at least two points.
Because $d_y$ achieves a minimum at $x\in f^{-1}(E)$, for
any $w\in \Sigma_x f^{-1}(E)$, $|w\uparrow_x^y|\geq \frac \pi 2$. Together with
that $(\Sigma_xf^{-1}(E))\cap H_x$ is extremal in $H_x$, by (1.4.1) in \cite{PP} we conclude that  $|\uparrow_x^y(\Sigma_xf^{-1}(E))\cap H_x|\leq \frac \pi 2$, thus
$|\uparrow_x^y(\Sigma_x f^{-1}(E))\cap H_x|=\frac \pi2$. Again by (1.4.1) in \cite{PP}, for any $\zeta\in H_x$, $|\zeta\uparrow_x^y|\le \frac \pi2$. By now following the same argument
as in the proof of $E=\{\bar x\}$ (with replacing $\text{diam}(H_x)$ by $B_{\frac \pi2}(\uparrow_x^y)=H_x$), we conclude that for any $\xi\in \Sigma_xX$, $|\xi\uparrow_x^y|\le \frac \pi2$ i.e., $x$ is a critical point of $d_y$.
\end{proof}

\begin{proof} [Proof of (A2)]

We proceed the proof by induction on $m=\dim(Y)$, starting with $m=1$ i.e., $Y=S^1$, thus $E=\emptyset$ (trivial). Assume (A2) is true for $m=k$.

For $m=k+1$, if $E=\{\bar x\}$, or $\dim(E)=1$ and $\bar x\in \partial E$, then by Lemma \ref{2k} $f^{-1}(\bar x)$ is an extremal subset. For the rest cases, because $Df_x|_{H_x}: H_x\to \Sigma_{\bar x}Y$ is weakly integrable ((A1)) and $\Sigma_{\bar x}E\subset \Sigma_{\bar x}Y$ is an extremal subset, by the inductive assumption we conclude that $Df_{x}^{-1}|_{H_x}(\Sigma_{\bar x}E)$ is an extremal subset of $H_x$, thus by Lemma \ref{2k} again, $f^{-1}(E)$ is an extremal subset of $X$.
\end{proof}

\subsection {Proof of (A3)}

\begin{proof} [Proof of (A3)]

Let $E\subset Y$ be a primitive extremal subset, thus $E$ is connected.

Case 1. Assume that $E=\{\bar x\}$. By (A2), $f^{-1}(\bar x)$ is an extremal subset of $X$. We claim that $f^{-1}(\bar x)$ is a union of primitive extremal subsets of the same dimension $k$, which implies that for all $x\in f^{-1}(\bar x)$, $\dim(V_{x})=\ell$.

If the claim fails, because $f^{-1}(\bar x)$ is connected extremal subset we may assume two primitive extremal subsets of different dimensions intersecting at $x\in f^{-1}(\bar x)$. Consequently, $V_{x}=\Sigma_x f^{-1}(\bar x)\notin \text{Alex}(1)$ (\cite{KL}), a contradiction.

Case 2. Assume that $\dim(E)=k\ge 1$. Let $\mathcal R(E)\subset \overset oE$ be the set of regular points in $E$ i.e., $\bar x\in \mathcal R(E)$ if $C_{\bar x}E$ is isometric to $\mathbb R^k$ (\cite{Fuj1}). By the splitting property, for $x\in f^{-1}(\bar x)$, $C_{x}f^{-1}(E)=\mathbb R^k\times C_{x}f^{-1}(\bar x)$, $C_xf^{-1}(\bar x)\in \text{Alex}^{\ell+1}(0)$. Similar to the proof of Case 1, we conclude that primitive subsets of $f^{-1}(E)$ of maximal dimension that contains $x$ have the same dimension $k+\ell$. Consequently, $\dim(V_{x})=\ell$ for all $x\in f^{-1}(\bar x)$, $\bar x\in \mathcal R(E)$.

Case 3. Assume that $x\in f^{-1}(\overset oE\setminus \mathcal R(E))$, we shall show that $\dim(V_x)=\ell$. A point $\bar x\in E$ is  weakly $j$-strained, $0\le j\le k$, then there is a relative open neighborhood of $\bar x$, $U_{\bar x}$, homeomorphic to $\mathbb R^j\times C(\Sigma)$ by a map, $(h_{\bar x},g_{\bar x}): K(h_{\bar x},g_{\bar x}):\to \mathbb R^{j+1}$ (\cite{Ka2}, \cite{Fuj2}), similar to the case that $E=Y$, thus $Y_s=\overset oY$. As expected, the following
proof is similar to the proof of Theorem \ref{2a}, we shall proceed by the proof the inverse induction on $j$, starting with $j=k$ (Case 2). By the same argument as in the proof of Lemma \ref{2e}, if $g_{\bar x}^{-1}(0)$ contains a `good' point i.e., $\dim(V_x)=\ell$, then all points in $g_x^{-1}(0)$ are `good' i.e., $x\in f^{-1}(g_{\bar x}^{-1}(0))$, $\dim(V_x)=\ell$ (here Theorem \ref{2a} is required). It remains to show that $g_x^{-1}(0)$ contains at least one `good' point.

Assume that for all $j+1\le k$, $g_{\bar x}^{-1}(0)$ contains at least one good point. Let $\bar x\in \overset{o} E$ be a weakly $j$-strained an interior point. Arguing by contradiction, assuming the set $\emptyset\ne F\subsetneq \overset{o} E$ consisting of weakly
$j$-strained points such that $g_{\bar x}^{-1}(0)$ contains no good point.
Then $F$ satisfies the conditions of Lemma \ref{2j}, thus we conclude that $\bar F$ is extremal subset of $Y$. Because $F\subsetneq \overset oE$, $\bar F\subsetneq E$. Because $E$ is primitive, $\bar F= E\setminus \overset oE$, a contradiction because $\emptyset\ne F\cap \overset oE$.
\end{proof}

\subsection {Proof of Proposition  \ref{0j}}

Let $f: X\to Y$ be a submetry, $X\in \text{Alex}^n(\kappa)$, $Y\in \text{Alex}^m(\kappa)$.

\begin{proof} [Proof of Proposition  \ref{0j}]

We first prove that (0.9.1) implies (0.9.2). Let $h: \Sigma\to \Sigma_0$ be as in (0.9.2). By (0.9.1), $h: \Sigma\to \Sigma_0$ is integrable,
and by (2.2.1) $h: \Sigma\to \Sigma_0$ is a fiber bundle.
Let $\pi: \tilde \Sigma_0\to \Sigma_0$ denote the metric universal cover, and consider the following commutative diagram of the pullback fiber bundle,
\[\xymatrix{\pi^*\Sigma\ar[r]^{\pi^*}  \ar [d]^{\hat h} & \Sigma \ar[d]^{h}\\ \tilde \Sigma_0\ar [r]^{\pi} & \Sigma_0.}\]
Because $\tilde \Sigma_0$ is simply connected, by (2.2.1) $\pi^*\Sigma$ is homeomorphic to $\tilde \Sigma_0\times h^{-1}(\bar p)$.

Consider the weakly integrable submetry, $\tilde h: X=C(\pi^*\Sigma)\to C(\tilde \Sigma_0)=Y$, $\partial Y=\emptyset$, $\tilde h(t,x)=(t,h(x))$, $t\ge 0$, $x\in \tilde \Sigma_0$. Then $X\in \text{Alex}^{n+1}(0)$, $Y\in \text{Alex}^{m+1}(0)$, and $\tilde h$ is weakly integrable. By (0.9.1), $\tilde h$ is integrable. Because $\tilde h$-fiber over the vertex of $Y$ is the vertex of $X$, $\tilde h$ is a local isometry, thus $\hat h: \pi^*\Sigma\to \tilde \Sigma_0$ is a local isometry, and therefore $n=m$.

We then prove that (0.9.2) implies (0.9.1). Let $f: X\to Y$ be as in (0.9.1). For $x\in X$, by (A1) $Df_{x}: H_x\to \Sigma_{\bar x}Y$ is weakly integrable. By (0.9.2), $\dim(H_x)=\dim(\Sigma_{\bar x}Y)=\dim(\Sigma_xW_x)$. Because $\partial Y=\emptyset$, $\partial \Sigma_{\bar x}Y=\emptyset$ i.e., $\partial \Sigma_xW_x=\emptyset$. By Lemma \ref{2i}, $H_x=\Sigma_xW_x$ i.e., $f$ is integrable at any $x\in X$.
\end{proof}

\vskip4mm

\section{Proof of Theorem B}

We will first prove (B2), using which we prove (B1).

\subsection{Proof of (B2)}

Our proof of (B2) is based on the following lemma (Lemma \ref{3a}).

Let $f:X\to Y$ be a submetry between two Alexandrov spaces, $\partial Y=\emptyset$.  As mentioned in the introduction, for $x\in X$, $\Sigma_{x}X$ contains two closed convex subsets, $H_x$ and $V_x$, which are $\frac \pi2$-apart i.e., $|vw|=\frac \pi2$, for any $v\in  H_x$, $w\in V_x$ (see Lemma \ref{8a}). By the rigidity in Toponogov triangle comparison (i.e., the equal case), it follows that any three
points in $H_x, V_x$, which are not all in $H_x$ or $V_x$, determine an embedded standard
triangle from $S^2_1$. The existence of standard triangles `almost everywhere' on $\Sigma_{x}X$ puts a strong restriction on its underlying geometric structure. The following lemma is a criteria for $\Sigma_xX$ to have a (rigid) join structure.

\begin{lemma} {\rm (Join structures)} \label{3a} Let $X\in \text{Alex}^n(1)$, and let $X_i$ $(i=0,1$) be convex subsets of $X$ such that $\partial X_0=\emptyset$, $X_0$ and $X_1$ are $\frac \pi2$-apart i.e., for any $x_i\in X_i$, $|x_0x_1|=\frac \pi2$. Assume that $X=[X_0X_1]$ is a topologically nice sphere. Then
$X$ is isometric to $X_0*X_1$, and $X_i$ are topologically nice spheres.
\end{lemma}

Note that Lemma \ref{3a} will be false if $X$ is not a sphere or if $X\supsetneq [X_0X_1]$ (see Example 6.4).

The following is a local version of Theorem D in \cite{RX} (seen from its proof) will be used
in the proof of (B2).

\begin{lemma} {\rm(Topologically nice over a regular point)} \label{3b} Let $X\in \text{Alex}^n(\kappa), Y\in \text{Alex}^m(\kappa)$, $\partial Y=\emptyset$. Let $f: X\to Y$ be a submetry. If $\bar x\in Y$ is a regular point such that all points in $f^{-1}(\bar x)$ are topologically nice, then $f^{-1}(\bar x)$ is a topological manifold.
\end{lemma}

We restate (B2) in the following lemma.

\begin{lemma}\label {3c} Let $\Sigma\in \text{Alex}^n(1), \Sigma_0\in \text{Alex}^m(1)$ and $\partial \Sigma_0=\emptyset$. Let $h:  \Sigma\to \Sigma_0$ be a weakly integrable. Then $n=m$ under either of the following conditions:

\noindent {\rm (3.3.1)} $n-m\le 1$.

\noindent {\rm (3.3.2)} $\Sigma$ is a topologically nice homeomorphic sphere.
\end{lemma}

\begin{proof} (3.3.1) Arguing by contradiction, assuming that if $n-m=1$. We proceed by induction on $m$, starting with $m=1$, thus $\Sigma_0$ is a circle. If $n\ge 2$, we get a contradiction
following the same proof of (2.3.1).

Assume that (3.3.1) holds for $\dim(\Sigma_0)=m\ge 1$. Let $h: \Sigma\to \Sigma_0$ be a weakly integrable submetry, $\dim(\Sigma_0)=m+1$. For any $x\in \Sigma$, by (A1) $Dh_x: H_x\to \Sigma_{\bar x}\Sigma_0$ is weakly integrable, thus by induction $\dim(H_x)=\dim(\Sigma_{\bar x}\Sigma_0)=m=\dim(\Sigma_xW_x)$. Because $\partial (\Sigma_{\bar x}\Sigma_0)=\emptyset$ i.e.,  $\partial (\Sigma_xW_x)=\emptyset$, thus $H_x=\Sigma_xW_x$ (Lemma \ref{2i}) i.e., $h: \Sigma\to \Sigma_0$ is integrable at $x$.

If $\Sigma_0$ is simply connected, then by (2.2.1) $\Sigma$ is homeomorphic to $\Sigma_0\times h^{-1}(\bar p)$, $\bar p$ is a regular point in $\Sigma_0$. By Lemma \ref{2h}, $\partial \Sigma=\emptyset$, thus $h^{-1}(\bar p)$ is a circle (Lemma \ref{3b}); a contradiction that
$\pi_1(\Sigma)$ is finite.

If $\Sigma_0$ is not simply connected, we derive a contradiction by replacing $\Sigma_0$ with its metric universal cover $\tilde \Sigma_0$, and $\Sigma$ with the pullback $h$-fibration over $\tilde \Sigma_0$.

(3.3.2) Again we proceed the proof by induction on $m$, starting with $m=1$, thus $n=1$ (as seen in
the above proof).

Assume that (3.3.2) holds for $\dim(\Sigma_0)=m\ge 1$. Let $\dim(\Sigma_0)=m+1$. Arguing by contradiction, assume $n>m+1$, and consider $h: \Sigma\to \Sigma_0$ be a weakly integrable submetry, $\Sigma$ is a topologically nice homeomorphic sphere. For any $x\in \Sigma$, consider the weakly integrable submetry, $Dh_x: H_x\to \Sigma_{\bar x}\Sigma_0$ ((A1)). We claim that
$H_x$ is a topologically nice homeomorphic sphere, which enables us to apply induction
to conclude that $\dim(H_x)=\dim(\Sigma_{\bar x}\Sigma_0)=\dim(\Sigma_xW_x)$. Similar to the proof of (3.3.1), because $\partial H_x=\emptyset$ (Lemma \ref{2h}), by Lemma \ref{2i} $H_x=\Sigma_xW_x$ i.e., $h: \Sigma\to \Sigma_0$ is integrable at (any) $x$.
Consequently, $\Sigma_0$ is a topological manifold without boundary, and $h: \Sigma\to \Sigma_0$ is a fiber bundle ((2.2.1)).

If $\Sigma_0$ is simply connected, then $\Sigma$ is homeomorphic to $\Sigma_0\times h^{-1}(\bar p)$ ($\bar p$ is regular, Lemma \ref{3b}), a contradiction to that $\Sigma$ is a homeomorphic sphere. If $\Sigma_0$ is not simply connected, similarly by replacing $\Sigma_0$ with $\tilde \Sigma_0$ and $\Sigma$ with $\pi^*\Sigma$, we get a contradiction.

Finally, we verify the claim. First, $H_x$ and $V_x$ are convex subsets of $\Sigma_x(\Sigma)$, and by Lemma \ref{8a} $H_x$ and $V_x$ are $\frac \pi2$-apart. Because $[V_xH_x]=\Sigma_x\Sigma$ which is a sphere, by Lemma \ref{3a} $\Sigma_x\Sigma=H_x*V_x$, and $H_x$ and $V_x$ are topological spheres.
\end{proof}

\subsection {Proof of Lemma \ref{3a}}
\begin{proof}[Proof of Lemma \ref{3a}]

Observe that if $X\cong X_0*X_1$, then $X$ is a topological nice sphere iff $X_0$ and $X_1$
are topologically nice spheres.

We shall prove that $X\cong X_0*X_1$, by induction on $\dim(X)=n$, starting with $n=1$; $X_0$ and
$X_1$ are sets consisting of two points of distance $\pi$. Assume that if $X$ in Lemma \ref{3a} has dimension $n-1$, then $X\cong X_0*X_1$.

Consider $\dim(X)=n$. For $x\in X_0$, we apply induction on $(\Sigma_xX,\Sigma_xX_0,\Sigma^\perp_xX_0)$, and for $x\in X_1$ with $\partial X_1=\emptyset$,
again by induction on $(\Sigma_xX,\Sigma_xX_1,\Sigma^\perp_xX_1)$, to conclude that $\Sigma_xX\cong \Sigma_xX_0*\Sigma_x^\perp X_0$, and $\Sigma_xX\cong \Sigma_xX_1*\Sigma_x^\perp X_1$, respectively. Consequently, $\dim(X_0)+\dim(X_1)\le n-1$,
and ``$=$'' iff $\dim(\Sigma^\perp_xX_0)=\dim(X_1)$.

Summarizing the above, we need to verify that $\partial X_1=\emptyset$ and $\dim(\Sigma_x^\perp X_0)=\dim(X_1)$; because assuming the two properties, from the proof of Theorem \ref{1j} we see that given $x_0\in X_0$, $x_1\in X_1$, there are exact $m$-geodesics joining $x_0$ and $x_1$, or equivalently, there is a group of $m$-elements acting freely on $\hat X$ such that $X\cong \hat X/\Gamma$, a contradiction to that $X$ is simply connected, unless $m=1$ i.e., $X\cong X_0*X_1$.

We now verify that $\partial X_1=\emptyset$ and $\dim(\Sigma_x^\perp X_0)=\dim(X_1)$. Because
$X=[X_0X_1]$ is a topologically nice sphere, applying the Alexander duality to $(X,X_0,X_1)$ (\cite{Sp}),
we obtain
$$0\ne \tilde H_{\dim (X_0)}(X_0;\mathbb Z_2)\cong \tilde H^{n-\dim{X_0}-1}(X_1;\mathbb Z_2),$$
thus $\dim(X_1)\ge n-\dim(X_0)-1$, or $\dim(X_0)+\dim(X_1)\ge n-1$, therefore $\dim(X_0)+\dim(X_1)=n-1$. Consequently, $\dim(\Sigma^\perp_xX_0)=\dim(X_1)$.
\end{proof}

\subsection {Proof of (B1)}

Our proof of (B1) relies on (B2), (A3), and the following lemma.

\begin{lemma}\label{3d} Let $X\in \text{Alex}^n(\kappa)$, $X\in \text{Alex}^m(\kappa)$, $\partial Y=\emptyset$, and let $f: X\to Y$ be weakly integrable. If $\bar x\in Y$ is integrable,
then there is $\rho(\bar x)>0$ such that $f$ is integrable on $B_{\rho(\bar x)}(\bar x)$.
Moreover, all points in $B_{\rho(\bar x)}(\bar x)$ are good i.e., for any $z\in f^{-1}(B_{\rho(x)}(x))$, then $\dim(V_z)=n-m-1$.
\end{lemma}

\begin{proof} In the proof of (A1), we show that for any $x\in X$ and $f: W_x\to W_{\bar x}$ is an isometry, one is able to enlarge $W_x$ such that $W_{\bar x}\supset B_\rho(\bar x)$, for some
$\rho=\rho(\bar x)$ independent of $x\in f^{-1}(\bar x)$.

Assume that $\bar x\in Y$ is integrable, and $W_{\bar x}\supset B_\rho(\bar x)$ satisfies the above. Without loss of generality, we
may assume that $W_{\bar x}$ is convex (defined by a strictly concave function that approximated
by distance functions, \cite{PP}). Consequently, for any $\bar x\ne \bar z\in W_{\bar x}$,
there is a gradient curve $c(t)$ (of the concave function) along which $\bar z$ flows to
$\bar x$, and $\bar z$ is an interior point in $c(t)$.

We now show that for $x\ne x'\in f^{-1}(\bar x)$, $W_x\cap W_{x'}=\emptyset$. If $z\in W_x\cap W_{x'}$, then we may assume that $f(z)\in W_{\bar x}$ is an interior of $c(t)$ as in above. Then horizontal lifting of $c(t)$ (i.e., the gradient curve through $z$ of the horizontal lifting
of the concave function) is not unique (on $W_x$ and $W_{x'}$), a contradiction to that $\bar z$ (or $z$) is an interior point.

For any $z\in f^{-1}(W_{\bar x})$, the horizontal lifting of a minimal geodesic from $\bar z$ to
$\bar x$ at $z$ is a horizontal minimal geodesic from $z$ to $x\in f^{-1}(\bar x)$, thus
$z\in W_x$. Observe that when choosing $W_z=W_x$, it is clear that $H_z=\Sigma_zW_z$.

By (2.2.1), $f: f^{-1}(W_x)\to W_{\bar x}$ is a trivial bundle. Let
$z_0\in f^{-1}(B_{\rho(\bar x)}(\bar x))$ such that $\bar z_0$ a regular point in $Y$. Then
$C_{z_0}(f^{-1}(\bar z_0))=\mathbb R^m\times C(V_{z_0})$, thus $\dim (V_{z_0})=\dim_{\text{top}}(C_{z_0}(f^{-1}(\bar z_0)))-1=n-m-1$.

For any $z\in f^{-1}(W_{\bar x})$, $\bar z$ is a regular point in $E=\text{Ext}(\bar z)$, $\dim(E)=k$ and $\dim(f^{-1}(E))=\tilde k$. Then $C_z(f^{-1}(E))= \mathbb R^k\times C(V_z)$, $\dim (V_z)=\tilde k-k-1$. From Section 2 in \cite{PP} in explaining Theorem \ref{1g}, one sees that $n-m=\dim_{\text{top}}(f^{-1}(\bar z))\leq \tilde k-k=\dim (C(V_z))$. Thus we get $\dim (V_z)=n-m-1$ i.e., $\bar z$ is good. By (A3), every interior point of $E$ is good, thus all points in $W_{\bar x}$ ($\supset B_{\rho(\bar x)}(\bar x)$) are good.
\end{proof}

\begin{proof} [Proof of (B1)]

If $n-m\le 1$, then for $x\in X$, $\dim(H_x)-\dim(\Sigma_{\bar x}Y)\le 1$ and $\partial \Sigma_{\bar x}Y=\emptyset$. By (A1), $Df_x: H_x\to \Sigma_{\bar x}Y$ is weakly integrable,
and by (B2) $\dim(H_x)=\dim(\Sigma_{\bar x}Y)=\dim(\Sigma_xW_x)$. By Lemma \ref{2h}, $\partial H_x=\emptyset$, thus $H_x=\Sigma_xW_x$ (Lemma \ref{2i}) i.e., $f$ is integrable at any $x\in X$.

Note that for the case that $Y=E_{\min}$, equivalently $Y=Y_s$, (B1) is Theorem \ref{2a}.

Observe that if a component of $f^{-1}(E_{\min})$ contains $z$ such that $\dim(V_z)=n-m-1$, then by (A3) all points in $E_{\min}$ are good, because all points in $E_{\min}$ are interior points.
By Lemma \ref{3d}, we may assume a neighborhood $U(E_{\min})$ in which all all points are good.

Observe that if each $E_{\min}\subset Y$ contains integrable point $\bar x$, then by Lemma \ref{3d}
$\bar x$ is good.

Because any point $\bar x\in Y$, the interior of $\text{Ext}(\bar x)$ has a nonempty intersection with some $U(E_{\min})$ (from the stratification of $Y$), again by (A3) we conclude that all interior points in $\text{Ext}(\bar x)$ are good; in particular $\bar x$ is good (thus all points in $Y$ are good).

Finally we shall show that if a component of $f^{-1}(E_{\min})$ contains a topologically nice point $z$, then $\dim(V_z)=n-m-1$. Because $z$ is topologically nice, by Lemma \ref{3a} $\Sigma_zX\cong V_z*H_z$, $H_z$ is a topological sphere, and in particular $\dim(V_z)+\dim(H_z)=n-2$. Applying (B2) to the weakly integrable submetry, $Df_z: H_z\to \Sigma_{\bar z}Y$, we conclude that $\dim(H_z)=\dim(Y)-1=m-1$, thus $\dim(V_z)=n-m-1$.
\end{proof}

\begin{corollary} \label{3e} Let $f: X\to Y$ be weakly integrable. If $X$ is topologically nice, then
$f$ is integrable, and $Y$ is topologically nice, a $f$-fiber is a topological manifold.
\end{corollary}

\begin{proof} By (B1), $f: X\to Y$ is integrable, for $x\in X$, $f: W_x\to f(W_x)$ is an isometry. By Lemma \ref{3a}, $\Sigma_xX\cong H_x*V_x$, and $H_x$ and $V_x$ are topologically nice spheres, thus $Y$ is topologically nice. By Lemma \ref{3b}, for a regular point $\bar p\in Y$, $f^{-1}(\bar p)$ is a topological manifold.
\end{proof}

\section {Proof of Theorem C}

First, (C1) follows from (2.2.1).

\begin{lemma}\label{4a} Let $f: X\to Y$ be as in (C2). Then $f$ is integrable.
\end{lemma}

\begin{proof} For any $x\in X$, we shall show that $Df_x: H_x\to \Sigma_{\bar x}Y$ is injective,
thus $H_x=\Sigma_xW_x$. Observe that if $Df_x$ is injective on directions
tangent to minimal geodesics, then $Df_x$ is injective.

Assume $v_1\ne v_2\in H_x$ such that $Df_x(v_1)=Df_x(v_2)$ tangent to a minimal geodesic
$\gamma$ at $f(x)$ i.e., the horizontal lifting of $\gamma$, $\gamma_i'(0)=v_i$ and $f(\gamma_i)=\gamma$. Because $f^{-1}(\gamma(t))$ is $\eta$-convex, for each small $t>0$,
there is a minimal geodesic $\alpha_t(s)$ in $f^{-1}(\gamma(t))$ connecting $\gamma_1(t)$ with $\gamma_2(t)$. Let
$\ell(t)=|\alpha_t(s)|$. Then $\ell(0)=0$, and
$$\ell(t)^+(t)=
-\cos |\gamma_1^+(t)\Uparrow_{\gamma_1(t)}^{\gamma_2(t)}|-\cos |\gamma_2^+(t)
\Uparrow_{\gamma_2(t)}^{\gamma_1(t)}| \le 0,$$
for $t>0$ small, hence $\ell(t)\equiv 0$, a contradiction.
\end{proof}

Here is a brief outline of our proof of (C2). We shall show that the canonical trivialization map
is isometry (i.e., local fiber bundle metrically splits) by showing its differential exists and is an isometry. By \cite{Lyt1},
it suffices to show that the differential exists almost everywhere and is an isometry, thus it suffices
to verify the desired properties over $(m,\delta)$-strained points in $Y$. Because
any two points in a small $\rho$-ball consisting of $(m,\delta)$-strained points
can be pushed back and forth via gradient flows of distance functions in a definite
time depending on $\kappa$ and $\rho$, thus restricting to a fiber, $\phi_{\tilde p}$ is local
bi-Lipschitz map, thus $D\phi_{\tilde p}$ exists.

\begin{lemma}\label{4b} Let $f: X\to Y$ be as in Theorem C. Assume that all points in
$Y$ are $(m,\delta)$-strained (thus $f$ is integrable (B1)). Then the canonical local
trivialization map is locally bi-Lipschitz.
\end{lemma}

\begin{proof}
From the proof of (2.2.1), there is a canonical
local trivialization: for $\bar x\in Y$, there is $r>0$ such that
for all $x\in f^{-1}(\bar x)$, $W_{\bar x}\supseteq B_r(\bar x)$; and the canonical trivialization is defined by the map, $\phi_{\bar x}: f^{-1}(B_r(\bar x))\to f^{-1}(\bar x)\times B_r(\bar x)$; $\phi_{\bar x}(z)=(x_z,f(z))$, $z\in W_{x_z}$.

Based on that around each point $\bar x$ in $Y$, there is $\rho(\bar x)>0$ such that every point in $B_\rho(\bar x)$ is $(m,\delta)$-strained with a radius $\rho$, we can have alternative description for the canonical trivialization. First, any two points $\bar x_1, \bar x_2\in B_{\frac \rho8}(\bar x)$ can be pushed back and forth by gradient flows of (selected) distance functions ($\{\frac 12 d_{\bar z_i}^2\}$) which are $\lambda$-concave with $\lambda=\lambda(\kappa,\rho)>0$, and the total time is bounded above by a constant $T(\kappa,\rho)>0$. Via the gradient flows of the distance functions, $\{\frac 12d_{f^{-1}(\bar z_i)}^2\}$ (\cite{KPT}, \cite{SY}), one gets a local trivialization by identifying a point in $f^{-1}(\bar x)$ with a point in $f^{-1}(\bar y)$ for all $\bar y\in B_{\frac \rho8}(\bar x)$. Because the gradient flows stayed in $W_x$ (\cite{KL}), it is clear that this local trivialization coincides with the canonical trivialization.

Based on the alternation description of a canonical local trivialization, it is clear that $\phi_{\bar x}$ is locally bi-Lipschtiz.
\end{proof}

A Lipschitz function has a differential almost every where. For a Lipschitz map, $f: X\to Y$, ($X$ and $Y$ are Alexandrov spaces), one naturally extends the notion of differential $Df$,
which exists almost everywhere, and almost everywhere is a linear map.
In particular, if $f$ is a submetry, then $Df$ exists everywhere (because a distance function is differentiable everywhere and almost everywhere line (\cite{Lyt1}, \cite{KL}, \cite{Pet}).

\begin{theorem} {\rm (\cite{Lyt1}, \cite{Wo1})} \label{4c} Let $X\in \text{Alex}^n(\kappa), Y\in \text{Alex}^m(\kappa)$, and let $f: X\to Y$ be a continuous map.

\noindent {\rm (4.3.1)} There is $\delta>0$ such that if $f$ is locally Lipschitz around $(n,\delta)$-strained points in $X$, then $Df$ exists almost everywhere, and is linear almost
everywhere.

\noindent {\rm (4.3.2)} If $Df$ in (4.3.1) is $1$-Lipschitz, then $f$ is $1$-Lipschitz. If
$Df$ in (4.3.1) is an isometry and $f$ is locally homeomorphic, then $f$ is a local isometry.
\end{theorem}

Note that in (4.3.2), $Df$ is an isometry almost everywhere is not enough to conclude that $f$ is a local homeomorphism; e.g., $f=\text{proj}: \bar B_1^2(0) (\subset \mathbb R^2)\to \bar B_1^2(0)/\sim$, $u\sim -u$, $u\in \partial \bar B_1^2(0)$.

\begin{proof}[Proof of (C2)]

Let $\pi: (\tilde Y,\tilde p)\to (Y,\bar p)$ be a metric universal covering map, and let $\hat X$ denote a $\pi$-pullback bundle of $f: X\to Y$. By (C1), $\hat X$ is a metric covering space with a canonical topological splitting i.e., there is a homeomorphism, $\phi_{\tilde p}: \hat X\to \hat f^{-1}(\bar p)\times \tilde Y$, $\phi_{\tilde p}(\tilde x)=(\text{proj}_{\bar p}(\hat x), \hat f(\hat x))$. Then $\hat f=\text{proj}_2\circ \phi_{\tilde p}: \hat X\to \tilde Y$.
Hence, (C2) follows from that $\phi_{\bar p}$ is an isometry, and by (4.3.2) it reduces to find a subset of full measure, $\hat X_0$, on which $D\phi_{\tilde p}$ exists and is an isometry, where $\hat f^{-1}(\bar p)\times
\tilde Y$ is equipped with the product metrics.

Let $\tilde Y_\delta\subseteq \tilde Y$ denote the set of $(m,\delta)$-strained points in $\tilde Y$ ($0<\delta<<1$). Because $\tilde Y_\delta$ is a full measure open subset of $Y$, $ \hat f^{-1}(\tilde Y_\delta)$ is a full measure open subset of $\hat X$. From the proof of Lemma \ref{4b}, it follows that $ \text{proj}_{\tilde p}:  \hat f^{-1}(\tilde Y_\delta)\to  \hat f^{-1}(\tilde p)$ is locally Lipschitz, thus $D\text{proj}_{\tilde p}$ exists almost everywhere on $\hat f^{-1}(\tilde Y_\delta)$ ((4.3.1)). Let $\hat  X_0\subset \hat X$ denote the set of regular points i.e., whose tangent cone is an Euclidean space. Then $\tilde X_0$ has a full measure, thus $\hat X_\delta=\hat X_0\cap  \hat f^{-1}(\tilde Y_\delta)$ has a full measure in $\hat X$. Without loss of generality we may assume that $D\phi_{\tilde p}$ exists on $\hat X_\delta$.

We now show that $D\phi_{\tilde p}$ is an isometry at $\hat x\in \hat X$. Because $\phi_{\tilde p}$ is isometric when restricted to $W_{\hat x}$, it suffices to show that $\text{proj}_{\tilde  p}: \hat f^{-1}(\hat f(\hat x))\to \text{proj}_{\tilde p} (\hat f^{-1}(\hat f(\hat x)))$ is a local isometry, which implies that $D\phi_{\tilde p}|_{V_{\hat x}}=D\text{proj}_{\tilde  p}$ is an isometry, and therefore $D\phi_{\tilde p}$ is an isometry.

For any $\hat x_1, \hat x_2\in \hat f^{-1}(\hat f(\hat x))$ such that $|\hat x_1\hat x_2|<\frac 12\eta$,
using the standard open-closed argument we will show that $|\hat x_1\hat x_2|=|\text{proj}_{\tilde  p}(\hat x_1)\text{proj}_{\tilde  p}(\hat x_2)|$. Let $\alpha(t)$ be minimal geodesic in $\hat f^{-1}(\hat f(\hat x))$ from $\hat x_1$ to $\hat x_2$. We observe the following properties:

(C2.1) Let $\bar \gamma$ be a minimal geodesic from $\hat f(\hat x_1)$ to $\tilde p$. For each $\alpha(t)$, there is a unique horizontal lifting $\hat \gamma_{\alpha(t)}\subset W_{\alpha(t)}$, a minimal
geodesic from $\alpha(t)$ to $\hat \gamma_{\alpha(t)}\cap \hat f^{-1}(\tilde p)$. In particular,
$|\alpha(t)\hat f^{-1}(\tilde p)|=|\bar \gamma|$. Because $|\hat x_1\hat x_2|<\frac 12\eta$, by continuity we may assume $s_0>0$ such that $|\hat \gamma_{\hat x_1}(s) \tilde \gamma_{\hat x_2}(s)|<\eta$ for $0<s<s_0$.

(C2.2) We shall show that for $0\le t\le \frac 12\eta$, $s_0$ in (C2.1) can be chosen small
such that for $0<s\le s_0$, $|\hat x_1\hat x_2|=|\hat \gamma_{\hat x_1}(s)\hat \gamma_{\hat x_2}(s)|$.

Note that by the standard open-closed argument (`open' is proved in (C2.2), and the `closed' follows from the continuity), (C2.2) implies the desired $|\hat x_1\hat x_2|=|\text{proj}_{\tilde p}(\hat x_1)\text{proj}_{\tilde  p}(\hat x_2)|$.

Because $\hat f^{-1}(\bar \gamma(s))$ is $\eta$-convex, there is a minimal geodesic
$[\hat \gamma_{\hat x_1}(s) \hat \gamma_{\hat x_2} (s)]\subset  \hat f^{-1}(\bar \gamma(s))$. Let $\ell(s)=|\hat \gamma_{\hat x_1}(s) \hat \gamma_{\hat x_2} (s)|$. Following the proof of
Lemma \ref{4a}, $\ell^+(s)\le 0$, thus $\ell^+ (s)\equiv 0$ for $s<s_0$, and therefore $|\hat  \gamma_{\hat x_1}(s) \hat \gamma_{\hat x_2} (s)|$ is a (local) constant.
\end{proof}

We conclude this section with a simple example that an integrable submetry with one
point global integrable may not be global integrable (comparing (0.9.2)).

\begin{example} (One point global integrable but not a global integrable submetry)
Let $f: X\to Y$ be an integrable submetry. We shall show an example that $f: X\to Y$ is global integrable at one point, but $f$ is not global integrable.

Let $S^1\rtimes \mathbb R^1=[0,1]\times \mathbb R^1/[(0,u)\sim (1,-u)]$ denote a flat twisted $\mathbb R^1$-bundle over $S^1$, and let $X=S^1\rtimes \mathbb R\times \mathbb R$ be the metric product, which is an open flat $3$-manifold with a soul $S=S^1\times (0,0)$. Because $S$ is a Riemannian manifold, via the horizontal lifting one sees that the Sharafutinov restriction, $\phi: X\to S$, is an integrable submetry, but $\phi$ is not globally integral (because $X$ is not a topological product of $S$ and $\mathbb R^2$).
For any $0\ne v\in \mathbb R^1$, $\phi: (X\supset)\, S\rtimes 0\times v\to S$ is an isometry i.e., $\phi$ is a global integrable at $(\{s\}\rtimes 0)\times v$, $s\in S$.
\end{example}

\section {Proof of Theorem D}

In our proof of Theorem D, a basic tool is the Perel'man result that a strictly non-critical map is a locally fiber bundle projection (comparing Theorem \ref{1g}); which we review below; indeed the construction of `slices'
transversal to fibers

\subsection {Strictly non-critical maps}

Let $Y\in \text{Alex}^m(-1)$, and $\bar p\in Y$ be a weakly $(m,\delta)$-strained point i.e.,  $\bar p_1,...,\bar p_m, \bar w\in Y$ such that
$$\measuredangle \bar p_i\bar p\bar p_j> \frac \pi2, \quad \exists \, \bar w\in Y, \quad \measuredangle \bar p_i\bar p\bar w> \frac \pi2.$$
By continuity, there is an open set $U$ such that for any $\bar x\in U$, the above angle
inequalities hold at $\bar x$. According to \cite{Per1}, the map, $f=(|p_1,\cdot|,\cdots,|p_m\cdot|): U\to \mathbb R^m$, is called strictly non-critical, which
has the following properties.

\begin{lemma} {\rm (\cite{Per1}) \label{5a} (5.1.1) (Local product structure)} Let $f: U\to \mathbb R^k$ be strictly non-critical at $\bar p\in U$. Then there is an open neighborhood $p\in V\subset U$ such that $V$ is homeomorphic to $\mathbb R^k\times (f^{-1}(f(p))\cap V)$.

\noindent {\rm (5.1.2) (Local bundle structure)} If $f: U\to \mathbb R^m$ is proper and strictly non-critical on $U$, then $f: U\to \mathbb R^m$ is a fiber bundle map.
\end{lemma}

In the proof, we shall also use the following simple fact.

\begin{lemma}\label{52} Let $(X,S,\phi)$ be as in Theorem 0.4. For $\bar x\in S$, assume that there is $\rho>0$ such that $\phi^{-1}(B_\rho(\bar x))$ satisfies the following conditions:

\noindent {\rm (5.2.1)} For any $x\in\phi^{-1}(\bar x)$, there is a subset $D_x$ such that $\phi: D_x\to \phi(D_x)=D_{\bar x}\subset S$ is a homeomorphism, and for $\bar z\in D_{\bar x}$, $D_x\cap \phi^{-1}(\bar z)=\{z\}$.

\noindent {\rm (5.2.2)} For $x\ne x'\in \phi^{-1}(\bar x)$, $D_x\cap D_{x'}=\emptyset$.

\noindent {\rm (5.2.3)} $\phi^{-1}(D_{\bar x})=\bigcup_{x\in \phi^{-1}(\bar x)}D_x$.

Then $\phi: \phi^{-1}(D_{\bar x})\to D_{\bar x}$ is a trivial fiber bundle map.
\end{lemma}

\begin{proof} We define a trivialization map, $\psi: \phi^{-1}(B_\rho(\bar x))\to B_\rho(\bar x)\times \phi^{-1}(\bar x)$, $\psi(z)=(\phi(z),D_z\cap \phi^{-1}(\bar x))$. Then $\text{proj}_1\circ \psi=\phi$ i.e., $\phi$ is a trivial bundle map.
\end{proof}

\subsection {Proof of Theorem D}

\begin{proof} [Proof of Theorem D]

Let $(X,S,\phi)$ be as in Theorem 0.4. For $\bar x\in S\in \text{Alex}^m(0)$, by (5.1.1) we may assume $\bar p_1,...,\bar p_m\in S$, and $\rho=\rho(\bar x)>0$, such that $f_{\bar x}=(d_{\bar p_1},...,d_{\bar p_m}): B_\rho(\bar x)\to \Bbb R^m$ is a homeomorphic embedding, $B_\rho(\bar x)$ is contractible to $\bar x$, and $f_{\bar x}(B_\rho(\bar x))\supset I_\eta^m$, $\eta>0$.

We shall show that $\phi: \phi^{-1}(f^{-1}_{\bar x}(I_\eta^m))\to f^{-1}_{\bar x}(I_\eta^m)\, (\subset B_\rho(\bar x))$ satisfies (5.2.1)-(5.2.3).

Without loss of generality, we may assume a Busemann function $b_{\bar p}$ is chosen such that $\bar p\in S$, thus $S\subset b_{\bar p}^{-1}(0)$. For $x\in X\setminus S$, let $R<b_{\bar p}(x)-100$. Then $x\in \Omega_R=\{z\in X, b_{\bar p}(z)\ge R\}$ is a compact convex subset and $|x\partial \Omega_R|\ge 10$.
Because $\phi$ is a submetry and because all points in $S$ are weakly $m$-strained, it is clear that around $x$, $f_{\bar x}\circ \phi: \Omega_R\to \Bbb R^m$ is a proper and strictly noncritical map. By (5.1.2), $\phi: \Omega_R\to S$ is a fiber bundle projection. In particular, $\phi: \phi^{-1}(B_\rho(\bar x))\cap \Omega_R\to B_\rho(\bar x)$ is a fiber bundle map.

We now explain that $\phi: \phi^{-1}(f_{\bar x}^{-1}(I^m_\eta))\to f_{\bar x}^{-1}(I^m_\eta)$ satisfies (5.2.1)-(5.2.3), thus a trivial bundle. According to \cite{Per1} (see p.7, Complement to Theorem A), for any $x\in \phi^{-1}(\bar x)$, the size of a `slice' at $x$ depends only on $I^m_\eta$ (a slice at $x$ is
a subset satisfying (5.2.1)-(5.2.3)), or in another word, as long as $x\in \Omega_R$ and not close to $\partial \Omega_R$, the slice at $x$ contains a subset homeomorphic to $f_{\bar x}^{-1}(I^m_\eta)$.
Since any $x\in f^{-1}(\bar x)$ is in some $\Omega_R$, it is clear that $\phi: \phi^{-1}(f_{\bar x}^{-1}(I^m_\eta))\to f^{-1}_{\bar x}(I^m_\delta)$ satisfies (5.2.1)-(5.2.3), thus $(\phi^{-1}(f_{\bar x}^{-1}(I^m_\eta)),f_{\bar x}^{-1}(I^m_\eta),\phi)$ is a trivial bundle.

One may also see the triviality from an alternative point of view: let $h_t: B_\rho(\bar x)\to B_\rho(\bar x)$ be a homotopy equivalence, $h_0=\text{id}_{B_\rho(\bar x)}$ and $h_1\equiv \bar x$, a constant map, and consider pullback bundles of $h_t$ from the bundle, $\phi: \phi^{-1}(B_\rho(\bar x))\cap \Omega_R\to B_\rho(\bar x)$. Because $h_1$ is homotopy equivalent to $h_0$, $\phi: \phi^{-1}(B_\rho(\bar x))\to B_\rho(\bar x)$ is equivalent to a trivial bundle over $B_\rho(\bar x)$; which implies that (5.2.1)-(5.2.3) are satisfied. Again because this triviality is independent of $R>>1$, one concludes that $\phi: \phi^{-1}(B_\rho(\bar x))\to B_\rho(\bar x)$ satisfies (5.2.1)-(5.2.3), thus this is a trivial bundle.
\end{proof}

\section {Proof of Theorem E}

\subsection{Proof of (E1)}


Let $(X,S,\phi)$ be in Theorem \ref{sol-alex}. Note that $S$ is determined by the Busemann function at
some point in $X$.

\begin{lemma}\label{6a} Let $(X,S,\phi)$ be as Theorem E; $\bar p\in S$ is a
regular point of $S$ such that every $\bar v\in \Sigma_{\bar p}^\perp S$ tangents to a ray.
Then $S$ coincides with the soul of $b_{\bar p}$.
\end{lemma}

\begin{proof} Because any $\bar v\in \Sigma_{\bar p}^\perp S$ tangents to a ray,
$\max b_{\bar p}=b_{\bar p}(\bar p)=0$.
Let $S'$ denote the soul determined by $b_{\bar p}$, $b_{\bar p}(S')\equiv 0$.
For $\bar q\in S'$,
let $\gamma_q$ denote a minimal geodesic from $\bar p$ to $\bar q$. Then
$\gamma'(0)$ is
orthogonal to $\Sigma_{\bar p}^\perp S$, $\gamma'(0)\in \Sigma_{\bar p}S$,
and because $S$ is
convex, $\gamma\subset S$ (see Lemma \ref{1b})
thus $\bar q\in S$. Because convex subsets without boundaries, $S'\subseteq S$
and  $\dim(S')=\dim(S)$, $S'=S$.
\end{proof}

For any $\bar x\in S$, and a minimal geodesic $\gamma_{\bar x}$ from
$\bar p$ to $\bar x$, one obtains a map,
$P_{\gamma_{\bar x}}: \Sigma_{\bar p}^\perp S\to \Sigma_{\bar x}^\perp S$,
$\bar w=P_{\gamma_{\bar x}}(\bar v)$ is determined by
requiring that $g\exp_{\bar p}t\bar v$, $\gamma_{\bar x}$ and
$g\exp_{\bar x}t\bar w$
bound a flat strip in (1.13.2).

\begin{lemma}\label{6b}
Let $(X,S,\phi)$ be as in Theorem E, and let
$P_{\gamma_{\bar x}}$ be in the above. Then $P_{\gamma_{\bar x}}$ is an
onto map,
thus each $\bar v\in \Sigma_{\bar x}^\perp S$ tangents toa ray.
\end{lemma}

\begin{proof}
We shall divide the verification, that for any
$\bar x\in S$, $P_{\gamma_{\bar x}}: \Sigma_{\bar p}^\perp S\to
\Sigma_{\bar x}^\perp S$ is an onto map, into two cases.

Case 1. Assume that $\bar x\in \mathcal {R}(S)$
(the set of regular points in $S$ i.e., points where the tangent cones are isometric to $\mathbb R^m$, $m=\dim(S)$). Because $\Sigma_{\bar x}S=S^{m-1}_1$ (the unit sphere)
i.e., $\Sigma_{\bar x}S$ contains $(m-1)$ pairs of points of distance $\pi$,
each pair determines a spherical suspension structure
on $\Sigma_{x}X$, thus one concludes that $\Sigma_{x}X$ has a join structure,
$\Sigma_{\bar x}S*\Sigma_{\bar x}^\perp S$. Consequently,
$\dim(\Sigma_{\bar x}^\perp S)=n-m-1$, and $P_{\gamma_{\bar x}}:
\Sigma_{\bar p}^\perp S\to \Sigma_{\bar x}^\perp S$ is injective.
Because $P_{\gamma_{\bar x}}(\Sigma_{\bar p}^\perp S)\subset
\Sigma_{\bar x}^\perp S$ without boundary, $P_{\gamma_{\bar x}}
(\Sigma_{\bar p}^\perp S)=\Sigma_{\bar x}^\perp S$ (indeed, $P_{\gamma_{\bar x}}$
is actually isometry, \cite{Li}, \cite{RW2}).

Case 2. Assume that $\bar x\in S\setminus \mathcal {R}(S)$, and $\bar v\in
\Sigma_{\bar x}^\perp S$,
we shall find $\bar w\in \Sigma_{\bar p}^\perp S$ such that
$P_{\gamma_{\bar x}}(\bar w)=\bar v$.

For $t_i\to 0$, $z_i=g\exp_{\bar x} t_i\bar v$, by the lower semi-continuous of angles there is sequence $s_i(t_i)\to 0$,
$\bar v_i\in \Sigma_{\gamma_{\bar x}(\ell-s_i)}^\perp S$ such that
$g\exp_{\gamma_{\bar x}(\ell-s_i)}t_i\bar v_i$ is close to $z_i$, where
$\ell=|\bar p\bar x|$.
Parallel transport $\bar v_i$ to
$\bar w_i\in \Sigma_{\bar p}^\perp S$ along $\gamma_{\bar x}(\ell-t)$,
passing to a subsequence we may assume $\bar w_i\to \bar w$. It is
clear that $P_{\gamma_{\bar x}}(\bar w)=\bar v$.
\end{proof}

\begin{proof} [Proof of (E1)]

It follows from Lemma \ref{6b} that for $\bar x\in S$,
$g\exp_{\bar x}: C(\Sigma_{\bar x}^\perp S)\to \phi^{-1}(\bar x)$ is a
bijection,
thus $g\exp: C(\Sigma^\perp S)\to X$ is a bijection such that
$\phi=\text{proj}\circ (g\exp)^{-1}$, where $\text{proj}:
C(\Sigma^\perp S)\to S$ denotes the projection map.
In particular, $\phi$ is a submetry, and equipped
$C(\Sigma^\perp S)$ with the pullback topology,
$g\exp$ is a homeomorphism.
\end{proof}

\subsection{Proof of (E2)}

A continuous surjection $f: X\to B$ between metric
spaces is called {\it completely regular} if for each $\bar p\in B$ and
$\epsilon>0$, there is a $\delta>0$ such that for any $\bar q\in
B_\delta(\bar p)$,
there is a homeomorphism $h:f^{-1}(\bar q)\to f^{-1}(\bar p)$ with
$|h(x)x|<\epsilon$ for all $x\in f^{-1}(\bar q)$.

The following is a bundle criterion we will use in the proof of (E2).

\begin{theorem} [Completely regular map, \cite{DH}, \cite{Ha}, \cite{Ki} ]
Let $f: X\to B$ be a completely regular map from a bounded complete
metric space $X$ onto a finite covering dimensional metric space
$B$. Suppose a fiber is homeomorphic to a compact $n$-manifold $M$.
Then $f:E\to B$ is a fiber bundle projection.
\end{theorem}

\begin{proof}[Proof of (E2)]

We will show that $\text{proj}: \Sigma^\perp S\to S$ is completely regular, thus a fiber bundle map (Theorem 6.3). Moreover, $\text{proj}: \Sigma^\perp S\to S$ naturally extends to a bundle map,
$\text{proj}:C(\Sigma^\perp S)\to S$. By (E1), one sees that $(X,S,\phi)$ is a fiber
bundle such that $\phi=\text{proj}\circ (g\exp)^{-1}$ i.e., $g\exp: C(\Sigma^\perp S)\to X$
is a bundle equivalence between $(C(\Sigma^\perp S),S,\text{proj})$ and $(X,S,\phi)$.

By Lemma \ref{3a}, for any $\bar x\in S$, $\Sigma_{\bar x}X$ is isometric to
$\Sigma_{\bar x}S*\Sigma_{\bar x}^\perp S$, $\Sigma_{\bar x}S$ and $\Sigma_{\bar x}^\perp S$ are
topologically nice spheres.
Consequently, $S$ itself is topologically nice.
Let $\Sigma^\perp S$ be equipped with the pullback topology via $g\exp$ from $\partial B_1(S)$.
Given any $\epsilon>0$, for $\delta>0$ sufficiently small, $\bar x\in B_\delta(\bar p)
\subset S$, the map in Lemma \ref{6b}, $P_{\gamma_{\bar x}}: \Sigma_{\bar p}^\perp S\to
\Sigma_{\bar x}^\perp S$ (which coincides with the parallel transport along a
minimal geodesic from $\bar p$ to $\bar x$), is a homeomorphism such that
$d(\bar v,P_{\gamma_{\bar x}}(\bar v))<\epsilon$.
\end{proof}

\begin{example}
(6.4.1) The complex projective space, $\mathbb CP^m$,
equipped with the Fubini-Study metric contains two $\frac \pi2$-apart closed
totally geodesic submanifolds, $X_0=\mathbb CP^{m-2}\subset \mathbb CP^m$ and
$X_1=\mathbb CP^1\in \mathbb CP^m$. Note $\dim(X_0)+\dim(X_1)=2m-2$, because
$\mathbb CP^m$ is not a sphere for $m\ge 2$.

(6.4.2) Let $S(S^1_{\frac 14})$ denote a spherical suspension of a circle of
circumference $\frac \pi2$, and let $X=S(\frac 12S(S^1_{\frac 14}))$ denote the
spherical suspension of $\frac 12S(S^1_{\frac 14})$, which is a rescaling
$S(S^1_{\frac 14})$ by $\frac 12$. Then $X\in \text{Alex}^3(1)$ is a
topologically nice $3$-sphere, which contains two $\frac \pi2$-apart convex
subsets without boundaries, $X_0=\frac 12S^1_{\frac 14}$ and $X_1=\{p,q\}$,
the vertices of $X$. Note that $\dim(X_0)+\dim(X_1)=1$, because $X\supsetneq [X_0X_1]$.
\end{example}


\section {Proof of Theorem F}


\subsection {Proof of Theorem F by assuming Key Lemma \ref{7c}}


In our proof of Theorem F, the starting point is the following result.

\begin{theorem} {\rm (\cite{Li})} \label{7a}
 Let $(X,S,\phi)$ be as in Theorem \ref{sol-alex}. Assume that $X$ is topologically nice and $\dim(S)=\dim(X)-2$. Then either $g\exp: C^\perp S\to X$ is a homeomorphism, or $\phi$ is weakly integrable.
\end{theorem}

Based on Theorem \ref{7a}, in the proof of Theorem F it remains to show that $(X,S,\phi)$ is weakly integrable and $\dim(S)=\dim(X)-2$ imply that $(\tilde X,\tilde S,\tilde \phi)$ splits, where $\pi: (\tilde X,\tilde p)\to (X,p)$ is the metric universal covering map, $\tilde \phi: (\tilde X,\tilde p)\to (\tilde S,\tilde p)$ is the lifting of $\phi: X\to S$ ($\bar p\in S$), and $\tilde S=\pi^{-1}(S)$.

\begin {lemma} \label{7b}
Let $(X,S,\phi)$ be as in Theorem \ref{sol-alex}. Assume that $\phi$ is weakly integrable and $\dim(S)=\dim(X)-2$. Then $\phi$ is integrable. 
\end{lemma}

\begin{proof}
Because $\phi$ is weakly integrable, for any $x\in X$, $D\phi: H_x\to \Sigma_{\phi(x)}S$ is weakly integrable (A1). Because
$\dim(H_x)\le \dim(H_x)+\dim(V_{x})\le n-2=\dim(\Sigma_{\phi(x)}W_{\phi(x)})+1$, $D\phi: H_x\to \Sigma_{\phi(x)}S$ is an isometry (B2). Because $\Sigma_{x}W_x\subseteq H_x$ and $D\phi: \Sigma_{x}W_x\to \Sigma_{\phi(x)}S$ is an isometry, $H_x=\Sigma_{x}W_x$ i.e., $\phi$ is integrable.
\end{proof}

Because $(X,S,\phi)$ is integrable, $(\tilde X,\tilde S,\tilde \phi)$ is global integrable. Observe that if $\pi_1(S)$ is not finite, then $\tilde S$ splits, $\tilde S=\hat S\times \mathbb R^k$ ($k\ge 1$), and $\hat S$ is compact and simply connected. Because $\tilde S\subset X$ is convex, the splitting on $\tilde S$ implies a splitting on $\tilde X$, $\tilde X=\mathbb R^k\times Z$, where $Z\in \text{Alex}(0)$ is open. It is clear that $\hat S$ is a soul of $Z$, and $\tilde \phi: Z\to \hat S$, is global integrable, thus $Z$ is homeomorphic to $\hat S\times \phi^{-1}(\bar q)$.

In view of the above, $\tilde X$ splits into $\tilde S\times \phi^{-1}(\bar q)$ if and only
 if
$Z$ splits into $\hat S\times \phi^{-1}(\bar q)$.

\begin{key lemma} \label{7c} Let $(X,S,\phi)$ be as in Theorem \ref{sol-alex}. Assume that $\phi: X\to S$ is global integrable and $\dim(S)=\dim (X)-2$. Then for any regular point $\bar q\in S$, the map, $\text{proj}_{\bar q}: X\to \phi^{-1}(\bar q)$, $\text{proj}_{\bar q}(x)=W_x\cap \phi^{-1}(\bar q)$, is $1$-Lipschitz, thus $\phi^{-1}(\bar q)$ is convex. Consequently, the convexity applies to
all $\bar q\in S$.
\end{key lemma}

\begin{proof} [Proof of Theorem F by assuming Key Lemma \ref{7c}]

(F1) Follows from Theorem \ref{7a} and (E2).

(F2) Assume that $\phi$ is weakly integrable. By Lemma  \ref{7b}, $\phi$ is integrable. Let $\tilde X$ denote the metric universal covering of $X$. For any $\bar q\in S$, by Key Lemma \ref{7c} $\phi^{-1}(\bar q)$ is convex (for a similar technique, see \cite {Wo2}), thus we apply (C2) to conclude that $\tilde X$ splits, $\tilde X=\hat S\times \Bbb R^k\times \phi^{-1}(\bar q)$, where $\tilde S$ is the universal cover of $S$, thus $\tilde S$ splits into $\Bbb R^k\times \hat S$, where $k$ is the maximal rank of a
free abelian subgroup of $\pi_1(X)$ and $\hat S$ is compact and simply connected.

\end{proof}


\subsection{Proof of Key Lemma  7.3}

In our proof of Key Lemma \ref{7c}, we employ a strategy similar to one in the proof of (C2), i.e.
based on Theorem \ref{4c} it reduces to show that at any {\it regular point} $\bar q\in S$ (which guarantees that $\text{proj}_{\bar q}$ is a locally Lipschitz map), thus $\text{proj}_{\bar q}: X\to \phi^{-1}(\bar q)$ almost everywhere has a differential that is $1$-Lipschitz.

First we recall some notation.
Let $(X,S,\phi)$ be as in Key Lemma \ref{7c}, and let $S_\delta\subseteq S$ denote the set of $(m,\delta)$-strained points in $S$ ($0<\delta<<1$).

Without loss of generality we may assume that $D\text{proj}_{\bar q}$ exists on $\hat X$,  where $\hat X=X_0\cap \phi^{-1}(S_\delta)$ and $X_0\subset X$ denote the set of regular points.
By (4.3.2), is remains to
show that, up to a measure zero subset of $\hat X$, $D\text{proj}_{\bar q}$ is $1$-Lipschitz.

Because $\text{proj}_{\bar q}(W_x)=\text{proj}_{\bar q}(x)$, we shall show that $D\text{proj}_{\bar q}: V_{x}\to D\text{proj}_{\bar q}(V_{x})$ is an isometry; up to a measured zero subset of $\hat X$ we shall `canonically' choose two independent
vectors, $(u,v)$, in $V_{x}$, and show that $D\text{proj}_{\bar q}$ preserves both norm and angle for $(u,v)$, thus $D\text{proj}_{\bar q}$ is $1$-Lipschitz.

Our selection of two independent vectors in $V_x$ based on the following geometric properties:

(7.4.1) Up to a measure zero subset of $\hat X$, $x\in \hat X$, $\Uparrow_x^{\partial \Omega_c}=\uparrow_x^{\partial \Omega_c}$ i.e., there is unique ray from $x$ (Lemma \ref{7e}). Then we choose $u=\uparrow_x^{\partial \Omega_c}\in V_x$.

A selection of $v$ requires more work.

Fixing a regular point in $S$, without loss of generality we may assume
$\bar p$, which used in $b_{\bar p}$, is a regular point. Because $\phi: X\to S$ is global integrable, for $\bar p\ne \bar q\in S$, each ray at $\bar p$, $g\exp_{\bar p}tv$, determines (via the partial parallel flat strip) determines a ray at $\bar q$ (which is independent of a minimal geodesic from $\bar p$ to $\bar q$). Consequently, one gets an isometric embedding, $S\times \{\exp_{\bar p}tv,t\ge 0\} \to X$; let $F_v$ denote the image. Let $\hat V_{\bar p}\subset V_{\bar p}$ denote the subset consisting of such $v$'s. Because $V_{\bar p}$ is a circle of radius $\le \pi$, one may choose $\{v_i\}_{i=1}^k\subset \hat V_{\bar p}$, $k\le 3$, such that for $v\in V_{\bar p}\backslash\hat V_{\bar p}$, $|v\{v_i\}|\le \frac \pi2$ (\cite{Li}). Let $F=\bigcup_iF_{v_i}$, which is closed and local convex away from $S$. Because $(X,S,\phi)$ is global integrable, one is able to apply \cite{Li} to conclude the concavity of $d_F$ (Lemma \ref{7f}).

(7.4.2) Let $F$ be defined in the above. If $\uparrow_x^{\partial \Omega_c}$ is independent of $\nabla d_F(x)$, then up to a measure zero subset of $\hat X$ we show that $v=\nabla d_F(x)\in V_x$ (Lemma \ref{7g}).


(7.4.3) If $\uparrow_x^{\partial \Omega_c}=\pm \nabla d_F(x)$, then up to a measure zero subset of $\hat X$ $d_F$ achieves a local maximum, $m_0$, at $x$ in $\Omega_{b_{\bar p}(x)}$ (Lemma \ref{7h}). Let $Z(x)$ denote the component of $d_F^{-1}(m_0)\cap \partial \Omega_{b_{\bar p}(x)}$ at $x$.
We show that up to a measure zero subset (see (7.11.2)) there is a local isometric embedding $\psi: W_x\times I_\delta\to Z(x)$, $\psi(W_x\times \{0\})=\text{id}_{W_x}$, and we choose $v=\gamma^+(0)\in V_x$ (see (7.11.1)), where $\gamma(t)=\psi(x,t)\subset \partial \Omega_{b_{\bar p}(x)}$, $t\in [0,\delta]$. In particular, $u=\uparrow_x^{\partial \Omega_c}$ and $v$ are
independent.

(7.4.4) We show that $D\text{proj}_{\bar p}$ preserves both norm and angle for $(u,v)$ chosen in
the above (Lemmas \ref{7l} and \ref{7m}).

In the rest of the paper, we will supply proofs for (7.4.1)-(7.4.4).
The following lemmas are all under the assumptions of Key Lemma \ref{7c}.

\stepcounter{theorem}
\begin{lemma}\label{7e} Let the assumptions be as in Key Lemma \ref{7c}. Then
(7.4.1) holds.
\end{lemma}

\begin{proof} (7.4.1) should be known fact, and for a completeness we give a brief explanation.

Consider the distance function to a fixed closed subset $C$, $d_C$. For any $x\in X$, let
$z_x\in C$ denote a projection of $x$ i.e., $d_C(x)=|xz_x|$, and let $\gamma(t)=[z_xx]$ be a minimal geodesic from $z_x$ to $x$. Observe that for all $t$ such that $x\ne \gamma(t)$, the projection of $\gamma(t)$ on $C$ is also $z_x$ such that the geodesic $[z_x\gamma(t)]$ is unique. Then the set $A(d_C)$, consisting of $x'\in X$, $x'$ cannot be realized as an interior point of $\gamma(t)$ defined in the above, has a measure zero  (Proposition 2.2 in \cite{Wo1}).

Observe that the above also applies to a Busemann function $b_{\bar p}$ that defines $(X,S,\phi)$; let $c_i\to -\infty$, and then $b_{\bar p}|_{\Omega_{c_i}}=d_{\partial \Omega_{c_i}}-|\bar p\partial \Omega_{c_i}|: \Omega_{c_i}\to \mathbb R$ (\cite{CDM}). For each $i$, $A_i(b_{\bar p})=A_i(d_{\partial \Omega_{c_i}})$ has measure zero, thus $A(b_{\bar p})=\bigcup_iA_i(d_{\partial \Omega_{c_i}})$ has measure zero.
\end{proof}

\begin{lemma}\label {7f}
Let the assumptions be as in Key Lemma \ref{7c}. Then
$d_F$ is concave in $X\setminus F$,
and $\nabla d_F$ has norm one almost everywhere in $\hat X$.
\end{lemma}

\begin{proof} As showed in \cite{Li}, $d_F$ is concave if for all $\bar x\in S$, $\Sigma_{\bar x}X$ has a join structure. This condition is satisfied in our circumstances: because $(X,S,\phi)$ is global integrable and $\dim(S)=\dim(X)-2$, as explained in the beginning of (A1) assuming Theorem \ref{2a}, $(C_xX, C_{\bar x}S,D_x\phi)$ is global weakly integrable, and by Lemma  \ref{7b} one sees that $(C_xX, C_{\bar x}S,D_x\phi)$ is global integrable. Because every $v\in \Sigma_{\bar x}X$ tangents to a ray in $C_{\bar x}X$, $\Sigma_{\bar x}X=\Sigma_{\bar x}S* V_{\bar x}$. Because $d_F$ is semi-concave, as in the proof of Lemma \ref{7e} for $b_{\bar p}$, it is clear that $\nabla d_F$ has norm one almost everywhere in $\hat X$.
\end{proof}

\begin{lemma}\label {7g} Let the assumptions be as in Key Lemma \ref{7c}. Then
(7.4.2) holds.
\end{lemma}

\begin{proof}
By Lemma \ref{7f}, $\nabla d_F$ is well-defined, which has norm one up
to a measure zero subset of $X$. For any $x\notin F$, we show that $W_x\subset d_F^{-1}(d_F(x))$, thus $\nabla d_F(x)\in V_x$. Because $W_x$ is compact and $W_x\cap F=\emptyset$,
we may assume $d_F|_{W_x}$ achieves a minimum at $q\in W_x$, $q\notin F$. For any $z\in W_x$,
let $\gamma(t)\subset W_x$ be a normal minimal geodesic with $\gamma(0)=q, \gamma(l)=z$. By the first variation formula, $|\Uparrow_q^F\uparrow_q^z|\ge \frac\pi2$, thus $|\Uparrow_q^F\uparrow_q^z|=\frac\pi2$ (Lemma \ref {ya}) i.e., $(\text{d}_F\circ \gamma) ^+(0)=0$. Since $(\text{d}_F\circ \gamma )''\leq 0 $, we have that $d_F\circ \gamma (l)\leq d_F\circ \gamma (0)$, i.e., the desired result.
\end{proof}

\begin {lemma} {\rm (Partial verification of (7.4.3))} \label{7h}
Let the assumptions be as in Key Lemma \ref{7c}.
 If $\nabla_x^{\partial \Omega_c}=\pm \nabla d_F(x)$, then

\noindent {\rm (7.8.1)} $\uparrow_x^{\partial\Omega_c}=\nabla d_F(x)$,

\noindent {\rm (7.8.2)} $d_F|_{\Omega_{b_{\bar p}(x)}}$ achieves a local maximal value at $x$.
\end{lemma}

Before presenting our proof of Lemma \ref{7h}, we give the following example for a partial motivation.

\begin{example}\label{7i} Let $R=\{(s,t), -1\le s\le 1, -2\le t\le 2\}$ be a rectangle, and let
$X=R\cup (\partial R\times \mathbb R_+)\in \text{Alex}^2(0)$ with soul $S=(0,0)\times \{0\}$, and
$F=([-1,1]\times \{0\})\cup (\{(-1,0),(1,0)\}\times \mathbb R_+)$. For $x(s,t)\in R$, $|s|<10^{-1},
t\in [\frac {11}{10},\frac{15}{10}]$, $\uparrow_{x(s,t)}^{\partial \Omega_c}=\nabla d_F(x(s,t))$; thus such points have a positive measure. Observe that in a neighborhood of $x(0,\frac 32)$, $\partial \Omega_{b_{\bar p}(x)}$ coincides with $d_F^{-1}(d_F(x))$. In another words,
$d_F|_{\Omega_{b_{\bar p}(x)}}$ achieves a maximum at $x$.
\end{example}

\begin{proof} [Proof of Lemma \ref{7h}]

(7.8.1) It suffices to show that
$|\uparrow_x^{\partial \Omega_{c}}\uparrow_x^F|\geq \frac\pi2$ (equivalently $|\uparrow_x^{\partial \Omega_{c}}\nabla d_F(x)|\le  \frac\pi2$).

Let $z\in F$ be a point such that $|xz|=|xF|$.  
By the first variational formula and the concavity of $b_{\bar p}$, we have
$$-\cos( |\uparrow_x^{\partial
\Omega_{c}} \uparrow_x^z|)=D_xb_{\bar p}(\uparrow_x^z)\geq
\frac{b_{\bar p}(z)-b_{\bar p}(x)}{|xz|},$$
thus to conclude that $|\uparrow_x^{\partial \Omega_{c}}\uparrow_x^F|\geq \frac\pi2$, we need to show that  $b_{\bar p}(z)\geq b_{\bar p}(x)$, which is obvious if $z\in S$. If $z\not\in S$, then $b_{\bar p}(z)\geq b_{\bar p}(x)$ follows from the following:
$$\frac{b_{\bar p}(x)-b_{\bar p}(z)}{|xz|}\leq D_zb_{\bar p}(\uparrow_z^x)=-\cos (| \uparrow_z^{\partial\Omega_{c}}\uparrow_z^x|)= 0,$$
where the last equality is from the first variational formula and $|\uparrow_z^{\partial\Omega_{c}}\uparrow_z^x|=\frac \pi2$ (Lemma \ref{ya}).

(7.8.2) For any $z\in\Omega_{b_{\bar p}(x)}$, by the first variation formula, we have $|\uparrow_x^{\partial \Omega_{c}}\uparrow_x^z|\geq\frac\pi2$.
Since by (7.8.1) $\uparrow_x^{\partial
\Omega_{c}}=-\uparrow _x^F$,  $|\uparrow_x^F\uparrow_x^z|\leq \frac\pi2$,
which implies $D_xd_F(\uparrow_x^z)\leq 0$.
Set $r=\frac12 |xF|$.
We have that for any $z\in B(x,r)\cap\Omega_{\bar c}$,
$[yz]\cap F=\emptyset$.
Then by the concavity of $d_F$, $$0\geq D_x d_F(\uparrow_x^z)\geq \frac{d_F(z)-d_F(x)}{|xz|}.$$
Hence $d_F(x)\geq d_F(z)$.
\end{proof}

\begin{lemma}\label{7j} Let the assumptions be as in Key Lemma \ref{7c}.
If $\nabla_x^{\partial \Omega_c}=\pm \nabla d_F(x)$, let $Z(x)$ be defined in (7.4.3). Then

\noindent {\rm (7.10.1)} $W_x\subseteq Z(x)\,(\subset \partial \Omega_c)$ is closed and locally convex, thus $n-2\le \dim(Z(x))\le n-1$.

\noindent{\rm (7.10.2)} If $\dim(Z(x))=n-1$, then $\partial Z(x)\neq \emptyset$.
\end{lemma}

\begin{proof}  (7.10.1) Let $\hat Z(x)$ denote the component at $x$ of $d_F^{-1}(d_F(x))\cap \Omega_{b_{\bar p}(x)}$. Since $\uparrow_x^{\partial \Omega_c}=\nabla d_F(x)$, then by (7.8.2) $x$ is a local maximum of $d_F$ on $\Omega_{b_{\bar p}(x)}$.
It is not hard to see that for any $y\in \hat Z(x)$, $y$ is a local maximal point of $d_F$. Thus $\hat Z(x)$ is locally convex. It $\hat Z(x)\not =W_x$,
we will show that $\hat Z(x)=Z(x)$ by showing that
$\hat Z(x)\subset \partial \Omega_{b_{\bar p}(x) }$, therefore to conclude that $Z(x)$ is also locally convex. Indeed, 
first since for any $v\in \hat \Sigma_yZ(x)$, $|\uparrow_x^Fv| = \frac \pi2$ and $\uparrow_x^{\partial \Omega_c}=-\uparrow_x^F$, we have that $\Sigma_x\hat Z(x)\subset \partial \Omega_{b_{\bar p}}(x)$.  Thus there is a neighborhood of $x$, $U$, such that $U\cap \hat Z(x)\subset \partial \Omega_{b_{\bar p}}(x)$. Then by open close argument it suffices to show that for any $y\in \bar U\cap  \overset {\circ}{\hat Z(x)}$, $\Sigma_y\hat Z(x) \subset \Sigma_y\partial \Omega_{b_{\bar p}(y)}$.
The desired property follows from $|\Uparrow_y^F\Sigma_y\hat Z(x)|\ge \frac \pi2$, $|\Uparrow_y^{\partial \Omega_c}\Sigma_y\hat Z(x)|\ge \frac \pi2$, $\dim(\Sigma_y\hat Z(x))=n-2$ and $\partial \Sigma_y\hat Z(x)=\emptyset$, since it is not hard to get that in this case $\Sigma_yX=S(\Sigma_y\hat Z(x))$.

(7.10.2) Arguing by contradiction, assume that $\partial Z(x)= \emptyset$. Since $\dim Z(x)=n-1$, by the local convexity of $Z(x)$, we have that $Z(x)=\partial\Omega_{b_{\bar p}(x)}$, which is impossible, since $\partial\Omega_{b_{\bar p}(x)}$ containing points of $F$.
\end{proof}

\begin{lemma}\label{7k} Let the assumptions be as in Lemma \ref{7j}. Then (7.4.3) holds, because

\noindent {\rm (7.11.1)} Assume that $\dim(Z(x))=n-1$. Then
there is a local isometric embedding $\psi: W_x\times I_\delta\to Z(x)\subset \partial \Omega_{b_{\bar p}(x)}$, $\psi(W_x\times \{0\})=\text{id}_{W_x}$. Then $v=\gamma^+(0)\in V_x$ is independent of $u=\uparrow_x^{\partial \Omega_c}$, where $\gamma(t)=\psi(x\times t)$, $t\in I_\delta$.

\noindent {\rm (7.11.2)} The union of subsets, $Z(x)$ with $\dim(Z(x))=n-2$, has measure zero.
\end{lemma}

\begin{proof} (7.11.1) Because $Z(x)\subset\partial \Omega_{b_{\bar p}(x)}$, that $\dim(Z(x))=n-1$ implies that $Z(x)$ is open in $\partial \Omega_{b_{\bar p}(x)}$. For $y\in Z(x)\backslash W_x$, $d_F(y)=d_F(x)=m_0$ implies that $W_y\subset Z(x)$ ((7.10.1)). Without loss of generality, we may assume that $y$ is in a
local convex neighborhood of $x$ and $\phi(y)=\phi(x)$. Thus by (7.10.1) $[xy]\subset Z(x)$. Therefore the local inclusion map, $W_x\times [xy]\to Z(x)$, is a locally isometric embedding.

(7.11.2) Because $\dim(Z(x))=n-2$, $W_x=Z(x)$, which is determined by $\alpha=b_{\bar p}(x)$. By the coarea formula,
$$\mu_n\left(\bigcup_{W(x)=Z(x)}Z(x)\right)=\int_{-\infty}^0\left(\int _{W_x=Z(x)} \mu_{n-1}(Z(x))\right)d\alpha=0,$$
where $\mu_m$ denote the $m$-dimensional Hausdorff measure.
\end{proof}

\begin{lemma}\label{7l} Let the assumptions be as in Key Lemma \ref{7c}.
Assume that $(u,v)=(\uparrow_x^{\partial \Omega_c},\nabla d_F(x))$ are independent. Then
$D\text{proj}_{\bar q}: V_x\to V_{\text{proj}_{\bar q}(x)}$ preserves both norm and
angle for $(u,v)$.
\end{lemma}

\begin{proof} First we need some preparations.

Let $\Phi_t(x)$ denote the gradient flow $d_F$. Then the following three properties hold.

\noindent {\rm (7.12.1)} For $x\in \hat X\setminus F$, let $z_x\in F$ such that $|xz_x|=|xF|$. For $t\ge 0$, $\Phi_t: W_{z_x}\to W_{\Phi_t(z_x)}$ is an isometry.

\noindent {\rm (7.12.2)} $\phi\circ \Phi_t(z_x)=\phi(x)$.

\noindent {\rm (7.12.3)} Let $\gamma(t)=[z_xx]$ denote a minimal geodesic from $z_x$ to $x$. For any regular point $\bar q\in S$, let $y=\text{proj}_{\bar q}(x)\in \phi^{-1}(\bar q)$, and let $\gamma_{\bar q}(t)=\text{proj}_{\bar q}(\gamma(t))\in \phi^{-1}(\bar q)$. Then $\gamma_{\bar q}$ is a minimal geodesic.

Let's justify the above properties.

(7.12.1) Note that $\Phi_t$ is well-defined if $F$ locally divides $X$ into two components (a little more is required for point in $S\subset F$, \cite{Li}), which follows if $\phi^{-1}(\bar p)$ is a topological surface. Because $X$ is homeomorphic to $S\times \phi^{-1}(\bar p)$ and $\bar p$ is regular in $S$, for $x\in \phi^{-1}(\bar p)$, $C_xX=\mathbb R^{n-2}\times K$, $K$ is a metric cone over $V_x=S^1$, thus $x$ is topologically nice in $X$. This enables us to apply Lemma \ref{3b} and conclude that $x$ is also a topological manifold point in $\phi^{-1}(\bar p)$ (comparing Corollary \ref{3e}).

(7.12.2) We first assume that $\bar z_x$ is a regular point of $S$. Then $\text{proj}_{\bar z_x}: X\to \phi^{-1}(\bar z_x)$ is
a locally Lipschitz map. Let $g(t)=|\gamma(t)\gamma_{\bar z_x}(t)|$. It remains to show that $g(t)\equiv 0$ (thus $\phi(x)=\phi(z_x)$). Observe that for each
$t>0$, $\nabla d_F(\gamma(t))$ is orthogonal to $\Sigma_{\gamma(t)}W_{\gamma(t)}$, thus by the first variation formula, we have that almost everywhere
$$g^+(t)=|\gamma(t)\gamma_{\bar z_x}(t)|^+=-|\gamma^+(t)|\cos|\Uparrow_{\gamma(t)}^{\gamma_{\bar z_x}(t)}  \gamma^+(t)| -|\gamma^+_{\bar z_x}(t)|\cos |\Uparrow_{\gamma_{\bar z_x}(t)}^{ \gamma(t)}\gamma^+_{\bar z_x}(t)|=0.$$
Note that we need that $\gamma_{\bar z_x}(t)$ is a Lipschitz curve, which is clear if $t>0$ is small,
so one may start with $x(t)$ (replacing $z_x$), and completes the proof.

If $\bar z_x$ is not a regular point, then take a sequence of points, $z_{x_k}\to z_x$, such that $\bar z_{x_k}$ are regular point of $S$. Then $\Psi_t(z_{x_k})\to \Psi_t(z_x)$, $\bar z_{x_k}=\phi(\Psi_t(z_{x_k}))\to \phi(\Phi_t(z_x))$, and thus $\phi(\Psi_t(z_x))=\bar x$.

(7.12.3) By (7.12.1), we have that $ \gamma_{\bar q}(t)=\Phi_t(\gamma_{\bar q}(0))$. Since $\Phi_t$ is 1-Lipschitz, the length of $\gamma_{\bar q}(t) \le |xz_x|$ which is $|yF|$ by the proof of Lemma \ref{7g}. Consequently, the length of $\gamma_{\bar q}(t)=|yF|$, thus the desired
result follows.

Now we continue the proof of Lemma \ref{7l}. Up to a measure zero set, we can assume that $[xF]$ can be extend beyond $x$, i.e. $x$ is an interior point of $[x'z_{x'}]$ for some $x'\in X$.

Denote $\text{proj}_{\bar q}(x)$ by $y$. It is clear that
$$D\text{proj}_{\bar q}(\uparrow_x^{\partial \Omega_c})=\uparrow_y^{\partial \Omega_c}.$$
By (7.12.3),  $$\quad D\text{proj}_{\bar q} (\nabla d_F(x))=\nabla d_F(y),\quad |D\text{proj}_{\bar q} (\nabla d_F(x))|=|\nabla d_F(y)|.$$

Since $d_F(\exp t\uparrow_x^{\partial \Omega_c})=d_F(\exp t\uparrow_y^{\partial \Omega_c})$ (the proof of Lemma \ref{7g}), to which the first variation formula in $t$ yields
$$\cos |\uparrow_x^{\partial \Omega_c}\nabla d_F(x)|=\cos |\uparrow_y^{\partial \Omega_c}\nabla d_F(y)|,$$
i.e., $|\uparrow_x^{\partial \Omega_c}\nabla d_F(x)|=|\uparrow_y^{\partial \Omega_c}\nabla d_F(y)|$.
\end{proof}

\begin{lemma} \label{7m} Let the assumptions be as in Key Lemma \ref{7c}.
Assume that $(u,v)=(\uparrow_x^{\partial \Omega_c},\nabla d_F(x))$ are dependent, and let
$(u,v)=(\uparrow_x^{\partial \Omega_c},\gamma^+(0))$ be as in Lemma \ref{7k}. Then
$D\text{proj}_{\bar q}: V_x\to V_{\text{proj}_{\bar q}(x)}$ preserves both norm and
angle of $(u,v)$.
\end{lemma}

Note that Key Lemma \ref{7c} follows from Lemmas \ref{7l} and \ref{7m}.

\begin{proof}[Proof of Lemma 7.13]

As in the proof of Lemma \ref{7l}, we can assume that $[xF]$ can  extend beyond $x$, i.e. $x$ is an interior point of $[x'z_{x'}]$ for some $x'\in X$.
By (7.11.2), we assume that $\dim (Z(x))=n-1$, and by (7.11.1), there is an isometry: $\psi: W_x\times I_\delta\to Z(x)$ such that $\psi (W_x\times\{0\})=\text{id}_{W_x}$.
Let $\gamma(t)=\psi(x \times t)$, $t\in I_\delta$, a normal geodesic in $Z(x)$ at
$x=\gamma(0)$, $\psi(W_x\times t)=W_{\gamma(t)}$, in particular $\gamma^{+}(t)\bot W_{\gamma(t)}$.

Similar as the proof of Lemma \ref{7l}, we have
that $\gamma\subset \phi^{-1}(\phi(x))$. Then $D\text{proj}_{\bar q}(\gamma^{+}(0))=\gamma_{\bar q}^+(0)$, where $\gamma_{\bar q}(t)=\text{proj}_{\bar q}(\gamma(t))$ and $\gamma_{\bar q}$ is also a geodesic.
It is not hard to see that
$|\gamma^+(0)\uparrow_x^F|=\frac{\pi}2=|\gamma_{\bar q}^+(0)\uparrow_y^F|$.
Thus the proof is finished.
\end{proof}

\section {Appendix}

\vskip4mm

\begin{lemma} \label{8a} $V_x=\{v\in \Sigma_xX,\, |vH_x|=\frac \pi2\}$. If $\partial H=\emptyset$, then $V_x$ and $H_x$ are $\frac \pi2$-apart.
\end{lemma}

\begin{proof} By definition (\cite{KL}), for $w\in V_x$, $|wH_x|\ge \frac \pi2$. Let $u_i\in \Sigma_x\setminus (H_x\cup V_x)$, $u_i\to w$. Let $w_i$ be the projection of $u_i$ in $V_x$, and $h_i$ the projection in $H_x$. Then $|w_ih_i|=\frac \pi2$ (Proposition 5.4 in \cite{KL}) and $[w_ih_i]$ converges to a minimal geodesic of length $\frac \pi2$ from $w$ to a point in $H_x$ i.e., $|wH_x|\le \frac \pi2$.

Assume that $\partial H_x=\emptyset$. By the standard argument in the proof of Lemma \ref{ya} one sees that $V_x$ and $H_x$ are $\frac \pi2$-apart.
\end{proof}


\begin{theorem} {\rm (\cite{Wi})} \label {8b}  Let $M$ be an open manifold of nonnegative sectional curvature with a soul $S$. Then the canonical bundle $(M,S,\phi)$ is fiber bundle equivalent to $(T^\perp S, S, \text{proj})$; precisely, there is a diffeomorphism, $\psi: T^\perp S\to M$, such that $\text{proj}=\phi\circ \psi$.
\end{theorem}

A basic tool in the proof of Theorem \ref{8b} is the so-called dual foliation technique. Here we will present a direct and simple proof (see comments before Theorem \ref{rig}).

\begin{proof} [Proof of Theorem \ref{8b}]

Recall from \cite{CG} that the bundle map in Theorem \ref{sol-cg}, $f=\text{proj}\circ \psi^{-1}: T^\perp S\to S$, where $\psi: T^\perp S\to M$ is a diffeomorphism obtained
by constructing a vector field on $M$, $X$, which can be viewed as a smooth approximation to
the (forward) gradient field of $d_S$, $\nabla^+ d_S$ (\cite{Pet}), and $X=\nabla d_S$ in a neighborhood of $S$. The construction solely relies on the property that $d_S$ has no critical point on $M\backslash S$ (\cite{GS}), equivalently $\nabla^+d_S$ is non-where zero.

In view of the above, it suffices to construct $X$ with an additional property: at each point $X$ tangents to a $\phi$-fiber.

First, fixing $\bar s_0\in S$, by Theorem \ref{sol} $F_0=\phi^{-1}(\bar s_0)$ is an embedded submanifold of $M$. Observe that restricting to $F_0$, the Busemann function, $b_{\bar p}$, achieves the maximum at $\phi(F_0)=\bar s_0$ (with either extrinsic or intrinsic distances), and the distance function $d_{\bar s_0}: F_0\to \Bbb R_+$, has no critical point, where $F_0$ is equipped with the intrinsic metric. By now we can construct a smooth non-where zero field on $F_0\backslash \{\bar s_0\}$, $X_0$ (as in \cite{CG}).

Let $2\rho$ denote the injectivity radius of $S$. Secondly, via parallel translation along horizontal lifting of radial segments in $B_\rho(\bar s_0)$, we extend $X_0$ on $F_0$ to a smooth vertical field on $\phi^{-1}(B_\rho(\bar s_0))\backslash \phi^{-1}(\bar s_0)$, still denoted by $X_0$. By (0.2.2) it is clear that $X_0$ is vertical and $\phi$-fibers in $\phi^{-1}(B_\rho(\bar s_0))$ are preserved by gradient flows of $X_0$.

Thirdly, given a $\rho$-net, $\{\bar s_i\}$ on $S$, and a partition of unity, $\{f_i\}$, associate
to the locally finite open cover $\{B_\rho(\bar s_i)\}$ for $S$, for each $i$ we obtain
smooth non-where zero vertical field $X_i$ on $\phi^{-1}(B_\rho(\bar s_i))$. We glue $\{X_i\}$ together on $M=\bigcup_i\phi^{-1}(B_\rho(\bar s_i))$ via $\{f_i\circ \phi\}$, to obtain a desired smooth non-where zero vertical field $X$ on $M\backslash S$, whose gradient flows preserve $\phi$-fibers
and $X$ is close to $\nabla d_S$. As mention at the beginning of the proof, one easily define a desired diffeomorphism, $\hat \psi: T^\perp S\to M$, such that $\phi=\text{proj}\circ \hat \psi^{-1}$.
\end{proof}

\begin{proof} [Alternative proof of Lemma \ref{1b}]

Recall that for any $q\in W$, if $\nabla d_x(q)\ne 0$, then $\xi=\frac{\nabla d_x(q)}{|\nabla d_x(q)|}\in \Sigma_qX$ is uniquely determined by $|\xi \Uparrow_q^x|=\max\{|\Uparrow_q^x u|, \, u\in \Sigma_qX\}$.

Arguing by contradiction, if $\xi\notin \Sigma_qW$, then let $\eta\in \Sigma_qW$ (the projection of $\xi$), such that $|\xi\eta|=\min \{|\xi v|,\, v\in \Sigma_qW\}$. We claim that there is $\alpha\in \Uparrow_q^x\cap \Sigma_qW$ such that
$|\Uparrow_q^x\eta|=|\alpha\eta|$. Assuming the claim, we derive that $|\alpha\xi|\le |\alpha\eta|$ (by Lemma \ref{ya} and Toponogov triangle comparison), thus
$$|\Uparrow_q^x\xi|\le |\alpha\xi|\le |\alpha\eta|=|\Uparrow_q^x\eta|\le |\Uparrow_q^x\xi|,$$
and therefore $\eta=\xi$, a contradiction.

We now verify the claim. If $\eta\in \Sigma_qW$ tangents to a minimal geodesic $\gamma\subset W$,
then passing to a subsequence, $[\gamma(t_i)x]$ in $W$ converge (as $t_i\to 0$) to a minimal geodesic $[qx]$ in $W$ with the direction $\alpha$. By the first variation formula (see Remark 4.5.8 in \cite{BBI}), $|\alpha\eta|=|\Uparrow_q^x \eta|$. Because any $\eta'\in \Sigma_qW$ is the limit of $\eta_k$ that
tangent to a minimal geodesic $\gamma_k\in W$, thus $|\alpha'\eta'|=|\Uparrow_q^x \eta'|$ for some $\alpha'\in \Sigma_qW$.
\end{proof}

\vskip10mm

\bibliographystyle{amsplain}

\end{document}